\documentclass{article}
\usepackage[utf8]{inputenc}
\usepackage{hyperref}
\usepackage{xcolor}
\usepackage{amsmath,amsthm,amssymb}
\usepackage{graphicx}
\usepackage{xypic}
\usepackage{upgreek}
\usepackage[normalem]{ulem} 
\usepackage{mathrsfs}
\usepackage{ytableau} 
\newcommand{\ignore}[1]{}

\newtheorem{theorem}{Theorem}[section]
\newtheorem{proposition}[theorem]{Proposition}
\newtheorem{lemma}[theorem]{Lemma}
\newtheorem{corollary}[theorem]{Corollary}
\newtheorem{conjecture}[theorem]{Conjecture}

\newcommand{\mulem}{{\bf{Lemma}}.\;}  

\theoremstyle{definition}
\newtheorem{defin}[theorem]{Definition}
\newtheorem{example}[theorem]{Example}

\newtheorem{remark}[theorem]{Remark}

\newtheorem{para}[theorem]{}

\newcommand{\beq}{\begin{equation}}
\newcommand{\eq}{\end{equation}} 

\newcommand{\C}{{\mathbb C}}
\newcommand{\Z}{{\mathbb Z}}
\newcommand{\R}{{\mathbb R}}

\newcommand{\mat}{\begin{pmatrix}}
\newcommand{\tam}{\end{pmatrix}}

\newcommand{\LL}{{\mathsf{L}}} 
\newcommand{\LLL}{{\mathsf{L}'}} 

\newcommand{\bra}[1]{\langle #1 |} 
\newcommand{\ket}[1]{| #1 \rangle}

\newcommand{\dfour}[4]{\smat #1 \\ &#2 \\ &&#3 \\ &&&#4 \stam}

\newcommand{\Stwo}[4]{\smat #1 \\ &&#2 \\ &#3 \\ &&&#4 \stam}

\newcommand{\aaa}{{\mathsf{a}}}
\newcommand{\fff}{{\mathsf{f}}}
\newcommand{\DR}{\mathcal{R}}

\newcommand{\one}{\square}
\newcommand{\two}{\square\!\square} \newcommand{\three}{\square\!\square\!\square} 
\newcommand{\four}{\square\!\square\!\square\!\square} 

\newcommand{\oneone}{\!\!\begin{array}{c} \square \vspace{-.081in}\\ 
\vspace{-.1in}\square \vspace{.08in} \end{array}\!\!}

\newcommand{\oneoneone}{\!\!\begin{array}{c} 
\square \vspace{-.081in}\\ 
\vspace{-.08in}
\square 
\\ 
\square 
\end{array}\!\!}

\newcommand{\oneonexone}{\!\!\begin{array}{c} 
\square \vspace{-.081in}\\ 
\vspace{-.08in}
\blacksquare 
\\ 
\square 
\end{array}\!\!}

\newcommand{\oneoneonex}{\!\!\begin{array}{c} 
\square \vspace{-.081in}\\ 
\vspace{-.08in}
\square 
\\ 
\blacksquare 
\end{array}\!\!}

\newcommand{\oneonexonex}{\!\!\begin{array}{c} 
\square \vspace{-.081in}\\ 
\vspace{-.08in}
\blacksquare 
\\ 
\blacksquare 
\end{array}\!\!}

\newcommand{\oneonex}{\!\!\begin{array}{c} \square \vspace{-.081in}\\ 
\vspace{-.1in}\blacksquare \vspace{.08in} \end{array}\!\!}
\newcommand{\twoone}{\!\!\begin{array}{l} \square\!\square  \vspace{-.081in}\\ 
\vspace{-.1in}\square \vspace{.08in} \end{array}\!\!}

\newcommand{\threeone}{\!\!\begin{array}{l} \square\!\square\!\square   \vspace{-.081in}\\ 
\vspace{-.1in}\square \vspace{.08in} \end{array}\!\!}

\newcommand{\twoonex}{\!\!\begin{array}{l} \square\!\square  \vspace{-.081in}\\ 
\vspace{-.1in}\blacksquare \vspace{.08in} \end{array}\!\!}

\newcommand{\onetwo}{\begin{array}{l} \square \vspace{-.081in}\\ 
\vspace{-.1in}\square\!\square \vspace{.08in} \end{array}\!\!}

\newcommand{\onetwox}{\begin{array}{l} \square \vspace{-.081in}\\ 
\vspace{-.1in}\blacksquare\!\blacksquare \vspace{.08in} \end{array}\!\!}

\newcommand{\oneN}[1]{\one\hspace{-.1876cm}
\raisebox{.045cm}{\mbox{\scriptsize{#1}}} \;}

\newcommand{\twoN}[2]{\two\hspace{-.5431cm} 
\raisebox{.051cm}{
\mbox{\scriptsize{#1 \!\!  #2}}} \hspace{.21cm}}

\newcommand{\twoNN}[2]{\two\hspace{-.5431cm} 
\raisebox{.051cm}{\scalebox{0.7}{
\mbox{\scriptsize{#1} \hspace*{-.15cm}  \scriptsize{#2}}}} \hspace{.21cm}}

\newcommand{\oneoneN}[2]{\oneone\hspace{-.272cm} 
\raisebox{.11cm}{\mbox{\scriptsize{#1 }}} \!\!\!\!
\raisebox{-.0751cm}{\mbox{\scriptsize{#2 }}} 
\hspace{.1cm}
}

\newcommand{\AR}{a_1}
\newcommand{\BR}{b_{12}}
\newcommand{\CR}{c_{12}}
\newcommand{\AAR}{a_2}
\newcommand{\ARx}{A_1}
\newcommand{\BRx}{B_{12}}
\newcommand{\CRx}{C_{12}}
\newcommand{\AARx}{A_2}
\newcommand{\ARxx}{\alpha_1}
\newcommand{\BRxx}{\beta_{12}}
\newcommand{\CRxx}{\gamma_{12}}
\newcommand{\AARxx}{\alpha_2}
\newcommand{\smat}{\left(\begin{smallmatrix}}
\newcommand{\stam}{\end{smallmatrix}\right)}
\newcommand{\xD}{D^{\!-\!}}

\newcommand{\Pas}{\mathrm{Pascal}} 
\newcommand{\SSSS}{\mathbb{S}} 
\newcommand{\SSSSC}{\mathbb{S}^{\C}} 
\newcommand{\sS}{\mathrm{S}}

\newcommand{\slosh}[1]{\smat 0&\frac{\mu_{#1}}{C_{#1}}\\ \mu_{#1} C_{#1}&0\stam }  

\newcommand{\slush}[1]{\smat 0&\frac{1}{c_{#1}}\\  c_{#1}&0\stam }  

\newcommand{\slesh}[1]{\smat 0&\frac{\mu_{#1}}{c_{#1} C_{#1}}\\ \mu_{#1} c_{#1} C_{#1}&0\stam }  

\newcommand{\sleesh}[2]{\smat 0&\frac{\mu_{#1}}{c_{#2} C_{#1}}\\ \mu_{#1} c_{#2} C_{#1}&0\stam }  

\newcommand{\footnotex}[1]{}  

\newcommand{\off}{/}

\newcommand{\Rec}{{\mathsf R}}     

\newcommand{\N}{{\mathbb N}}
\newcommand{\ul}[1]{\underline{#1}}
\newcommand{\Match}{{\mathsf{Match}}}

\newcommand{\PP}{{\mathsf P}}  
\newcommand{\SSS}{{\mathfrak S}}  
\newcommand{\TTT}{{\mathfrak T}}  

\newcommand{\Power}{{\mathcal P}}   

\newcommand{\alphaa}{\pmb{\alpha}} 

\newcommand{\partp}{\pi_{\mathsf{p}}}
\newcommand{\partq}{\pi_{\mathsf{q}}} 
\newcommand{\partr}{\pi_{\rho}} 
\newcommand{\partt}{\ul{\pi}} 
\newcommand{\parts}{\pi_{\mathsf{s}}} 
\newcommand{\partrr}{\pi_{\mathsf{r}}}

\newcommand{\functor}{\mathfrak{Func}}  
\newcommand{\soutx}[1]{\ignore{#1}}

\newlength{\ccm}
\newcommand{\frametitle}[1]{}

\newcommand{\frx}[2]{}

\newcommand{\chapter}{\section}

\newcounter{minidef}[section]

\newcommand{\mdef}{\refstepcounter{theorem} 
\medskip \noindent ({\bf \thetheorem}) }
\newcounter{minicapt}

\newcommand{\fii}{\mathsf{f}} 
\newcommand{\ai}{\mathsf{a}}  
\newcommand{\pfii}{{\fii}} 
\newcommand{\uppfii}{\pfii}
\newcommand{\pai}{{\ai}}
\newcommand{\mfii}{\underline{\fii}} 
\newcommand{\mai}{\underline{\ai}}
\newcommand{\uppai}{\pai}
\newcommand{\ooo}[1]{\textcolor{orange}{\textcircled{\tiny #1}}}

\newcommand{\ber}[5]{ 
\xymatrix{
       &  {a_{#1}}_{\ooo{2}} \ar[dr]|{#4} \\
{{a_1}}^{\ooo{1}} \ar[ur]|{#3} \ar[rr]|{#5} &&  {}_{\ooo{3}}{a_{#2}}
}}

\newcommand{\berr}[6]{ 
\xymatrix{
       &  {{#2}}_{\ooo{2}} \ar[dr]|{#5} \\
{{#1}}^{\ooo{1}} \ar[ur]|{#4} \ar[rr]|{#6} &&  {}_{\ooo{3}}{{#3}}
}}

\newcommand{\bermss}[6]{ 
\xymatrix{
       &  {a_{#2}}_{\ooo{2}} \ar[dr]|{#5}="b" \\
{{a_{#1}}}^{\ooo{1}} \ar[ur]|{#4}="a" \ar[rr]|{#6} &&  {}_{\ooo{3}}{a_{#3}}
\ar@{=}_{} "a";"b"
}}

\newcommand{\bermpp}[6]{ 
\xymatrix{
       &  {a_{#2}}_{\ooo{2}} \ar[dr]|{#5}="b" \\
{{a_{#1}}}^{\ooo{1}} \ar[ur]|{#4}="a" \ar[rr]|{#6}="c" &&  {}_{\ooo{3}}{a_{#3}}
\ar@{=}_{} "a";"b"
\ar@{=}_{} "a";"c"
\ar@{=}_{} "b";"c"
}}

\newcommand{\mbermpp}[9]{ 
       &  {a_{#2}}_{\ooo{2}} \ar[dr]|{#5}="b" \\
{{a_{#1}}}^{\ooo{1}} \ar[ur]|{#4}="a" \ar[rr]|{#6}="c" &&  {}_{\ooo{3}}{a_{#3}}
\ar@{=}|{#7} "a";"b"
\ar@{=}|{#8} "a";"c"
\ar@{=}|{#9} "b";"c"
}

\newcommand{\mbermupp}[9]{ 
       &  {{#2}}_{\ooo{2}} \ar[dr]|{#5}="b" \\
{{{#1}}}^{\ooo{1}} \ar[ur]|{#4}="a" \ar[rr]|{#6}="c" &&  {}_{\ooo{3}}{{#3}}
\ar@{=}|{#7} "a";"b"
\ar@{=}|{#8} "a";"c"
\ar@{=}|{#9} "b";"c"
}

\newcommand{\abcab}{\ar@{=}_{} "a";"b"}
\newcommand{\abcac}{\ar@{=}_{} "a";"c"}
\newcommand{\abcbc}{\ar@{=}_{} "b";"c"}

\newcommand{\bermxx}[7]{ 
\xymatrix{
       &  {a_{#2}}_{\ooo{2}} \ar[dr]|{#5}="b" \\
{{a_{#1}}}^{\ooo{1}} \ar[ur]|{#4}="a" \ar[rr]|{#6}="c" &&  {}_{\ooo{3}}{a_{#3}}
#7
}}

\newcommand{\bertx}[9]{  
\xymatrix@R=37pt{
& {}^{a_{#1}} \ooo{2}_{} \ar[rr]^{#5} \ar[drrr]^{#8}  && \ooo{3}^{a_{#2}} \ar[dr]^{#7} \\
{}_{a_1}\ooo{1} \ar[ur]^{#4} \ar[urrr]^{#6} \ar[rrrr]^{#9} &&&& \ooo{4}_{a_{#3}}
}}

\newcommand{\bert}[9]{ 
\xymatrix@R=37pt{
& {a_{#1}}_{\ooo{2}} \ar[rr]^{#5} \ar[drrr]^{#8}  && {}_{\ooo{3}}{a_{#2}} \ar[dr]^{#7} \\
{{a_1}}^{\ooo{1}} \ar[ur]^{#4} \ar[urrr]^{#6} \ar[rrrr]^{#9} &&&& {}^{\ooo{4}}{a_{#3}}
}}

\newcommand{\berto}[9]{ 
\xymatrix@C=10pt{
1 \ar[drr]^{#4}="a"
  \ar@/^1pc/@{=}[rrrrrr]
           &&                  &          &
                                           && 1
                                           \ar[dll]^{#5}_{\;\;\;}="b"
                                           \ar@/^2pc/@{=}[dddlll] \\
           && #1 \ar[rr]^{#6}="c" \ar[dr]^{#7} &          & #3   \\
           &&                  & #2 \ar[ur]^{#8} &   \\
           &&                  & 1 \ar[u]^{#9} \ar@/^2pc/@{=}[uuulll]
}}

\newcommand{\bertovold}[9]{ 
\xymatrix@C=10pt{
1 \ar[drr]^{#4}="a"
  \ar@/^1pc/@{=}[rrrrrr]
           &&                  &          &
                                           && 1
                                           \ar[dll]^{#5}_{\;\;\;}="b"
                                           \ar@/^2pc/@{=}[dddlll] \\
           && #1 \ar[rr]^{#6}="c" \ar[dr]_{#7} &          & #3   \\
           &&                  & #2 \ar[ur]^{#8} &   \\
           &&                  & 1 \ar[u]^{#9} \ar@/^2pc/@{=}[uuulll]
           \ar@{=} "a";"b"
            \ar@{=} "b";"c"
            \ar@{=} "a";"c"
}}

\newcommand{\bertov}[9]{ 
\xymatrix@C=10pt{
1 \ar[drr]|{#4}="a"
  \ar@/^1pc/@{=}[rrrrrr]
           &&                  &          &
                                           && 1
                                           \ar[dll]|{#5}="b"
                                           \ar@/^2pc/@{=}[dddlll] \\
           && #1 \ar[rr]|{#6}="c" \ar[dr]|{#7} &          & #3   \\
           &&                  & #2 \ar[ur]|{#8} &   \\
           &&                  & 1 \ar[u]|{#9} \ar@/^2pc/@{=}[uuulll]
           \ar@{=} "a";"b"
            \ar@{=} "b";"c"
            \ar@{=} "a";"c"
}}

\newcommand{\abcccc}{ 
        \ar@{=} "a";"b"
        \ar@{=} "b";"c"
        \ar@{=} "a";"c"    }

\newcommand{\cdeeee}{ 
        \ar@{=} "c";"d"
        \ar@{=} "d";"e"
        \ar@{=} "c";"e"    }

\newcommand{\adffff}{ 
        \ar@/_/@{=} "a";"d"
        \ar@/_/@{=} "d";"f"
        \ar@/_/@{=} "a";"f"    }

\newcommand{\beffff}{ 
        \ar@/^/@{=} "b";"e"
        \ar@/^/@{=} "e";"f"
        \ar@/^/@{=} "b";"f"    }

\newcommand{\bertoxx}[9]{ 
\xymatrix@C=10pt{
1 \ar[drr]|{#4}="a"
  \ar@/^1pc/@{=}[rrrrrr]
           &&                  &          &
                                           && 1
                                           \ar[dll]^{#5}_{\;\;\;}="b"
                                           \ar@/^2pc/@{=}[dddlll] \\
           && #1 \ar[rr]|{#6}="c" \ar[dr]_{#7} &          & #3   \\
           &&                  & #2 \ar[ur]_{#8} &   \\
           &&                  & 1 \ar[u]^{9} \ar@/^2pc/@{=}[uuulll]
#9
}}

\newcommand{\bertox}[9]{ 
\xymatrix@C=10pt{
1 \ar[drr]|{#4}="a"
  \ar@/^1pc/@{=}[rrrrrr]
           &&                  &          &
                                           && 1
                                           \ar[dll]|{#5}="b"
                                           \ar@/^2pc/@{=}[dddlll] \\
           && #1 \ar[rr]|{#6}="c" \ar[dr]|{#7}="d" &          & #3   \\
           &&                  & #2 \ar[ur]|{#8}="e" &   \\
           &&                  & 1 \ar[u]|{\x}="f" \ar@/^2pc/@{=}[uuulll]
#9
}}

\newcommand{\pyr}[9]{
\xymatrix@C=30pt@R=13pt{
1 \ar[rr]|{#1} \ar[dr]|{#2} \ar[dddr]|{#3}
                         &                  & #9 \\
                         & #7 \ar[ur]|{#4} \ar[dd]|{#5}    \\ \\
                         & #8 \ar[uuur]|{#6}
}}

\newcommand{\ccc}{\bullet}
\newcommand{\trin}[6]{ \xymatrix{& #3\; \ccc\;\; \ar@{-}[dr]^{#4}
\\
#1 \ccc \ar@{-}[rr]_{#6} \ar@{-}[ur]^{#2} && \ccc #5  }}

\newcommand{\Bcat}{\mathsf{B}}
\newcommand{\Mat}{\mathsf{Mat}} 
 \newcommand{\sgn}{\operatorname{sgn}}
\newcommand{\ppm}[1]{\textcolor{red}{#1}}
\newcommand{\ecr}[1]{\textcolor{orange}{#1}}
\newcommand{\ft}[1]{\textcolor{purple}{#1}}

\newcommand{\RR}{\mathsf{R}^{\circ}} 
\newcommand{\JJJ}{\mathbb{J}^{\pm}} 
\newcommand{\xx}{\ul{\mathsf{x}}} 
\newcommand{\Sym}{\Sigma} 
\newcommand{\PPmat}{\mathsf{P}\!} 
\newcommand{\IImat}{\mathbf{1}\!}  
\newcommand{\F}{\mathsf{F}}  
\newcommand{\FF}{\mathsf{F}}  
\newcommand{\MM}{\mathscr{M}}  

\title{Classification of charge-conserving loop braid representations}
\author{Paul Martin, 
Eric C. Rowell, and 
Fiona Torzewska}
\date{\today}

\begin{document}

\maketitle

\tableofcontents 

\medskip \hspace{.41in} 
\newpage 

\begin{abstract}
Here a loop braid representation is a monoidal functor 
$\FF$ 
from the loop braid category 
$\LL$ 
to a suitable target category,
and is $N$-charge-conserving if that target is the category 
$\Match^N$  
of charge-conserving matrices
(specifically $\Match^N$ is 
the {same} rank-$N$ charge-conserving monoidal subcategory of 
the monoidal category $\Mat$
used to classify braid representations in \cite{MR1X}) {with $\FF$ strict, and surjective on $\N$, the object monoid}.
We classify and construct all such representations. 
In particular we prove that representations fall into
varieties indexed by 
a set in bijection with the set of 
pairs of plane partitions of 
total degree $N$. 
\end{abstract}

\medskip 

\section{Introduction}

Viewing the braid group $B_n$ as a group of motions of $n$ points in the 2-disk  leads to vast generalisations when pondered in 3 spatial dimensions,
{including} 
motions of links in the 3-ball \cite{Dahm}.  The simplest of these is the loop braid group $LB_n$: motions of $n$ unlinked, oriented circles 
\cite{Damiani,Goldsmith,Dahm,TMM}.
The representation theory of $LB_n$ is 
(despite much intriguing progress - see for example 
\cite{Bardakov,Vershinin,Soulieetal,BMM19,DMM,KMRW})
largely unknown, 
and the aim of a systematic study of extending braid representations to $LB_n$ inspired  \cite{DMR} in which  a loop braid group version of the Hecke algebra was discovered.  This revealed a surprise: {there exists} a $4\times 4$ non-group-type (see \cite{KMRW}) Yang-Baxter operator $R$ that admits a lift $(R,S)$ yielding a local representation of $LB_n$.
Is this $R$ an isolated example, or does it fit into a larger family?  An appropriate context for answering this question is suggested by a salient feature of this $R$: it is \emph{charge conserving}\footnote{The idea for the term charge-conserving comes from the XXZ spin-chain setting 
- cf. e.g. \cite[Ch.8 {\it et seq}]{Baxter82} - 
{hence also `spin-chain representation',}
but the spin-chain context makes less sense for loop braid. 
} 
in 
{the sense of \cite{MR1X} (also}
described below). 

In this article we classify charge conserving loop braid representations.  A preliminary step 
is to classify charge conserving braid representations, which was carried out in \cite{MR1X}. 
Our results can be interpreted as a classification of monoidal functors from the Loop braid category $\LL$
to the category of charge-conserving matrices $\Match^N$
that are
{surjective on objects}, and strict.
{This $\LL$ is the diagonal category made up of loop braid groups
$LB_n$, 
exactly paralleling the relationship  between MacLane's braid category \cite{MaclaneBook} and the Artin braid groups \cite{Artin}. 

Just as the braid groups and the Yang-Baxter equation manifest as key components of several areas of mathematics and physics, so the loop braid groups are key to applications that require a higher-dimensional generalisation.  
{Their} study is thus partially 
motivated by various such applications. 
One is the aim of 
formulating a notion of higher quantum group (cf. e.g. \cite{Kassel,BMM19}).
Another is  
the aim of
determining statistics of loop-like excitations (see e.g. \cite{NLYphys}) in $3$D topological phases of matter (see e.g. \cite{LKW,BCW}), which in turn has applications to topological quantum computation, see e.g. \cite{Nayaketal,Rowell_Wang}.
And another example is 
construction of solutions to the tetrahedron equation \cite{KapranovVoevodsky,Elgueta}.

\label{pa:claim}
The result 
(Theorem~\ref{th:main})
may be summarised  as follows. 
\\
The set of all
varieties of charge-conserving loop braid
representations may be indexed by the set of 
`signed multisets' of   
compositions
where each composition has at most two parts. 
A {\em{signed multiset}} is 
an ordered pair of multisets. 
For example
\[ \lambda \; = \; 
\ytableausetup{smalltableaux} 
\left( \; 
\begin{ytableau} { }  \end{ytableau} \;\;\;
\begin{ytableau} {}  \end{ytableau} \;\;\;
\begin{ytableau} {}  \end{ytableau} \;\;\;
\begin{ytableau} {}&{} \\ {}  \end{ytableau} \;\;\;
\begin{ytableau} {} \\ {}&{}  \end{ytableau} \;\;\;
\begin{ytableau} {}&{} \\ {}&{}  \end{ytableau} \;\;\;
, \; \;
\begin{ytableau} {}&{}&{}  \end{ytableau} \;\;\;
\begin{ytableau} {}&{} \\ {}  \end{ytableau} \;\;\;
\begin{ytableau} {}&{} \\ {} \end{ytableau} \;\;\;
\begin{ytableau} {} \\ {}&{}&{}  \end{ytableau} \;\;\;
\right)
\]
(To connect with the corresponding 
index set for braid representations 
one should think of two-coloured
compositions with each part having a different colour,
so that the colouring is forced and 
hence need not be explicitly recorded.)
\\
We will show explicitly \\ 
(I) how to construct a 
variety of representations from each such index; 
and \\ 
(II) that every charge-conserving representation arises this way.

\medskip

As discussed in detail 
in \S\ref{ss:lab},  
category $\LL$ is monoidally generated by 
two kinds of exchange of pairs of loops,
a non-braiding exchange denoted $s$ and 
a braiding exchange denoted $\sigma$. 
Thus we can give a solution,
a monoidal functor $F$, by giving the pair $(F(s), F(\sigma))$.

As alluded to above, the  
classification of braid representations {in \cite{MR1X}}
actually
progressed serendipitously 
from the aim of a systematic study of extensions of braid representations to loop braids
(in Damiani et al \cite{DMR}).
So in the present paper the original aim is realised.
\\
(To unpack this background a bit:
rather than {\em extension}, this can be seen as 
`merging' braid and symmetric group representations --- which raises the question of how to bring their separate universes 
(spaces on which they act)
together. 
Each has its own up-to-isomorphism freedom. So one idea is to rigidify one or both of them when bringing them together. The idea of rigidification on the braid side set up some choices and a direction of travel which, so far, ends with charge-conservation... which then turned out to facilitate a complete classification in this setting!)
\\
Given the route that led to braid-representation classification,
it is natural to make loop braids one of the
next structures to be studied using the charge-conserving machinery (and its broad underlying philosophy of paying active attention to the target category as well as the source).

A rough `route map' 
for the present paper
is provided by the stages of the braid representation classification in \cite{MR1X}.
In particular then, one would start with a physical realisation of loop braids, lifting the `Lizzy category' from \cite{MR1X}. 
For the sake of brevity 
(and given that the strong parallel is stretched by the absence of a 4d laboratory)
we 
have the option to 
jump this 
and pass to the next stage: a presentation
- see \S\ref{ss:story}. 
But see \S\ref{ss:lab} (supported by \S\ref{ss:AppA}) for a workable heuristic.
The target category is recalled in \S\ref{ss:Match},
where the further properties of Match categories that we
shall need 
(cf. \cite[\S3]{MR1X})
are obtained. 
In \S\ref{ss:N2s} (and \S\ref{ss:N3s})
we prove the key Lemmas determining
the form of solutions in low rank.
In \S\ref{ss:prep} we introduce the combinatorial 
structures that we shall need,
adapting those developed in \cite[\S5]{MR1X} in light 
of \S\ref{ss:N2s}.
And in \S\ref{ss:MainTheorem} we prove the 
classification Theorem. 

\subsection*{Recipe in brief}

We now {\it outline}
our combinatorial parameterisation of 
{isomorphism classes}
{(under the group of symmetries as in \cite{MR1X})}
of functors $\FF:\LL\rightarrow \Match^N$.

 {As mentioned above,} 
a functor $\FF:\LL\rightarrow \Match^N$ is determined by a pair of charge-conserving $N^2\times N^2$ matrices, {we denote these here by a pair $(R,S)$}.
{As with all charge conserving matrices,} $R$ and $S$ may be encoded as a sequence $R\leftrightarrow (a_1,a_2,\ldots,a_N,A(1,2),\ldots,A(N-1,N))$ where $a_i\in\C^*$ and the $A(i,j)$ are $2\times 2$ matrices, and similarly $S\leftrightarrow (b_1,\ldots,b_N,\ldots,B(i,j),\ldots)$.  
{(See \S\ref{ss:MatchN} for details.)}

{Let $J_N^{\pm}$ denote the set of signed multisets of two-part composition diagrams of total degree $N$.
For example the following multisets are in $J_3^{\pm}$, with diagrams before the comma having a $+$, and diagrams after a $-$. (Full details are in \S\ref{ss:J}.)
\[
(\one^3,), (\onetwo^1,), (\two^1,\one^1), (,\one^1 \two^1)
\]
Within each sign, we observe the following convention for ordering diagrams: first in ascending total size, secondly, for diagrams of equal total size, in ascending order of the second part of each composition.
Given an element of $J_N^{\pm}$, we refer to compositions as nations, labelling nations as $n_1,n_2,\ldots$, and within a nation $n_t$ we refer to each part of the composition as a county, labelling the top part of the composition diagram $s_{t,1}$ and the second part $s_{t,2}$.
}

{Now to an element $\lambda\in J_N^{\pm}$ we construct a pair $(R,S)$ as follows. First, label the boxes in $\lambda$ with the $\{1,\ldots, N\}$ in order with the first numbers going left to right in the first county of $n_1$, then the second county and so on with $n_2,n_3,\ldots$. }
The following is an example with $N=11$:

\beq \label{eq:lamex-antenatal}
\lambda \leadsto (\twoN{1}{2} \; \twoN{3}{4} \;
\oneoneN{5}{6} , 
\;\oneN{7} \; \oneN{8} \; \oneN{9} \; \twoNN{10}{11})
\eq

 Now, for each nation $n_s$ we assign parameters $\alpha_t$ to county $s_{t,1}$ and $\beta_t$ to county $s_{t,2}$ (if it is non-empty) such that $\alpha_t+\beta_t\neq 0$.  Next, for each pair of distinct nations $n_s,n_t$ we assign two non-zero parameters $\mu_{s,t},C_{s,t}$.  

Firstly, if $i$ resides in county $s_{t,1}$ then $a_i=\alpha_t$  whereas if $i$ resides in county $s_{t,2}$ then $a_i=\beta_t$.  If $\sgn(n_t)=+$ then $b_i=1$ (resp. $b_i=-1$) if $i$ resides in $s_{t,1}$ (resp. in $s_{t,2}$).  The sign of $b_i$ is opposite to this in case $\sgn(n_t)=-$.

Consider each pair of individuals $i<j$.
\begin{enumerate}
    \item 
If $i\in n_s$ and $j \in n_t$ with $s\neq t$ then $A(i,j)=\begin{pmatrix} 0 & \mu_{s,t}/C_{s,t}\\ \mu_{s,t}C_{s,t} &0\end{pmatrix}$, and $B(i,j)=\begin{pmatrix} 0 & 1\\
1& 0
\end{pmatrix}$.

\item 
If $i$ and $j$ are in the same nation $n_t$ but different counties $s_{t,x}$ and $s_{t,y}$ 

(note $x<y$ by construction),
then $A(i,j)=\begin{pmatrix} \alpha_t+\beta_t & \alpha_t\\ -\beta_t & 0\end{pmatrix}$ and $B(i,j)=\sgn{(n_t)}\begin{pmatrix} 0 & 1\\
1& 0
\end{pmatrix}$.
\item If $i$ and $j$ are both in the first, 
(respectively second), county 
 in $n_t$  
then
$A(i,j)=\sgn{(n_t)}\begin{pmatrix} \alpha_t &0\\0&\alpha_t\end{pmatrix} $,
(respectively
$A(i,j)=\begin{pmatrix} \beta_t &0\\0&\beta_t\end{pmatrix} $)
and $B(i,j)=\begin{pmatrix} 1 & 0\\
0& 1
\end{pmatrix}$.
\end{enumerate}

Our results imply that 
\begin{enumerate}
    \item This construction of $(R,S)$ does provide a functor $\FF:\LL\rightarrow\Match^N$ and
    \item For any such functor, the corresponding pair $(R,S)$ may be transformed into an equivalent pair $(R^\prime,S^\prime)$ of the above form by means of two basic symmetries: simultaneous local basis permutations and/or simultaneous conjugation by a diagonal matrix $X$.
\end{enumerate}

In \cite{KMRW} we used the nomenclature, adapted from \cite{AS}, \emph{loop braided vector spaces} (LBVSs) for a general triple $(V,R,S)$ where $V$ is an $N$-dimensional vector space and $(R,S)$ defines a functor $\FF:\LL\rightarrow \Mat^N$.  Our main result can then be phrased as a classification of charge conserving LBVSs.

\medskip

\noindent {\bf Acknowledgements}.
We thank Emmanuel Wagner, Celeste Damiani and Alex Bullivant
for useful conversations.
PPM thanks EPSRC for support under Programme Grant 
EP/W007509/1.
PPM thanks Paula Martin 
and Joao Faria Martins
for useful 
conversations, and Leonid Bogachev for a 
reference on asymptotics of plane partitions. The research of ECR was partially supported USA NSF grants DMS-1664359, DMS-2000331 and DMS-2205962.
FT thanks support from the EPSRC under grant EP/S017216/1.

\section{Some basics of loop braids} \label{ss:lab}

Mac~Lane's monoidal braid category $\Bcat$  
\cite[Sec.XI.4]{MaclaneBook}
has object
monoid generated by a single object - a single strand of
hair, or a single point from which this hair is extruded.  
The category is then generated by an 
elementary braid in $\Bcat(2,2)$ and its inverse.
The monoidal category $\LL$ 
has object monoid generated by a single circle or loop;
and the category is generated by {\em two} `exchanges'
in $\LL(2,2)$. 
In \cite{MR1X} we used the `Lizzy category' to give 
a geometric framework for $\Bcat$. 
This section aims merely to visualise the two generators of
$\LL$ in an analogous way.%
\footnote{As opposed to the hybrid combinatorial visualisations of
\cite{BCW} borrowed from virtual braids, for example.}
(As such the section can  optionally be skipped. For representation
theory we can rely on the presentation given in 
\S\ref{ss:story}.) 

\medskip

For each $n \in \N$,
let $C_n$ be a configuration  of $n$ 
unlinked oriented 
circles  
in a box in $\R^3$.
{We will fix $C_n$ so the $i$-th loop is a circle of small radius in the 
$xy$-plane centred at $(i,0,0)$.
(Up to isomorphism it will not matter precisely
which configuration we take for $C_n$.) 
}
The loop braid group $LB_n$ is a `motion group' 
for $C_n$. 
See for example Dahm \cite{Dahm}, Goldsmith \cite{Goldsmith}, 
Lin \cite{Lin}, Fenn--Rimany-Rourke \cite{FRR},
Baez-Crans-Wise \cite{BCW}, Brendle-Hatcher \cite{BH}, Damiani \cite{Damiani} and references therein. 
Or see Appendix~\ref{ss:AppA}. 

In the spirit of 
the braid category 
$\Bcat$
the groups $LB_n$ form a natural diagonal subcategory of 
a 
motion groupoid
(informally speaking, 
$\mathrm{Mot}_{\R^3}$ as in \cite{TMM}, except we should stress
compact support --- 
in \cite{TMM} it is proved that a loop braid group is a motion group in a box $B^3$ with fixed boundary).
The monoidal structure is indicated by placing one row of circles following another, to make a longer row:
$C_n \sqcup C_m = C_{n+m}$.
\soutx{ 
\ft{We must be careful here, in \cite{TMM} we prove Damiani loop braid group is motion group for $B^3$ with fixed boundary.}
\ppm{[-I know this is only a comment so far, but it 
would be useful if it was stated more precisely.]}
\ppm{[I would also have the para containing it a bit lower in exposition.]}
}

\medskip 

The braid category $\Bcat$ is the diagonal groupoid whose groups of morphisms are the ordinary braid groups. 
In this context 
the braid groups have various relevant realisations
(for representation theory it is convenient to work with efficient presentations, but for intuition
and application geometric realisations are more useful). 
In \cite{MR1X} a realisation as hair-braiding is used. 
In this case   
one may consider a square 
(or other topological disk)
that is a fixed-height section through the hair, thus cutting the hair at $n$ points in the square. 
An initial configuration, $P_n$ say, places 
the $n$ points at regular intervals in the square.
In our loop-braid case the square 
section through a 3d braid-laboratory 
is replaced by a box
(it would be a section through a 4d loop-braid laboratory), 
and the
points by circles, as in Fig.\ref{fig:loopnnD3a}.

\begin{figure}
    \centering
  \includegraphics[width=8cm]{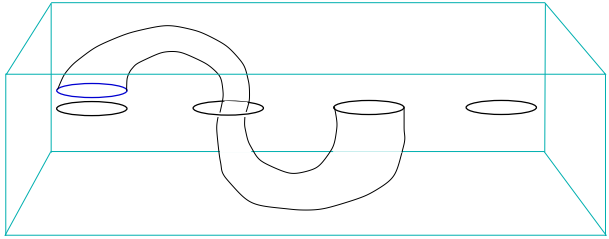}
    \caption{Schematic overlay visualisation of a loop motion.}
    \label{fig:my_label666}
\end{figure}
\begin{figure}
    \centering
  \includegraphics[width=8cm]{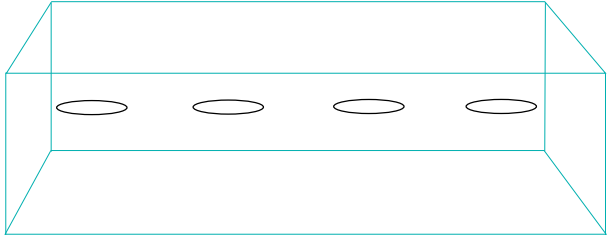}
    \caption{Initial configuration of loops in a box
    ($C_n$ with $n=4$).
    Also this is the overlay visualisation of a 
    static motion, representative of the identity in $LB_4$.
    \label{fig:loopnnD3a}}
\end{figure}

\begin{figure}
    \centering
    \includegraphics[width=8cm]{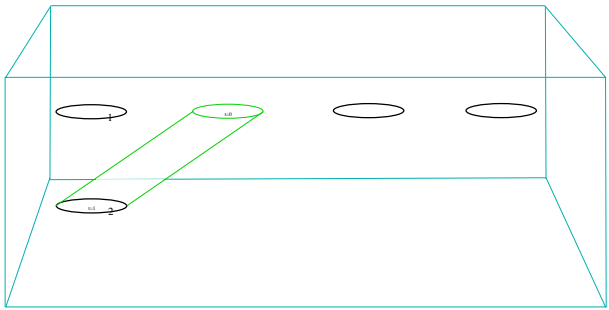}
    \caption{Overlay view of simple non-diagonal motion, moving loop-2 under loop-1. 
    }
    \label{fig:79-0}
\end{figure}

In the braid case one can either 
think of the braiding as taking place over time 
(particles move in the square and world-lines braid), or acting 
by physical continuity 
on the hair above the point of action - between this point and the fixed scalp as it were. 

In the former perception, we can visualise the braiding by showing all the locations visited by each 
point - still just drawn in the square.
An `overlay' visualisation. 
This has the drawback that the same location can be visited at different times
(indeed a simple static point is an extreme form of this), 
but the merit
is that the visualisation remains in 2d.
The drawback can be minimised by only drawing 
`simple'  
changes 
\footnote{Here we will leave the 
simple-change 
notion entirely informal and example-based. 
Note that the separation between two `particles' in some intermediate moment can be small, so  simpleness 
is not necessarily enough to render distinct paths
from $a$ to $b$ homotopic
in the relevant space.}
in a single instance \cite{MMT}.  
For this
it is convenient momentarily to break out of the diagonal groupoid 
with object set $\{ P_n \}_n$
and into the more general,
in the sense of considering partial braidings, 
passing to configurations different from $P_n$.
Collectively the braidings and partial braidings are 
`motions'.
(Depending on the realisation, a motion may be an object set trajectory, or a path in a space of homeomorphisms of the ambient space that restricts to this trajectory. Here we focus on describing the object set trajectory.)
The morphisms in the category  (the elements in the
groups)   are certain equivalence classes of motions
- see below.

\medskip 

In the loop braid case we thus have 
overlay visualisations in 3d.
We can further use artists-impression to indicate 
these 3d objects on the 2d page, finally giving representations 
like those in 
Fig.\ref{fig:79-0} and 
Fig.\ref{fig:79alt}. 
Here the colour-code is green for the initial and
intermediate points, and black for the final points.
Note in this visualisation that the `world-line' of a static point is simply a point.

Two 
`motions'
between the same initial and final configurations
(configurations $a$ and $b$, say)
are equivalent if one can be 
homotopically  
deformed into the other in the box.
This is a process that is not made 
easy to track by the overlay picture! 
(Cf. Fig.\ref{fig:79alt}.)
But this will not
be an issue for us here.
Note that the specific figure illustrates one analogue here of the 
triangle equivalence of polygonal knots.
By another such analogue, the image of a circle at some
intermediate point in a motion need not be a simple 
translate of the circle.
Examples of this follow shortly.
(However see  
(\ref{fact:lin}).)

\begin{figure}
    \centering
    \includegraphics[width=6.8cm]{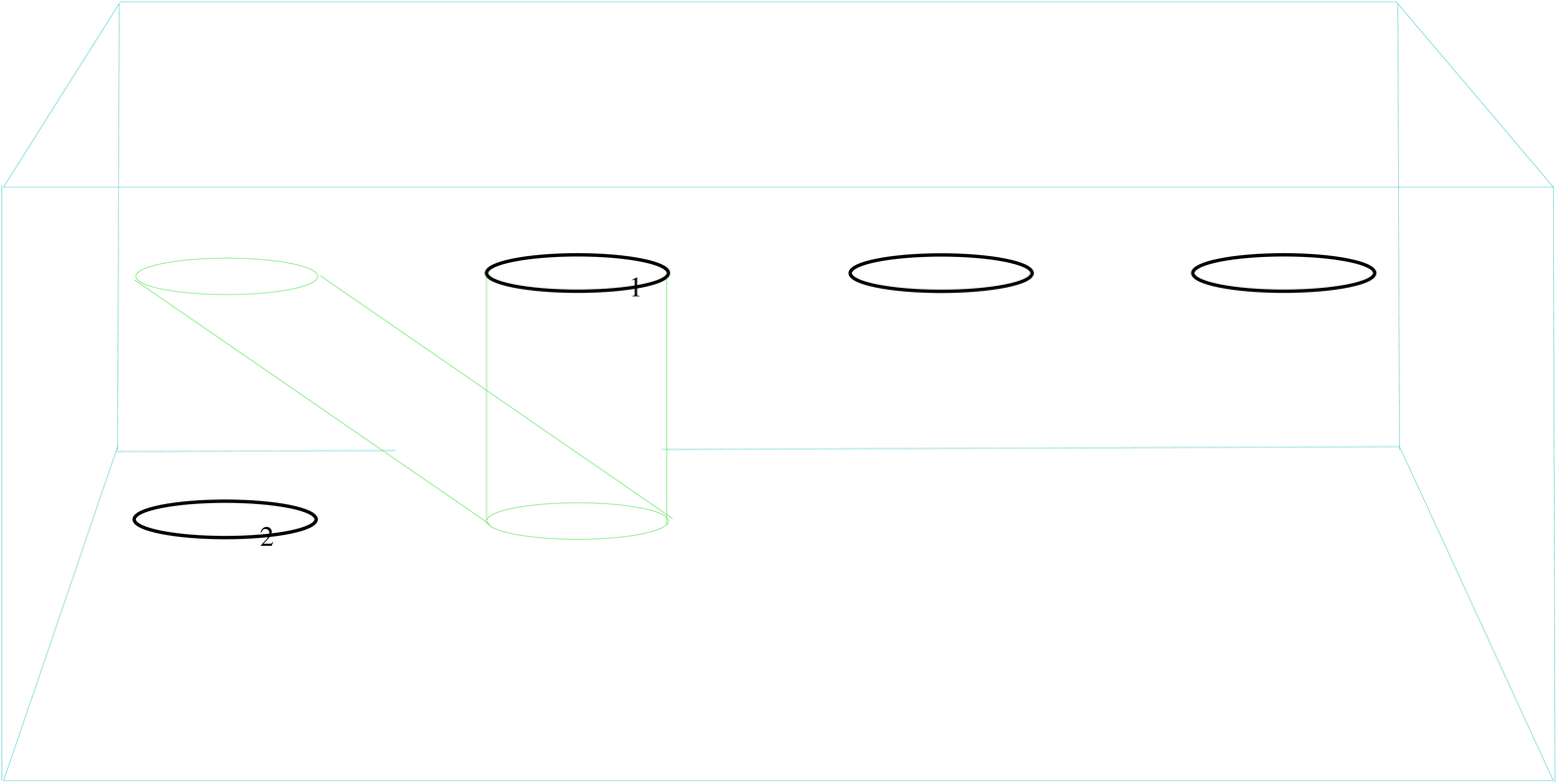}
    \hspace{.31in} 
    \includegraphics[width=6.8cm]{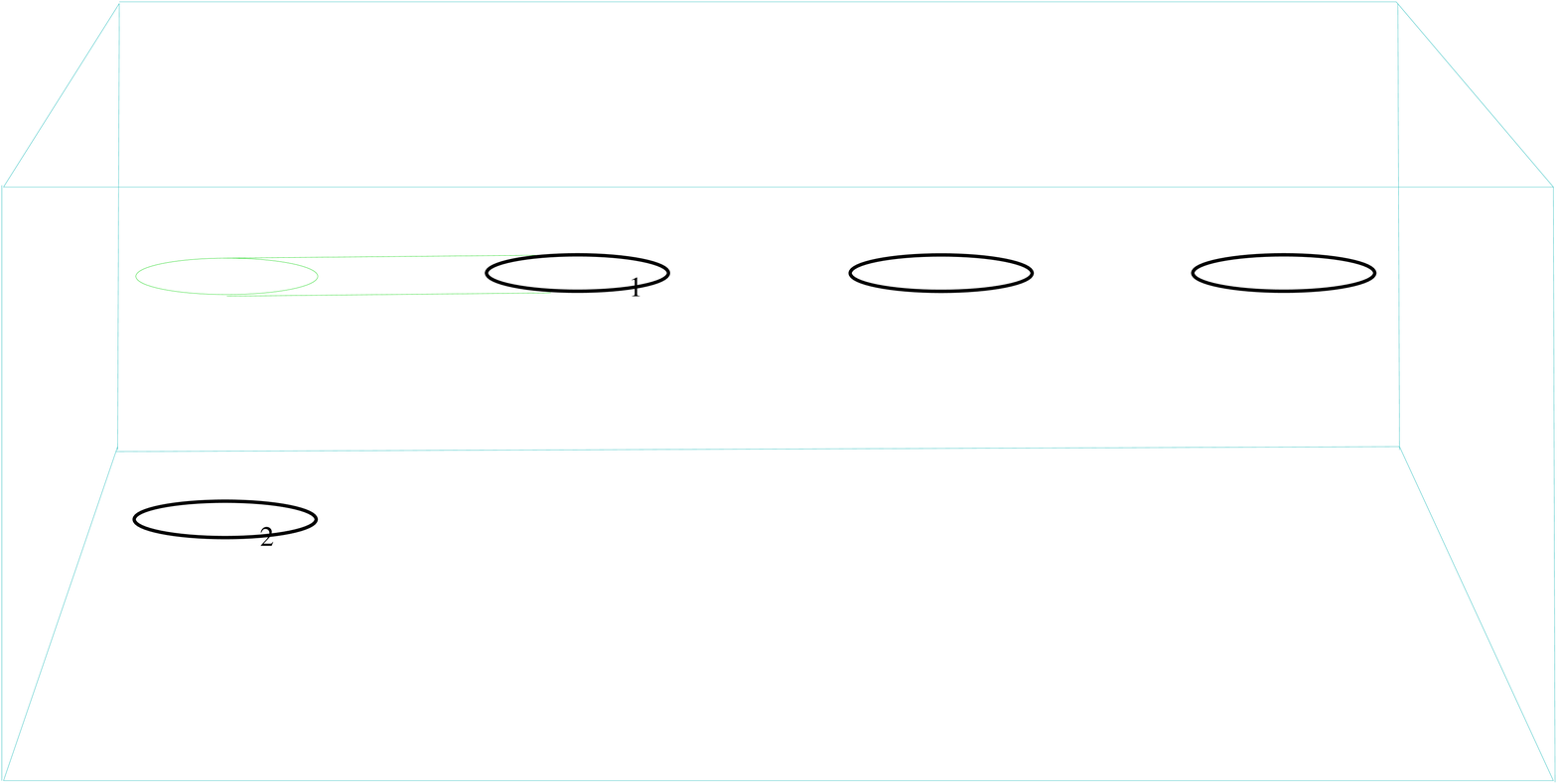}
    \caption{Equivalent motions, moving 
     loop-1 over, indirectly or directly. 
    }
    \label{fig:79alt}
\end{figure}

Consider also the 
the sequence in Fig.\ref{fig:79}. This sequence is 
a complete one, in the sense that we finally return, set-wise, to the starting configuration. 
Similarly Fig.\ref{fig:79xx}.
Note in general that two motions are composable if the tail of one is the source of the next.
A visualisation of the composition would amount to overlaying and replacing all now-intermediate black to green (but we will not need this).

\begin{figure}
    \centering
    \includegraphics[width=8cm]{xfig/loopnn2a12BOXgB.eps}
    \includegraphics[width=8cm]{xfig/loopnn2bc21BOXgBs.eps}
    \includegraphics[width=8cm]{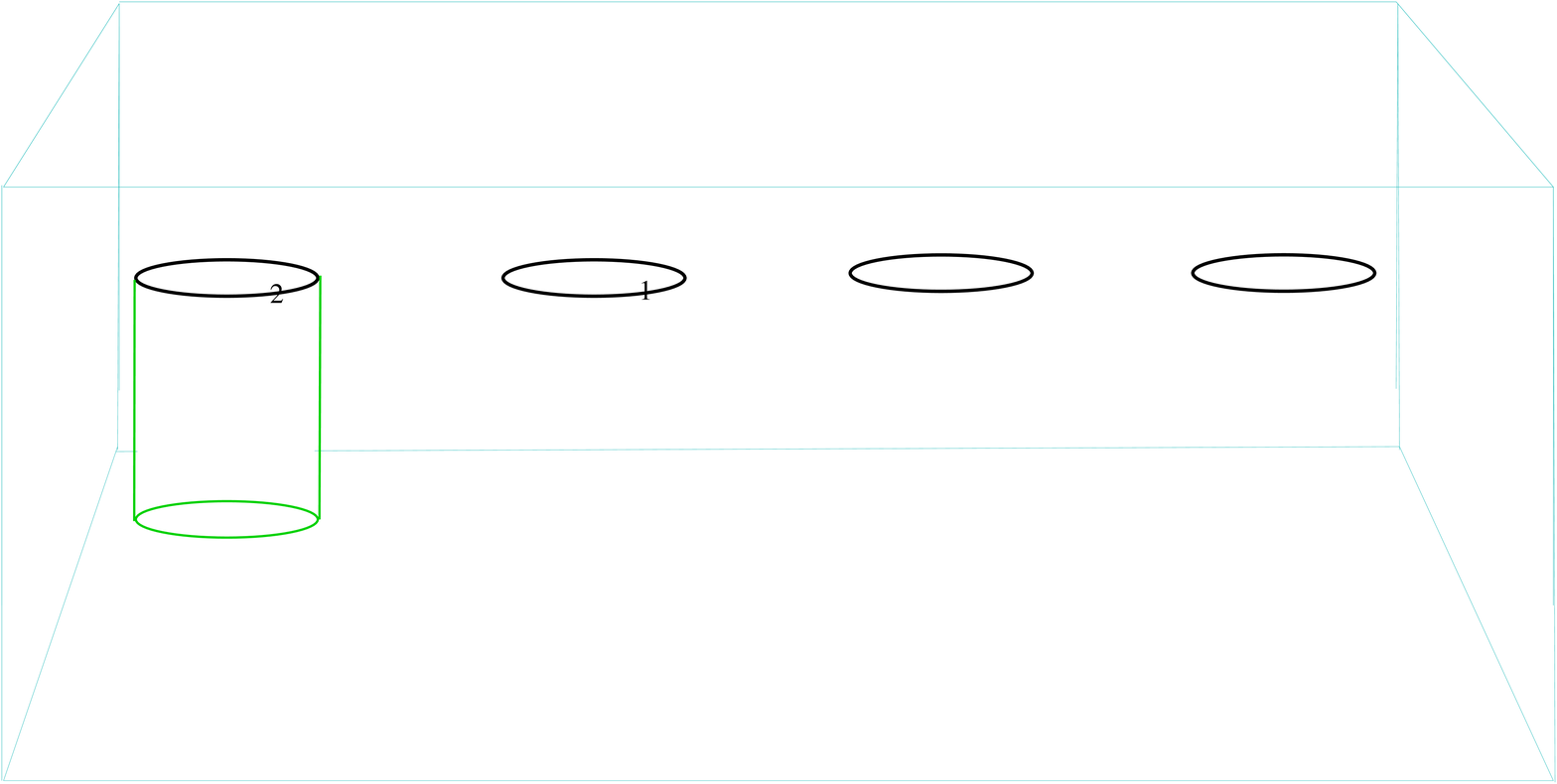}
    \caption{Sequence exchanging loop-1 and loop-2.
    The sequence runs top to bottom: 
    (a) Move loop-2 down. 
    (b) Move loop-1 over. 
    (c) Move loop-2 up.}
    \label{fig:79}
\end{figure}

\begin{figure}
    \centering
    \includegraphics[width=8cm]{xfig/loopnn2a12BOXgB.eps}
    \includegraphics[width=8cm]{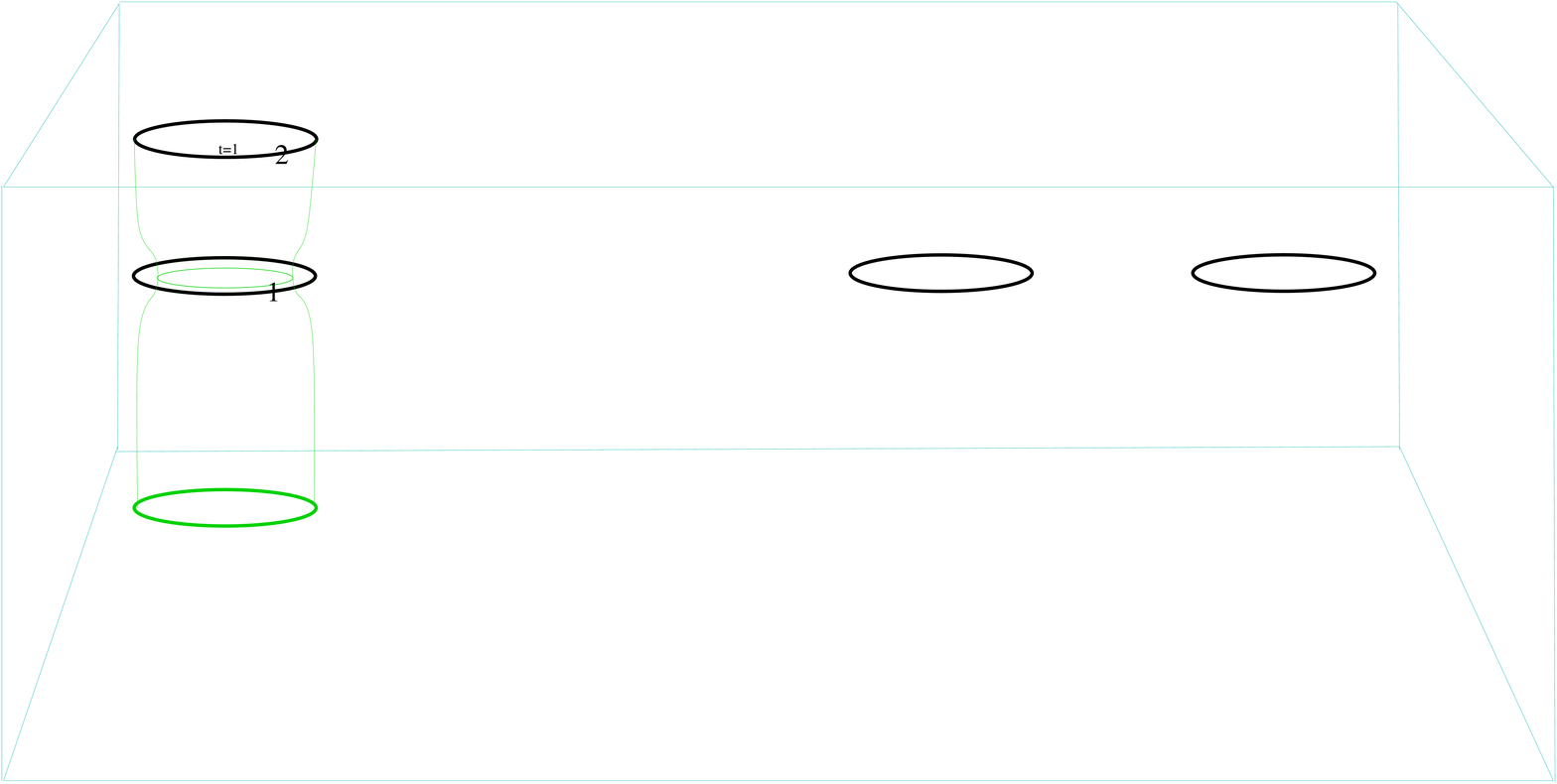}
    \includegraphics[width=8cm]{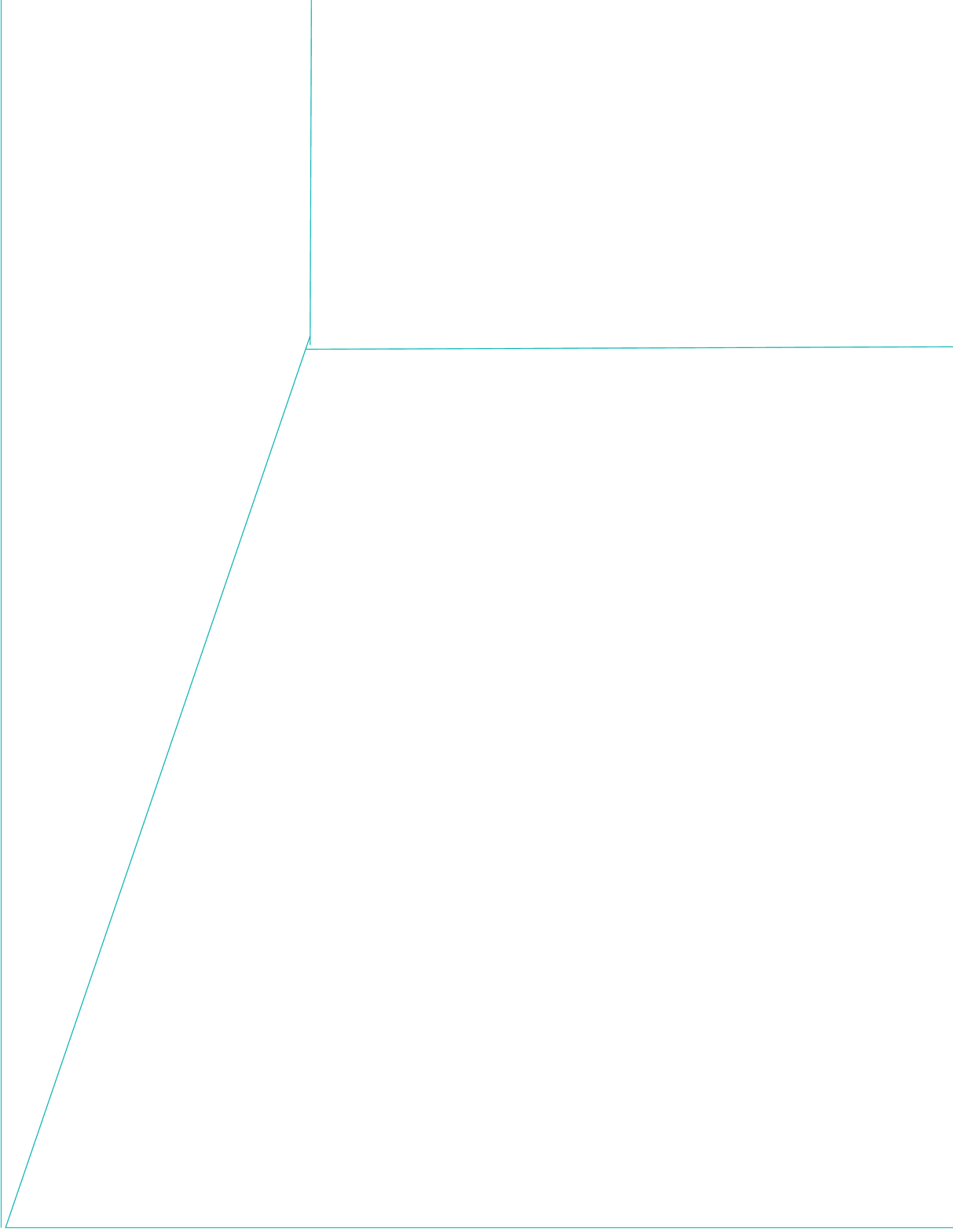}
    \includegraphics[width=8cm]{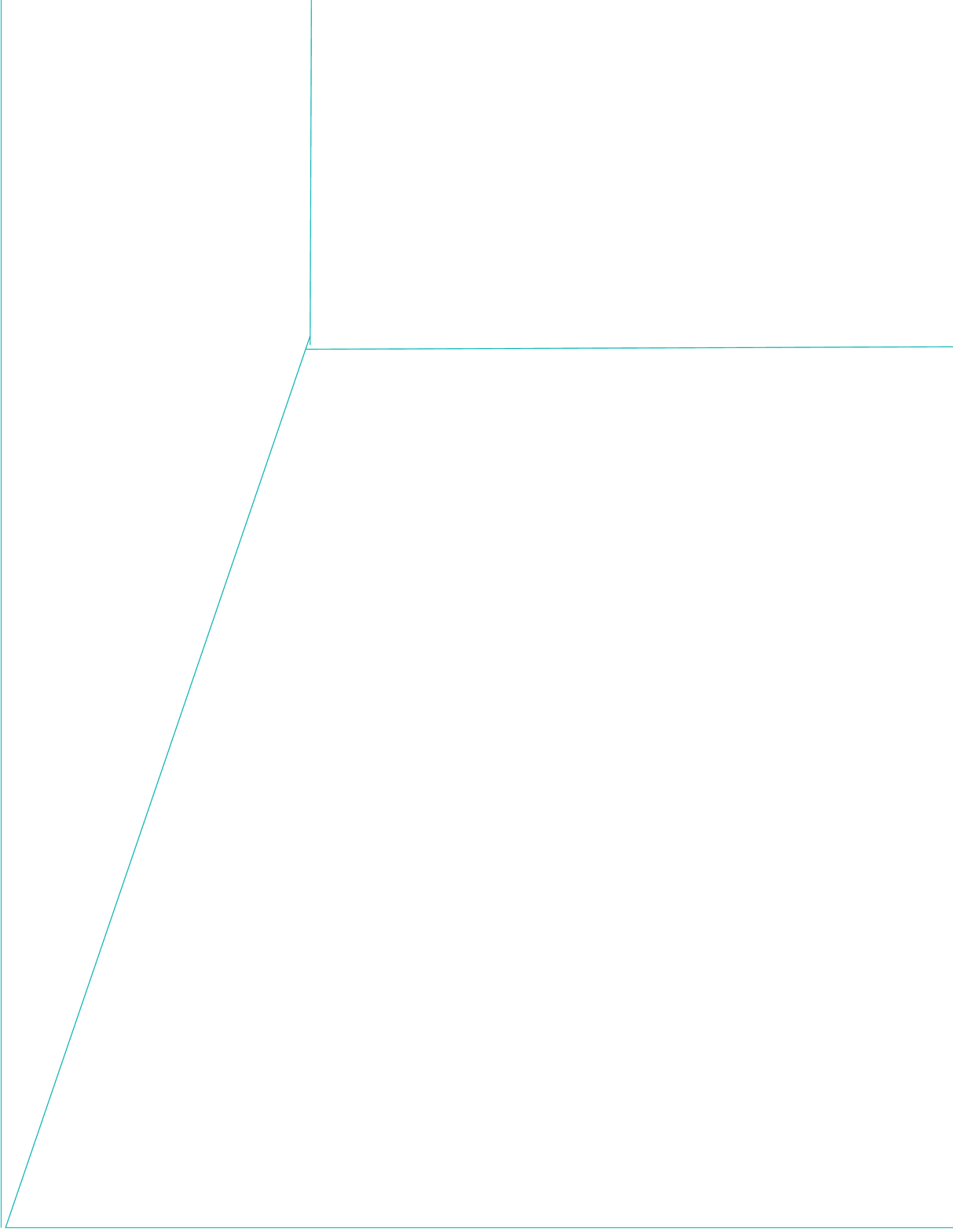}
    \caption{Sequence exchanging loop-1 and loop-2: 
    (a) Move loop-2 under loop-1. 
    (b) Move loop-2 up through loop-1. 
    (c) Move loop-1 over.
    (d) Move loop-2 down. }
    \label{fig:79xx}
\end{figure}

\mdef  \label{fact:lin}
{It is a useful fact that for each loop $i$ 
every 
equivalence class of 
motions contains a representative in which loop 
$i$ is at most translated during the motion (i.e.
circularity, attitude and size are preserved).
N.B. this is certainly not true for any {\em pair} of
loops.
}

\medskip

\begin{figure}
    \centering
    \includegraphics[width=3.4cm]{xfig/loopnnD3a.eps}
  $  \otimes \;$
    \includegraphics[width=3.4cm]{xfig/loopnnD3a.eps}
    \; =
    \raisebox{-.251in}{
    \includegraphics[width=6cm]{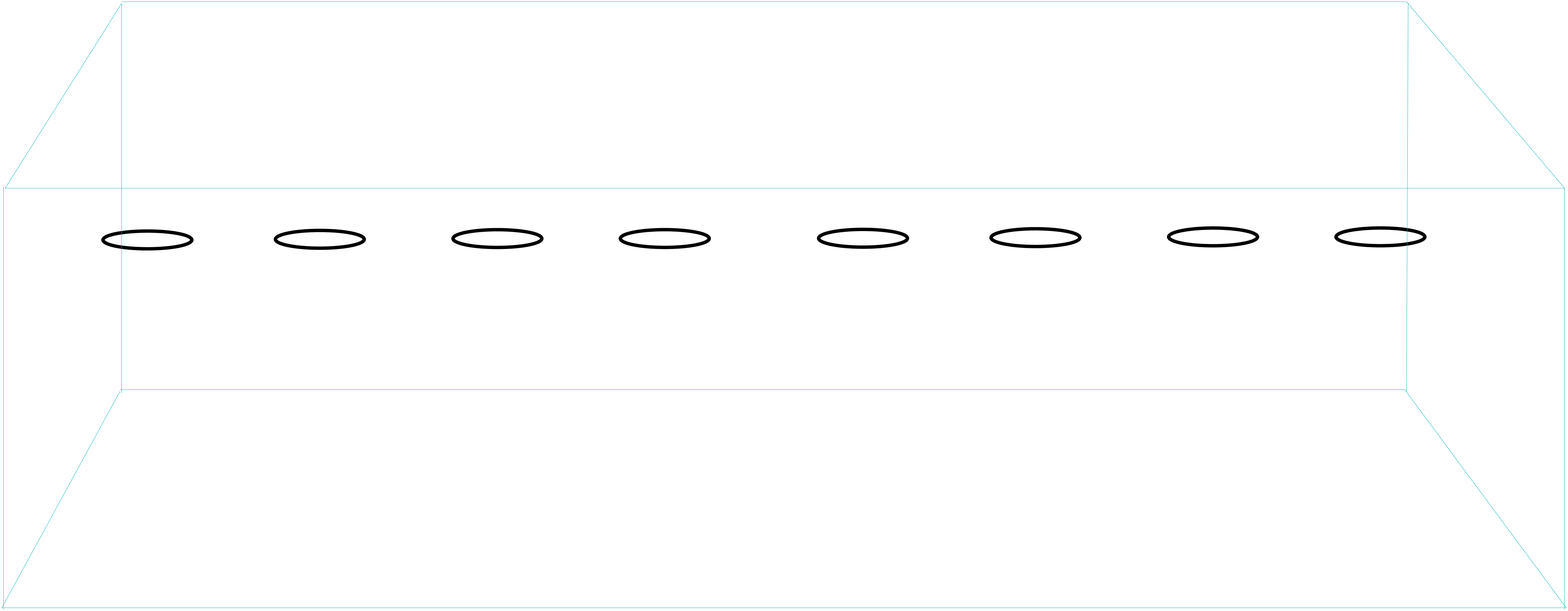}}
    \caption{Schematic for monoidal composition.}
    \label{fig:monoidal}
\end{figure}

\newcommand{\vars}{\varrho}  

\mdef \label{th:Lpres}
Here the loop braid category $\LL$ is the 
`category of loop braid motions'
(strictly speaking a subcategory of the category
of all loop motions), as follows.
\\ 
Let $\varsigma \in \LL(2,2)$ denote the class of
motions  in which two loops are exchanged
by one `passing through' another,
as in Fig.\ref{fig:79xx} (ignoring the extra two loops).\\ 
Let $\vars \in \LL(2,2)$ denote the class where the two loops 
pass around each other, as in Fig.\ref{fig:79}.
\\
Then $\LL$ is the category generated by these motions
and inverses.
(This does not include  motions in which a single
loop undergoes an orientation-reversing flip.)
The monoidal composition is as indicated in 
Fig.\ref{fig:monoidal}. 

We will not need to dwell on this 
beautiful but technical construction 
in order to do representation theory, because we have
the   following 
isomorphism (\ref{pa:iso}).
Indeed the main reason for recalling aspects of the
construction of $\LL$ here is to grasp why the 
presentation 
given by (\ref{pa:iso})
takes the form that it does.

\medskip 

\section{The setup, and the story so far} \label{ss:story}

We study \emph{natural} monoidal functors $\FF:\LL\rightarrow\Match^N$, that is monoidal functors such that $\FF(1)=1$.
A strict natural monoidal functor 
$\FF:\LL\rightarrow\Match^N$
is called $N$-{\emph{charge-conserving}}.
Next we explain how to work with $\LL$; explain key properties of $\Match^N$; and hence explain how to give a functor $\FF$. 

\medskip 

A presentation for the loop braid category 
as a strict monoidal category
may be given as follows.

\mdef  \label{de:L'}
The category $\LLL$ is the strict monoidal 
(diagonal, groupoid)  
category with object monoid the natural numbers,
and two generating morphisms (and inverses) 
both in $\LL'(2,2)$,  
call them $\sigma$ and $s$,
obeying 
\beq  \label{eq:s21} 
s^2 = 1 \otimes 1   
\eq 
where (as a morphism)
 1 denotes the unit morphism in rank one;
\soutx{$1_i = 1\otimes 1\otimes ...\otimes 1$ 
denotes the unit morphism in rank $i$}
\begin{equation}  \label{eq:sss}
s_1 s_2 s_1 = s_2 s_1 s_2
\end{equation} 
where $s_1 = s \otimes 1$ and $s_2=1\otimes s$,
\begin{equation} \label{eq:sigmass}
\mathrm{(I)} \; \sigma_1 \sigma_2 \sigma_1 = \sigma_2 \sigma_1 \sigma_2 
, \hspace{1cm}
\mathrm{(II)} \; \sigma_1 \sigma_2 s_1 = s_2 \sigma_1 \sigma_2
, \hspace{1cm} 
\mathrm{(III)} \; \sigma_1 s_2 s_1 = s_2 s_1 \sigma_2  .
\end{equation} 
Note that because of (\ref{eq:s21}), 
relation (\ref{eq:sigmass})(III) is equivalent to 
$s_1s_2\sigma_1=\sigma_2s_1s_2$.
On the other hand the `reverse' of (\ref{eq:sigmass})(II)
is not imposed.

\medskip 

\mdef   \label{pa:iso}
The map on generators of $\LL'$ to $\LL$
given by 
$\sigma \mapsto \varsigma$ and $s \mapsto \vars$
(recall  (\ref{th:Lpres}))
extends to an isomorphism. 
(For the $\sigma/\varsigma$ motion it makes a difference which
loop is which, 
leading to the asymmetry of (\ref{eq:sigmass})(II),
so later we will be careful with conventions.)

\mdef
By 
(\ref{pa:iso})
we may simply consider the representation theory of $\LL'$. 
Or equivalently, we have a representation of $\LL$ 
provided that the images of $\varsigma$ and $\vars$ 
obey identities corresponding to the presentation of 
$\LL'$ above. 
In practice we will simply identify the two categories henceforward. 

\medskip

Just as a strict monoidal functor 
$F: \Bcat \rightarrow \Match^N$ 
(or to any target)
can be given by
giving the image of the elementary braid 
$\sigma \in \Bcat(2,2)$, so a functor 
$F: \LL \rightarrow \Match^N$ can be given by
giving the pair 
\beq  \label{eq:Fss}
F_* =(F(s),F(\sigma)) 
\eq 
(here $*$ is some label for the representation).
We discuss this in more detail in \ref{ss:Appg+r}.
The
pairs 
that follow in \S\ref{ss:N2s} {\it et seq} 
thus give functors in this way.

\medskip

\mdef \label{pa:scheme}
An initial organisational scheme 
for solutions 
is provided by the classification on restriction to $\Bcat$, as in \cite{MR1X}, 
so we recall this briefly in \S\ref{ss:Brecall}.

The foundational aspect of this is the $\Match^N$ category itself, which we recall briefly   
in \S\ref{ss:MatchN}.

\newcommand{\aalph}{\ul{\alpha}}

\subsection{The story so far: \texorpdfstring{$\Match^N$}{MatchN} categories}  \label{ss:MatchN}

Here $\Mat$ is the monoidal category of matrices over
a given commutative ring $k$. We take $k=\C$.
(See (\ref{de:KP}) below for monoidal product conventions.)
For $N \in \N$ let $\ul{N} := \{ 1,2,...,N\}$.

A {\em natural} monoidal category 
is a strict monoidal category with
object monoid freely generated by a single object.
Recall from \cite{MR1X} that $\Mat^N$ denotes 
the natural monoidal subcategory of $\Mat$ generated 
(in the obvious sense)
by the object $N$ in $\Mat$ (renamed object 1 in $\Mat^N$, with $1\otimes 1=1+1$).

\mdef \label{de:cc}
The index set for rows of a matrix in $\Mat^N(n,m)$
is the set of words in $\underline{N}^n$, and similarly
for columns.
We sometimes write $|ijk\rangle$ to emphasise that the
word $ijk$ is being used as a column/row index.

A matrix $M\in\Mat^N(n,n)$ is {\em charge conserving}
if 
$M_{w,w'} = \langle w | M | w' \rangle \neq 0$ 
implies that 
$w$ is a perm of $w'$. 
That is $w = \sigma w'$ for some $\sigma\in\Sigma_n$,
where symmetric group 
$\Sigma_n$ acts by place permutation.

\mdef  \label{de:Match}
The subset of $\Mat^N$ of charge conserving (cc) 
matrices
forms a monoidal subcategory (see for example 
\cite[Lem.{3.7I}]{MR1X}) denoted $\Match^N$. 

\newcommand{\aada}{\smat a_1 \\ &a&b \\ &c&d \\ &&& a_2 \stam }

\mdef \mulem \cite[Lem.3.7III]{MR1X} \label{lem:37III} 
$\,$ 
For each $M\leq N$ and each injective function 
$\psi: \ul{M}\rightarrow\ul{N}$ there is a monoidal
functor  $f^\psi : \Match^N \rightarrow \Match^M$ given
on morphisms by $f^\psi (a)_{vw} = a_{\psi v,\psi w}$
($a\in\Match^N(m,m)$ say). 
\\
In particular we have a group action of the symmetric group $\Sym_N$ on $\Match^N$.
\qed

\smallskip 

\newcommand{\sigmax}{\omega} 

\noindent 
{\em Example}. Consider the nontrivial bijection
$\sigmax:\ul{2}\rightarrow\ul{2}$. On $\Match^2(2,2)$
this gives 
\[
\smat a_1 \\ &a&b \\ &c&d \\ &&& a_2 \stam 
\;\stackrel{f^\sigmax}{\mapsto}  \;
\smat 0&1\\1&0\stam \otimes \smat 0&1\\1&0\stam
\smat a_1 \\ &a&b \\ &c&d \\ &&& a_2 \stam
\smat 0&1\\1&0\stam \otimes \smat 0&1\\1&0\stam
\hspace{1.24in} \]
\beq \label{eq:fsig}  \hspace{.54in} 
=\smat 0&0&0&1\\0&0&1&0\\0&1&0&0\\1&0&0&0\stam
\smat a_1 \\ &a&b \\ &c&d \\ &&& a_2 \stam
\smat 0&0&0&1\\0&0&1&0\\0&1&0&0\\1&0&0&0\stam
=\smat a_2 \\ &d&c \\ &b&a \\ &&& a_1 \stam
\eq

\mdef 
Let $C$ be a natural category, as in \cite{MR1X}. 
For $\varsigma \in C(2,2)$ 
its image under a functor 
$F: C \rightarrow \Match^N$ 
with $F(1)=1$ will be some 
$R \in \Match^N(2,2)$.
Let $K_N$ denote the (simplicially directed) 
complete graph, as in \cite{MR1X}.
The (possibly) nonzero entries in $R$ are in 
correspondence with assignments of a scalar 
$a_i = a_i(R)$ to each vertex $i$ of $K_N$ and a 
matrix 
\[
\mat a_{ij} & b_{ij} \\ c_{ij} & d_{ij} \tam 
=
\mat a_{ij}(R) & b_{ij}(R) \\ c_{ij}(R) & d_{ij}(R) \tam 
\]
to each directed edge.
Altogether for $N=2$ we have
\beq \label{eq:Rabcd}
R = \mat 
a_1(R) \\ & a_{12}(R) & b_{12}(R) 
     \\ &c_{12}(R)&d_{12}(R)    \\ &&& a_2(R) 
\tam 
\eq

\mdef \label{par:KGraph}
Recall from (\ref{eq:Fss}) that to give a functor $\FF:\LL\rightarrow\Match^N$ 
it is enough to give the images $\FF(s)$ and $\FF(\sigma)$.
Here we assume the object map $\FF(1)=1$, so
$\FF(s)$ and $\FF(\sigma)$
$\in\Match^N(2,2)$.

For $N=2$ giving $M\in\Match^2(2,2)$ is easy to do explicitly, as in (\ref{eq:fsig})
or (\ref{eq:Rabcd}),
but for general $N$ it can be helpful to view 
elements of $\Match^N(2,2)$ 
geometrically,  
using
the complete graph $K_N$, as in \cite{MR1X}.
The point is that the non-zero elements of 
$M \in \Match^N(2,2)$ break up into $1\times 1$ and
$2\times 2$ blocks, with the former 
$a_i = \langle ii |M|ii \rangle$
indexed by 
vertices $i$ and the latter 
$A(i,j)$
(the submatrix 
$\smat a_{ij} & b_{ij}\\c_{ij}&d_{ij}\stam$
associated to $|ij\rangle$ and $|ji\rangle$)
by edges $(i,j)$ of $K_N$.

Alternatively we may encode a fixed $M$ as a list of 
the scalars followed by a list of the matrices
{in a suitable order}.
For our pair of matrices 
$R=\FF(\sigma),S=\FF(s)$, thus giving a functor, 
we will use:
\beq \label{eq:alphaform} 
\aalph(R) = 
(a_1, a_2, ...,  a_N, 
 A(1,2), A(1,3), A(2,3),...,A(N-1,N))
\eq 
\[
\aalph(S) = (b_1, b_2, ..., b_N, 
  B(1,2), B(1,3), B(2,3),...,B(N-1,N))
\]
That is, to each edge $(i,j)$ of $K_N$ we will associate 
two matrices, giving
$A(i,j)$ and $B(i,j)$, and so on, giving $\FF$ this way.

\subsection{The story so far: classification of braid representations} \label{ss:Brecall}

The classification of braid representations we need
can be given 
in rank $N$ indexing 
either by the set $\SSS_N$ of two-coloured multitableaux; 
or, working up to the $\Sym_N$ symmetry,
by the set $\TTT_N$ of multisets of 
two-coloured compositions, of total degree $N$ \cite{MR1X}.

In rank $N=2$ for example these multisets 
may  
be written 
\[ 
\two, \qquad \oneone, \qquad 
\oneonex,  \qquad 
\square\; \square 
\]
For each index $\lambda$ 
there is an
orbit of varieties of solutions.
A variety is obtained 
(complete with named  variables)
by numbering the boxes of $\lambda$ up to order within rows.
(We will recall them explicitly shortly.)
The size of this variety depends on whether one wants to give representations up to `gauge' isomorphism 
(i.e. a set of representations that are a transversal of isomorphism classes) or {\em all} representations. 
The former is natural for $\Bcat$ itself, but since 
extension to $\LL$ will restrict the symmetry of isomorphism it is
natural here to consider the latter.

We will retain the convenient language of \cite{MR1X}.
Thus an individual composition in a multiset is a {\em nation};
and a part in a composition is a {\em county}.

The rank $N=2$ and $N=3$
cases will be particularly important for us. 
In rank $N=2$ the six varieties have convenient 
special names. 
Case $\two$ 
--- a single county in a single nation ---
gives rise to the trivial or 0-variety of solutions.
Case $\one\;\one$ 
is the /-variety. 
Case $\oneonex$ 
is the $\aaa$-variety and 
(depending on box numbering) $\ul{\aaa}$-variety.
Case $\oneone$ yields the $\fff$-variety
and $\ul{\fff}$-variety.

\newcommand{\kkk}{\chi} 
\newcommand{\gammad}{\gamma} 

\mdef \label{par:rank2B}
We now recall explicitly 
{the classification for $N=2$ of} 
all $R\in \Match^N(2,2)$ which satisfy the relation (\ref{eq:sigmass} II),
{i.e. the complete set from \cite[Prop.3.21]{MR1X} of cc braid representations in rank $N=2$.}
We have six families of solutions given here. In each of the following families all variables must be non-zero.
Organisationally we also insist that in the last four cases $\alpha+\beta\neq 0 $ since 
the subset $ \alpha+\beta=0$   
is included in the $F_/$ family of solutions; 
and similarly in $F_f$ and $F_{\underline{f}}$ cases $\alpha\neq\beta$ as this is included in the cases $F_\aaa$ and $F_{\underline{\aaa}}$ respectively.
\begin{align}
F_0&=\begin{pmatrix}
\alpha &&& \\ &\alpha \\ &&\alpha \\ &&&\alpha
\end{pmatrix}, 
&&F_{/}=
\begin{pmatrix} 
\alpha &&& \\ && \gammad \kkk \\ & 
\frac{\gammad^{}}{\kkk}\\ &&&\beta 
\end{pmatrix}, 
&&&F_{\fii}=\begin{pmatrix} 
\alpha &&& \\ &\alpha +\beta &\kkk \\ &-\frac{\alpha \beta}{\kkk} \\ &&&\alpha
\end{pmatrix} 
\nonumber \\ \label{eq:ccBN2}
F_{\aaa}&=\begin{pmatrix} 
\alpha &&& \\ &\alpha +\beta &\kkk \\ &-\frac{\alpha \beta}{\kkk} \\ &&&\beta
\end{pmatrix}, 
&&F_{\underline{\fii}}=\begin{pmatrix} 
\alpha &&& \\ & &\kkk \\ &-\frac{\alpha \beta}{\kkk} & \alpha +\beta \\ &&&\alpha
\end{pmatrix}, 
&&&F_{\underline{\aaa}}=\begin{pmatrix} 
\alpha &&& \\ & &\kkk \\ &-\frac{\alpha \beta}{\kkk} & \alpha +\beta \\ &&&\beta
\end{pmatrix}
\end{align}
The parameter $\kkk$, where it appears, is here 
`unphysical', i.e. it can be changed by X-symmetry {(discussed in \S\ref{ss:0X})}
and does not affect the spectrum.
In contrast for $F_{/}$ say, $\alpha,\beta,\gammad$ all
affect the spectrum.
To match the `gauge choice' 
made in \cite{MR1X} we should set
$\kkk = \gammad$ in $F_{/}$ 
(and $\gammad$ here is $\mu$ there);
and $\kkk = \beta $ in other cases.

The aim of Section~\ref{ss:N2s} then is to 
establish 
which of the above braid solutions extend to loop braid solutions. Precisely, given an $F(\sigma)$ of one of the above types, when does there exist an $F(s)$ such that all \eqref{eq:s21}, \eqref{eq:sss}, and \eqref{eq:sigmass} are satisfied?

\mdef In case $N=3$ we have 
(from \cite[Prop.5.1]{MR1X}) the braid multisets
\[
\one^3 \qquad \one\; \two \qquad 
\one \oneone \qquad \one \oneonex 
\qquad \three \qquad \twoone \qquad \oneoneone 
\qquad \twoonex \qquad \oneoneonex \qquad 
\oneonexone 
\]
corresponding to the braid solutions as 
recalled in (\ref{pr:9rule}).
There are also 
$\onetwo$, $\onetwox$, and $\oneonexonex$, treated in \cite{MR1X} by using the additional flip symmetry.

The main point in extending the
$N=3$ braid solutions to loop braid 
is going to be to
note from Lemma~\ref{lem:anof} 
that extension at each edge 
is, up to sign, either by an identity matrix 
(extending type 0) --- so that the vertex eigenvalues
at each end are forced the same;
or by a signed 
permutation/$\PP$ 
matrix where 
the vertex signs are either forced different 
(type $\aaa$) or nominally free (type /);
with no extension for type $\fff$. 
We {will thus} see \soutx{immediately} the following rules: 
\\
1. In the same county vertices must have the same sign;
\\
2. In different counties in the same nation the signs
are different (hence there are at most two counties 
in a nation);
\\
And a `non-rule': Vertex signs are not constrained between different nations.

In \S\ref{ss:N2s} we obtain the $N=2$ Lemma.
In \S\ref{ss:prep} we will then implement the rules to 
obtain all solutions in all ranks.

\section{Preparations for calculus in Match categories} \label{ss:Match}

Both to prove the main Theorem, and the key Lemma, we will need to be able to compute in Match categories.
Here we develop the required machinery.

\mdef Let $\MM$ be a monoidal presentation for a 
strict monoidal,
indeed natural,
category $C$. 
We say a relation has width $n$ if it is in $hom_C(n,n)$.
For example the relation (\ref{eq:sigmass}(II)) above has width 3.
We say the presentation $\MM$ has width $w$ if $w$ is the
maximum width among relations in $\MM$. 
For example the presentation in 
(\ref{eq:s21}-\ref{eq:sigmass}) has width 3. 

\mdef
Let $C$ be presented by $\MM$.
We may give a monoidal functor $\F:C\rightarrow D$
by giving the images of the generators.
A function $F$ from generators to $D$ gives a functor 
$\F$ provided that the images of the relations hold.
The image of each relation $a=b$ say can be checked
by checking $F(A)=0$ where 
the `anomaly'
$F(A)=F(a)-F(b)
\in D(w,w)$ (we assume $D$ linear).  

As noted in (\ref{de:cc}), a basis element of the space acted on by
$\Match^N(w,w)$ is $| i_1 i_2 .. i_w \rangle$,
$i_j \in \{1,2,..,N\}$.
By the cc property every 
$\F(A)|i_1 i_2 .. i_w\rangle
 = \sum_{\sigma\in\Sigma_w} k_{\sigma}(\F(A)) |\sigma i_1 i_2 .. i_w \rangle$ for some scalars $k_\sigma$,
 indeed for every $M\in \Match^N(w,w)$,
$
M |i_1 i_2 .. i_w\rangle
 = \sum_{\sigma\in\Sigma_w} k_{\sigma}(M) |\sigma i_1 i_2 .. i_w \rangle
$,
i.e. the $\F(A)$ action on $\underline{N}^w$ breaks 
into $\Sigma_w$ orbits. 
Thus:

\mdef \mulem  \label{lem:restricto}
The image of a relation 
of width $w$ 
is verified if and only if
it is verified on each subspace 
$\{ \sigma|i_1 i_2 .. i_w\rangle \;|\; \sigma\in\Sigma_w  \}$, 
the subspace of the subset  
$\{ i_1, i_2, .., i_w \} \subset \ul{N}$. 
We note that this is the same as verification on
$\Match^{w}(w,w)$ up to relabelling, 
{\em for each subset}.
I.e. the same as to say that $F$ restricts to a
functor $\F' : \LL \rightarrow \Match^w$ 
on each subset.
(But note also that the various subsets interlock,
and every restriction must hold.)
\qed 

\mdef \textbf{Remark}. \label{rem:width 3}
Note that the monoidal presentation 
(\ref{eq:s21}-\ref{eq:sigmass})
has width 3. 
Thus a pair $F_*$ 
as in (\ref{eq:Fss})
induces a functor
$\F : \LL \rightarrow \Match^N$ if anomalies vanish
in every restriction to $\{i,j,k\}\subseteq \ul{N}$. 

Specifically 
for $R_1 R_2 R'_1 = R'_2 R_1 R_2$ say
--- with 
$\aalph(R) =
(A_1, A_2, ..,A_N, 
\smat A_{12}&B_{12} \\ C_{12}&D_{12}\stam,
\smat A_{13}&B_{13} \\ C_{13}&D_{13}\stam,
..)
$
(as in \ref{eq:alphaform}),
$\; \aalph(R') 
=(a_1, a_2, ..,a_N, \smat a_{12}&b_{12} \\ c_{12}&d_{12}\stam,..)$ 
say ---
every term in the anomaly acting on $|ijk\rangle$ 
is a cubic with indices in $i,j,k$.
E.g.
$A_{12} A_{23} a_{12}$, $A_{13} B_{12} c_{12}$
when $ijk=123$.

\subsection{X-symmetry} \label{ss:0X}

Let $C$ be a natural category (a strict monoidal category
with object monoid freely generated by a single object,
denoted 1).
The proof of the X-Lemma 
in \cite{MR1X}
observes that if $R$ is the image of a
generic 
element in $C(2,2)$ 
under a functor $F:C \rightarrow \Match^N$
--- thus with elements $a_{ij} = a_{ij}(R)$ and so on --- 
then the {\em braid anomaly} 
\[
A_R \; :=\; (R\otimes 1_N )(1_N\otimes R)(R\otimes 1_N)
   - (1_N\otimes R)(R\otimes 1_N)(1_N\otimes R)
\]
has entries that are cubics in the various indeterminate
entries in $R$;
but in particular in each entry we have one of the following (writing $b_{ij}$ for $b_{ij}(R)$ and so on):
\\
$b_{ij}$ (or $c_{ij}$) appears as an overall factor;
\\
$b_{ij}$ and $c_{ij}$ only appear in the form $b_{ij} c_{ij}$.

Each of the three factors in a term in the cubic come, note, from one of the three factors in a term in $A_R$. 
Now suppose we have a second element $S$ in $C(2,2)$,
with entries $a_{ij}(S)$ and so on. 
Then in $A_{RRS}$, say, one of the factors in each term
in a cubic will now comes from $S$, so the cubics are
modified by some $b_{ij}$ becoming $b_{ij}(S)$ and so on.
We observe that 
simultaneous conjugation by an invertible diagonal matrix
still preserves equalities for such cubics.
The general X-Lemma follows immediately from this.

\mdef \mulem \label{lem:X} 
Let $X\in \Mat^N(2,2)$ be any 
invertible diagonal matrix 
(hence $X\in \Match^N(2,2)$)
and $F: \{ \sigma, s \} \rightarrow \Match^N(2,2)$
be any pair. If $F$ induces a functor 
$\F:\LL \rightarrow \Match^N$
then so does $F^X$ where $F^X(\sigma)=XF(\sigma)X^{-1}$ and 
$F^X(s)=XF(s)X^{-1}$.
\qed  

\medskip

A more explicit proof is given in \S\ref{ss:ket}.

\medskip 

This construction gives an action 
on 
the set of all functors, denoted
$(\LL,\Match^N)$ 
(see also  
(\ref{de:funcat})),
of the abelian group
$ \Delta^N := \Mat^N_{\Delta}(2,2)$ 
of invertible diagonal matrices.

This action together with the action of $\Sym_N$ 
(the bijections from (\ref{lem:37III}))
generate 
a group of  symmetries of $(\LL,\Match^N)$,
that we can call `core symmetries'.

\subsection{Ket calculus: conventions}

\mdef  \label{de:KP}
Recall (see e.g. \cite{MR1X}) the convention that the 
$\Mat^N$ (and hence
$\Match^N$) categories use the Kronecker product in the Ab convention. 
The Ab convention is as indicated by:
\[
\mat a&b \\ c&d \tam \otimes \mat e& f \tam 
=
\mat ae & be & af & bf \\ ce&de& cf& df \tam,
\hspace{1cm}
\mat a_1 \\ a_2 \tam \otimes \mat b_1 \\ b_2\tam 
=\mat a_1 b_1 \\ a_2 b_1 \\ a_1 b_2 \\ a_2 b_2\tam 
\hspace{1cm}
\smat 1\\0\stam\otimes\smat 0\\1\stam 
= \smat 0\\0\\1\\0\stam
\]
(In fact either convention is fine, but we need to fix one.
Note that {\em Maxima} and  MAPLE use the aB convention by default.) 
This means in particular that the ordered basis 
1,2 of $\C^2$ 
(specifically we might take $|1\rangle=\mat 1\\0\tam$,
$|2\rangle =\mat 0\\1\tam$, although even this is a choice)
passes, for $(\C^2)^{2}$, to order 
11,21,12,22 with earlier 
indices changing more quickly.
That is $\ket{21} = \ket{2}\otimes\ket{1} = 
\smat 0\\1\stam\otimes\smat 1\\0\stam
= \smat 0\\1\\0\\0\stam
$.
Thus without further adulteration 
\[
\aada \ket{21} = \aada \smat 0\\1\\0\\0\stam = 
\smat 0\\a\\ c \\0\stam = a\ket{21} +c\ket{12} 
\]
\beq \label{eq:12bd}
\aada \ket{12} = \aada \smat 0\\0\\1\\0\stam = 
\smat 0\\b\\ d \\0\stam = b\ket{21} +d\ket{12} 
\eq 
\[
\bra{21} \aada  =  \smat 0&1&0&0\stam \aada = 
\smat 0&a& b &0\stam = a\bra{21} +b\bra{12} 
\]
For $M,N \in \Match^N(1,1)$ we have 
$M\otimes N |ij\rangle =M|i\rangle \otimes N|j\rangle$
and so on.
\newcommand{\kett}[1]{\ket{\ul{#1}}}
\footnote{
In \cite{MR1X} the notation $\ket{\ul{ij}}$ is used,
and different conventions are adopted.
The choice of conventions is, in any case, essentially
unimportant there. But here we must be careful.
}

\subsection{Back to generalities} 

\mdef 
Let $1_N$ denote the identity matrix in $\Match^N(1,1)$.
Given $R,S \in \Match^N(2,2)$ then 
define
\beq 
A_{RRS} \; =\; 
(R\otimes 1_N )(1_N\otimes R)(S\otimes 1_N)
   - (1_N\otimes S)(R\otimes 1_N)(1_N\otimes R)
   \;\;\;\; \in \Match^N(3,3) 
\eq 
In particular this yields the {\em braid anomaly} 
matrix 
$A_R = A_{RRR}$ by substituting $S=R$.

In general for $R \in \Match^N(2,2)$ we will write
$
R_1 = R \otimes 1_N
$ 
and $R_2 = 1_N \otimes R$. So then 
$A_{RRR} = R_1 R_2 R_1 - R_2 R_1 R_2$. 

\medskip 

In light of (\ref{eq:sigmass}) it will be useful to have facility with computing this kind of operator. One approach is to resolve charge-conserving matrices as sums of monomial matrices, as in \S\ref{ss:monom}. 
Another is to use a ket calculus as in \S\ref{ss:ket}.

\subsection{Calculus via sums of monomial matrices} \label{ss:monom}

\mdef 
Supposing that $N \in \N$ is fixed, 
let $\PPmat = \PPmat_{(N)}$ denote the simple permutation matrix in 
$\Match^N(2,2)$,
given by 
$\PPmat\, |ij\rangle = |ji\rangle$ for all $i,j$;
and $1_N$ the identity matrix in $\Match^N(1,1)$.
Indeed for given $N$ we may simply write $\PP$ for $\PPmat_{(N)}$ and $\IImat$ for $1_N$. 

Note that one way to resolve 
$R \in \Match^N(2,2)$  into a sum of two 
monomial matrices is 
\beq 
R  \;\; = \;\; \Delta(R) \; (1_N \otimes 1_N) \; 
     + \;\; D(R) \; \PPmat_{}
\eq 
where the diagonal matrix $\Delta(R)$ has zeros in the 
$ii$ positions and the diagonal matrix $D(R)$ does not. 
Thus for example with $N=2$:
\beq  \label{eq:D=}
\Delta(R) = 
\mat 0 \\ & a_{12}(R) \\ && d_{12}(R) \\ &&&0 \tam 
\hspace{.31cm} \mbox{and} \hspace{.531cm}
D(R)=
\mat a_1(R) \\ & b_{12}(R) \\ && c_{12}(R) \\ &&& a_2(R) \tam 
\eq 
so
\beq  \label{eq:R=}
R = \mat 0 \\ & a_{12}(R) \\ && d_{12}(R) \\ &&&0 \tam 
1_4
+
\mat a_1(R) \\ & b_{12}(R) \\ && c_{12}(R) \\ &&& a_2(R) \tam \PP_{}
\eq 

\mdef 
Consider the case $R=D(R) P_N$. 
A useful observation is
\beq 
D(R) \; \PP_{} \;=\; \PP_{} \; \xD(R)
\eq 
where $\xD(R)$ is given by swapping 
the 12,12 and 21,21 entries and so on, so here
$b_{ij}$ and $c_{ij}$.

So
for example in the simple monomial case,
with $\DR = D(r) \PP_{}$
and 
$\DR' = D(R) \PP_{}  $, say,
then 
\beq \label{eq:monomiRR}
\DR_1 \DR'_2 \; = \; 
((D(r) \PP_{}) \otimes 1_N ) \;\;
(1_N \otimes ( D(R) \PP_{} ))
\eq 
can be `straightened' by promoting the second $D(R)$.
Of course the tensor product affects the promotion.

For example with $N=2$,
and naming variables by:  
$\DR=\smat a_1 \\ &&b_{12}\\&c_{12}\\&&&a_2\stam$,
$\DR'=\smat A_1 \\ &&B_{12}\\&C_{12}\\&&&A_2\stam$,
$\DR''=\smat \alpha_1 \\ &&\beta_{12}\\&\gamma_{12}\\&&&\alpha_2\stam$,
the product in (\ref{eq:monomiRR}) becomes
\[
\DR_1 \DR'_2 = 
\smat \AR \\ &\BR \\ &&\CR \\ &&&\AAR \\
&&&&\AR \\ &&&&&\BR \\ &&&&&&\CR \\ &&&&&&&\AAR 
\stam P_1 
 \smat \ARx \\ &\ARx \\ &&\BRx \\ &&&\BRx \\
 &&&&\CRx \\ &&&&&\CRx \\ &&&&&&\AARx \\ &&&&&&&\AARx 
\stam  P_2
\]
\beq 
=
\smat \AR\ARx \\ &\BR\BRx \\ &&\CR\ARx \\ &&&\AAR\BRx \\
&&&&\AR\CRx \\ &&&&&\BR\AARx \\ &&&&&&\CR\CRx \\ &&&&&&&\AAR\AARx  
\stam  P_1P_2
\eq
\newcommand{\Pone}{\smat 1 \\ &&1 \\ &1 \\ &&&1 \\
&&&& 1 \\ &&&&&&1 \\ &&&&&1 \\ &&&&&&&1 
 \stam }
\newcommand{\Ptwo}{\smat 
1 \\ &1 \\ &&&& 1 \\ &&&&& 1 \\ && 1 \\ &&&1 \\ &&&&&& 1 \\ &&&&&&& 1
\stam }
Thus
\[
\DR_1 \DR'_2 \DR''_1 \; = \; 
((D(r) \PP_{}) \otimes 1_N ) \;\;
(1_N \otimes ( D(R) \PP_{} )) \;\;
((D(\rho) \PP_{}) \otimes 1_N )
\]
(note again the chosen variable names) 
can be `straightened' by 
\[ \hspace{-1.6cm}
\smat \AR\ARx \\ &\BR\BRx \\ &&\CR\ARx \\ &&&\AAR\BRx \\
&&&&\AR\CRx \\ &&&&&\BR\AARx \\ &&&&&&\CR\CRx \\ &&&&&&&\AAR\AARx  
\stam  P_1P_2
\smat \ARxx \\ &\BRxx \\ &&\CRxx \\ &&&\AARxx \\
&&&&\ARxx \\ &&&&&\BRxx \\ &&&&&&\CRxx \\ &&&&&&&\AARxx  
\stam  P_1
\]
\[ \hspace{-1cm} 
=\smat \AR\ARx\ARxx \\ &\BR\BRx\ARxx \\ &&\CR\ARx\BRxx \\ &&&\AAR\BRx\BRxx  \\
&&&&\AR\CRx\CRxx \\ &&&&&\BR\AARx\CRxx \\ &&&&&&\CR\CRx\AARxx \\ &&&&&&&\AAR\AARx\AARxx  
\stam  P_1P_2P_1
\]
\newcommand{\deight}[8]{\smat #1 \\ &#2 \\ &&#3 \\ &&&#4 \\
&&&&#5 \\ &&&&&#6 \\ &&&&&&#7 \\ &&&&&&&#8  
\stam}
Meanwhile
\[
\DR''_2 \DR_1 \DR'_2 \; = \; 
(1_N \otimes (D(\rho) \PP_{})  ) \;\;\;
((D(r) \PP_{}) \otimes 1_N ) \;\;
(1_N \otimes ( D(R) \PP_{} )) \;\;
=
\]
\[
 \hspace{-1.6cm}
 \smat \ARxx \\ &\ARxx \\ &&\BRxx \\ &&&\BRxx \\
&&&&\CRxx \\ &&&&&\CRxx \\ &&&&&&\AARxx \\ &&&&&&&\AARxx  
\stam P_2
\smat \AR\ARx \\ &\BR\BRx \\ &&\CR\ARx \\ &&&\AAR\BRx \\
&&&&\AR\CRx \\ &&&&&\BR\AARx \\ &&&&&&\CR\CRx \\ &&&&&&&\AAR\AARx  
\stam P_1P_2 
\]
\[
= \smat \ARxx\AR\ARx \\ &\ARxx\BR\BRx \\ &&\BRxx\AR\CRx \\ &&&\BRxx\BR\AARx \\
&&&&\CRxx\CR\ARx \\ &&&&&\CRxx\AAR\BRx \\ &&&&&&\AARxx\CR\CRx \\ &&&&&&&\AARxx\AAR\AARx    
\stam P_2P_1P_2
\] 
From this elementary warm-up exercise 
we observe immediately that the permutation factors 
in $\DR_1 \DR'_2 \DR''_1 $ and $\DR''_2 \DR_1 \DR'_2 $ agree:
$\PP_1 \PP_2\PP_1=\PP_2\PP_1\PP_2$, 
so that for example, we have the following.

\medskip

\begin{lemma} \label{lem:DPbraidanom}
Fix $N=2$ and $\PP= \PP_{(N=2)}$. 
For all $\DR = D(\DR)\; \PP\;$ and 
$R'' = \Delta(R'') 1_4 +  D(R'') \PP$
(including singular) 
\beq \label{eq:RRR''} 
\DR_1 \; \DR_2 \; R''_1 \; =\;  R''_2 \; \DR_1 \; \DR_2
\eq 
\end{lemma} 
\proof{
First consider the case $R''=\DR'' = D(R'') \PP$. 
Compare  the diagonal factors in the evaluations of 
$\DR_1 \DR'_2 \DR''_1$ and $\DR''_2 \DR_1 \DR'_2$ above 
and note that
in the present case 
$\DR=\DR'$ so 
$a_1 = A_1$, $c_{12} =C_{12}$, so 
$a_1 C_{12} =a_1 c_{12} = A_1 c_{12}$ and so on. 
Thus 
(\ref{eq:RRR''}) holds in this case.

Now suppose we turn on the diagonal terms in 
(\ref{eq:R=}) for $R''$. 
Thus $R'' = \Delta''+Q''$, say. 
Of course $\DR_1 \DR_2 Q''_1 = Q''_2 \DR_1 \DR_2$ 
from the established case of (\ref{eq:RRR''}) so the
full version 
requires
$\DR_1 \DR_2 \Delta''_1 -\Delta''_2 \DR_1 \DR_2$ to vanish.
Again by the established part of the Lemma we have 
$\DR_1 \DR_2 \Delta''_1 \PP_1 =\Delta''_2 \PP_2 \DR_1 \DR_2$,
so $\DR_1 \DR_2 \Delta''_1 
=\Delta''_2 \PP_2 \DR_1 \DR_2 \PP_1
= \Delta''_2 \PP_2 \PP_2 \DR_1 \DR_2  $  
(using established special case
$\DR_1 \DR_2 \PP_1 = \PP_2 \DR_1 \DR_2$ at the last step)
as required.
\qed}

\medskip 

Thus in particular we have, for each 
$\DR =D(\DR) \PP$, a
solution to the braid relation by putting 
$R''=\DR$,
but we also have solutions to the mixed relations
for any $R''$ of the more general form.

\section{Rank \texorpdfstring{$N=2$}{N=2} loop braid solutions}

We will prove 
here  
that there are
three kinds of solutions in rank $N=2$, given 
in Lem.\ref{lem:anof}.

\mdef  \label{pa:00}
A couple of organisational principles are convenient to have in mind for spin-chain braid representations
$F:\Bcat\rightarrow\Match^N$. 
These are derived more or less directly from \cite{MR1X} 
(and familiarity with this 
paper will significantly help the reader here). 
\\
(P.I) Each braid representation in rank $N$ 
restricts to a representation in lower rank by taking a subset of indices $\{ 1,2,..., N\}$. 
Thus in particular each representation 
with $N>2$ restricts to  a collection of $N=2$ representations
(just as complete graph $K_N$ restricts to a collection of $K_2$s).
And these $N=2$ representations 
fall into one of six types:
0, /, $\pai$, $\mai$, $\fff$, $\mfii$  
(recalled explicitly in (\ref{eq:ccBN2})).
\\
(P.II) Noting that each $N=2$ braid representation 
is   given by $F(\sigma)$ parameterised by 
\[
F(\sigma) = 
\mat 
a_1 \\ &a_{12}&b_{12} \\ &c_{12}&d_{12} \\ &&&a_2
\tam  
\]
we note (a)
that either $F(\sigma)$ is a scalar multiple of
the identity, or at least one of $a_{12}, d_{12}$ vanishes;
(b) that $a_1, a_2$ are eigenvalues, and $a_{12}+d_{12}$
is the sum of the `middle two' eigenvalues.
\\
(P.III) 
\label{p:F(s)} 
Restricted to the generator $s$ our $F(s)$ must give a braid representation that is also a representation of the permutation category --- of symmetric groups.
In particular $s^2 =1$ so the Jordan form of $F(s)$ is
diagonal.
For $F(s)$ with $N=2$ its
middle $2\times 2$ has eigenvalues 
$+1,+1$ or $-1,-1$,  
or eigenvalues $1,-1$.
Thus either it is up-to-sign the identity matrix
--- whereupon we have $F(s)=\pm 1_4$, type 0 by the
classification;
or else its diagonal entries obey
$a_{12} + d_{12} = 0$
and hence (cf. (P.II)) 
$a_{12} = d_{12} = 0$ --- type /.  
\\
(P.IV) By the X-Lemma (see (\ref{lem:X})) 
each X-orbit of solutions has a representative where the
nonzero off-diagonal elements of $F(s)$ are all 1.

\subsection{Rank \texorpdfstring{$N=2$}{N=2} loop braid solutions: main Lemma} \label{ss:N2s}

Here we give a complete set of 
solutions in rank $N=2$:

\begin{lemma} \label{lem:anof}
If
$(F(s),F(\sigma))$ gives a loop braid functor
$\F : \LL \rightarrow \Match^2$
then it is one of the following
(organised according to the braid representation type
of $F(\sigma)$).
\\
(I) If
$R = F(\sigma)$ is of type $\aaa$  
or type $\mai$
then $F(s)$ must take the form 
$S=\Stwo{1}{1/c}{c}{-1} = \Sigma \PP \;$ 
where $\Sigma = \dfour{1}{1/c}{c}{-1}$
for some $c\neq 0$,
up to overall sign.
Then for each specific such $S$, i.e. each $c$, 
we get a type $\aaa$ solution if and only if
for some $A_1, A_2$ we have
$R=\smat A_1 \\&A_1+A_2&A_1/c\\&-cA_2&0\\&&&A_2\stam$.
And similarly in type $\mai$. 
Note this means that  
the braid ``gauge'' parameter $\kkk$ in 
(\ref{eq:ccBN2}) is locked to $\kkk = A_1/c$.
In particular if we want again to be free to choose
$\kkk$ then $c = A_1/\kkk$ is forced; and if we want to chose $c$ (to set $c=1$ say, see later), for given $A_1$,
then 
$\kkk  $ is {\em not} free.
Note $A_1,A_2 \neq 0$ by invertability and $A_1+A_2\neq 0$ in type $\aaa$ or $\mai$. 
Note, in this case 
$SR = \smat A_1 \\ & -A_2 \\ &c(A_1+A_2)& A_1 
        \\ &&& -A_2 \stam$.
\\
Applying $f^\sigma$ from (\ref{eq:fsig}) we get 
$f^\sigma (S)  = \smat -1 \\ &&c \\ &1/c \\ &&& 1\stam$
and 
$f^\sigma (R) = \smat A_2 \\ &0&-cA_2 \\ &A_1 /c &A_1 +A_2 \\ &&& A_1 \stam$.
We can bring this $S$ back into the previous form by applying the homomorphism given by $S\leadsto -S$
and taking a new $c$ given by $c\leadsto -1/c$.
Since (almost) all $A_1,A_2$ give solutions, 
interchanging them takes us to a point on the same variety, 
then giving
$R\leadsto \smat A_1 \\ &0&A_1 /c \\ &-c A_2  &A_1 +A_2 \\ &&& A_2 \stam$,
which, with the restored $S$, 
is thus a viable parameterisation for type $\mai$.
\\
(II) \hspace{.01cm} If $R = F(\sigma)$ is of type $\fff$ or $\mfii$
there is no solution. 
\\ (III) If $R = F(\sigma)$ is of type 0
then $S=\pm 1_4$,
and there is no further constraint on $F(\sigma)$.
\\ (IV) If $R = F(\sigma)$ is of type / then 
$S = \Sigma \PP \;$ where $\Sigma = \dfour{1}{1/c}{c}{\pm1}$
for some $c\neq 0$,
up to overall sign,
and then
$R = \smat A_1 \\&&\mu/C c 
          \\& {\mu} C c \\ &&& A_2\stam$
with the four further variables 
$A_1, A_2, C, \mu$ non-zero but otherwise free.
Note   
these variables
are independent of $c$.
To match (\ref{eq:ccBN2}) we put 
${\mu} Cc=\gamma^{}/\kkk$
and $\mu/Cc = \gamma \kkk$, 
so $\mu / \gamma^{} =\gamma/\mu = Cc\kkk$
so $\mu=\gamma$ and 
$Cc = 1/\kkk$.  
The parameterisation has been chosen so that 
$c$ again captures the X-symmetry, while 
$A_1, A_2, C, \mu$ are
`physical/non-gauge' in the sense that
(unlike $c$) they {\em do} affect operator spectrum.
In particular 
$SR = \smat A_1 \\ &{\mu} C
          \\&&\mu/C \\ &&& A_2 \stam$.
\end{lemma}
\proof{ 
Observe that all possibilities for $F(\sigma)$ are
considered, by (\ref{eq:ccBN2}), so it remains to verify the
extensions $F(s)$ in each case and determine any 
further constraints on $F(\sigma)$.
First consider the general form of 
$F(s) = S$.
The braid classification 
as in (\ref{eq:ccBN2})
applies to $F(s)$, 
and (\ref{pa:00})(P.III), so
$S$ is type 0 or /. 
In cases (I), (II), (IV)
$R_1 \neq R_2$, so 
from the $R_1S_2S_1=S_2S_1R_2$ identity 
(\ref{eq:sigmass})(III)
we have that $S$ cannot be the 
unit matrix (or minus) here.
So by (\ref{pa:00})(P.III)
we have that 
$S=\Sigma \PP \;$ where 
$\Sigma = \dfour{1}{1/c}{c}{\pm1}{}{}{}{}$ 
for some $c$ and some sign,
up to an overall sign.
Conversely in case (III) $S$ must be type 0.

\medskip 

\noindent 
Cases (I) and (II):
\\
Here we consider $F(\sigma)$ of form  $R=\Delta+\DR = \Delta +D\PP\;$
where $\Delta = \dfour{0}{a}{d}{0}$
and 
$$
D\PP = \smat A_1 \\ &&B \\ & C \\ &&& A_2\stam
=\dfour{A_1}{B}{C}{A_2} \smat 1 \\&&1 \\&1 \\&&&1\stam
$$
for given {$,A_1,A_2$} and suitable $a,d,B,C$.
This is the most general form, but we know in particular
from (\ref{pa:00})(II) that $ad=0$.
For this case of $R$, then, the RRS anomaly
here takes the form
\beq \label{eq:anomddS}
(\Delta +\DR)_1 \; (\Delta +\DR)_2 \; S_1 
 \; - \; S_2 \; (\Delta +\DR)_1 \; (\Delta +\DR)_2
\eq 
Expanding (\ref{eq:anomddS}) we get
\[
\underbrace{\DR_1 \DR_2 S_1 - S_2 \DR_1 \DR_2}_{=0}
+
\DR_1 \Delta_2 S_1 - S_2 \DR_1 \Delta_2
+
\Delta_1 \DR_2 S_1 - S_2 \Delta_1 \DR_2
+
\underbrace{\Delta_1 \Delta_2 S_1}_{=0} -
\underbrace{ S_2 \Delta_1 \Delta_2}_{=0}
\]
where the first cancellation follows from \ref{lem:DPbraidanom};
and the $\Delta_1 \Delta_2 = 0$ because by \cite{MR1X}
in type $\aaa$ or $\fff$ we have
$\Delta = \smat 0\\&a \\ &&0 \\ &&& 0 \stam$
(or indeed 
$\Delta = \smat 0\\&0 \\ &&d \\ &&& 0 \stam$
for $\underline{\aaa} $ or $\mfii$
which 
cases we can treat by symmetry).

Keeping to the $d=0$ case for now,
the first term we need to compute is thus
\[
\DR_1 \Delta_2 S_1 = 
D_1 \PP_1 \Delta_2 \Sigma_1 \PP_1
=
D_1 \PP_1 \!\!
\smat 0\\ &0\\ &&a \\ &&&a \\ 
&&&&0 \\&&&&& 0 \\ &&&&&& 0 \\ &&&&&&& 0 \stam\!\!
\smat 1 \\ & 1/c \\ && c \\ &&& \pm 1 \\ 
&&&& 1 \\ &&&&& 1/c \\ &&&&&& c \\ &&&&&&& \pm 1 \stam \!\!
\PP_1
=
D_1 \!\!\smat 0 \\ & ca \\ && 0 \\ &&&\pm a \\ 
&&&&0 \\ &&&&& 0 \\ &&&&&& 0 \\ &&&&&&& 0  \stam\!\!
\PP_1^{\,2}
\]
while the second is
\[
S_2 \DR_1 \Delta_2 = 
\Sigma_2 \PP_2 D_1 \PP_1 \Delta_2
=
\smat 1 \\ & 1 \\ && 1/c \\ &&& 1/c \\
&&&& c \\ &&&&& c \\ &&&&&& \pm 1 \\ &&&&&&&\pm1\stam
\PP_2 \deight{A_1}{B}{C}{A_2}{A_1}{B}{C}{A_2} \PP_1 
\deight{0}{0}{a}{a}{0}{0}{0}{0}
\]
\[
=\smat 1 \\ & 1 \\ && 1/c \\ &&& 1/c \\
&&&& c \\ &&&&& c \\ &&&&&& \pm 1 \\ &&&&&&&\pm1\stam
\PP_2 \deight{0}{Ba}{0}{A_2 a}{0}{0}{0}{0} \PP_1 
=\deight{0}{B a}{0}{0}{0}{c A_2 a}{0}{0}  \PP_2 \PP_1
\]
then
\[
\Delta_1 \DR_2 S_1 = \deight{0}{a}{0}{0}{0}{a}{0}{0}\!\!\!
\deight{A_1}{A_1}{B}{B}{C}{C}{A_2}{A_2} \!\!
\PP_2 \!\! \deight{1}{1/c}{c}{\pm1}{1}{1/c}{c}{\pm1}  \!\!\PP_1\]
\[= \!\! 
\deight{0}{a A_1 /c}{0}{0}{0}{aC(\pm1)}{0}{0}  \!\!\PP_2 \PP_1
\]
and
\[
S_2 \Delta_1 \DR_2  = 
\Sigma_2 \PP_2 \!\!
\deight{0}{a}{0}{0}{0}{a}{0}{0} \!\!\!
\deight{A_1}{A_1}{B}{B}{C}{C}{A_2}{A_2}
\PP_2  \!
\]
\[
=  \!\!
\deight{1}{1}{1/c}{1/c}{c}{c}{\pm1}{\pm1} \!\! \PP_2 \!\!
\deight{0}{a A_1}{0}{0}{0}{a C}{0}{0}
\PP_2 
= \deight{0}{a A_1}{0}{aC/c}{0}{0}{0}{0}
\]
Observe from the $\PP$ factors
that $\DR_1 \Delta_2 S_1$ 
must cancel $-S_2\Delta_1\DR_2$; 
and
$-S_2\DR_1\Delta_2$ must cancel $\Delta_1\DR_2 S_1$.
Necessary and sufficient in both cases are:
\beq  \label{eq:Bca}
Bca = a A_1 , \hspace{1cm}
\pm a A_2 = aC/c
\eq 
Here $a \neq 0$ 
and so $B=A_1 / c$ and $C= \pm c A_2$, so 
$BC = \pm  A_1 A_2$. 
Since 
$BC$ is determined by $A_1, A_2$
we 
are in type $\aaa$ not $\fff$, so $BC = -A_1 A_2$ so the 
relative sign in 
$S$ is forced $-$ as given in the Lemma;
and $a=A_1 +A_2$ and $R$ is as given.
This concludes the proof of (I) and (II).
}

Note that if we want to fix $c=1$ 
in case (I)
(using the X-symmetry as in Lemma~\ref{lem:X})
then 
(having also fixed that $\bra{11} S \ket{11} =1$
and determined then that
$\bra{22} S \ket{22} =-1$)
we also fix 
$B=A_1$ and $C=-A_2$. 
The only way to vary these latter two is to vary $c$
in $S$. 

\medskip 

\noindent 
Case (III): 

Notice first that in each case, \eqref{eq:s21}, \eqref{eq:sss}, and \eqref{eq:sigmass} are clearly satisfied.
With  $F(\sigma)$ of the form $F(\sigma)=\alpha 1_4$, 
writing the ansatz 
\beq  \label{eq:aij}
F(s) = \smat 
a_1 
\\ & a_{12} & b_{12} 
\\ & c_{12} & d_{12}
\\ &&& a_2
\stam
\eq 
then (\ref{eq:sigmass} II) gives $F(s)\otimes 1_2 = 1_2 \otimes F(s)$, hence
\[
\smat 
a_1 
\\ & a_{12} & b_{12} 
\\ & c_{12} & d_{12}
\\ &&& a_2
\\ &&&& a_1 
\\ &&&&& a_{12} & b_{12} 
\\ &&&&& c_{12} & d_{12}
\\ &&&&&&& a_2
\stam
=
\smat a_1 \\ & a_1 \\ && a_{12} & & b_{12} 
\\ &&& a_{12} && b_{12} 
\\ && c_{12} && d_{12}
\\ &&& c_{12} && d_{12} 
\\ &&&&&& a_2
\\ &&&&&&& a_2
\stam 
\]
We immediately read off 
$a_1 = a_{12} = d_{12} =a_2 $
and 
$b_{12} = 0 = c_{12}$.
With $s^2 =1$, this 
gives the result.

\noindent
Case (IV): 

From the first paragraph, any $F(s)$ exist extending the $/$ variety, they are also in the $/$ variety.

\newcommand{\dd}{\mathfrak{d}} 

Now, with 
\[
F(s) = \dd \PP = 
   \smat a_1 \\ & b_{12} \\ && c_{12} \\ &&& a_2 \stam
   \smat 1 \\ &&1 \\ &1 \\ &&& 1 \stam ,
\]
(\ref{eq:sigmass} II) becomes
\[
 (D \PP\otimes 1_2 )(1_2 \otimes D\PP)(\dd\PP\otimes 1_2)
-(1_2 \otimes \dd\PP) (D \PP\otimes 1_2 )(1_2 \otimes D\PP)
= 0
\]
which, by Lemma~\ref{lem:DPbraidanom} is satisfied for any matrix $\dd$.
Using Lemma\ref{lem:DPbraidanom} again, (\ref{eq:sigmass} III) adds no additional constraints, thus using (\ref{p:F(s)}) we have all solutions.
\qed

\section{Constructions for the main Theorem} \label{ss:prep}

Our task here is first to give  certain 
combinatorial sets $\SSSS_N$ and 
$J_N^{\pm} \hookrightarrow \SSSS_N$ 
for each rank $N$; 
second to give a parameter space for each 
$\lambda \in \SSSS_N$; 
and third to explain how each element
and point in parameter space 
yields a decoration of the complete graph $K_N$,
that encodes a pair of matrices in $\Match^N(2,2)$, 
and hence formally a monoidal functor 
$\FF: \LL \rightarrow \Match^N$
Finally 
we must prove that this pair indeed gives a functor,
and that all such functors arise this way.

Again following \cite{MR1X}, we can enumerate 
solutions in two ways. 
One is to give all solution varieties;
and the other is to give a transversal up to 
the $\Sym_N$ symmetries manifested by Lem.\ref{lem:37III}.
(Recall, e.g. from \cite{MR1X}, that the 
absolute notion of isomorphism is less straightforward 
for monoidal-category representation theory than for 
Artinian representation theory. Our objective here is 
not to give a transversal with respect to some ultimate notion of isomorphism,
but to understand all representations,
taking advantage of the isomorphisms that serve this practical end.)

\subsection{Index sets for enumerating solution varieties}

Here we construct the 
index sets $\SSSS_N$,  
and corresponding parameter spaces 
(we match to actual solutions in \S\ref{ss:recipe1}). 

\mdef  \label{de:SSSN}
Recall from \cite{MR1X} that 
for braid representations in rank-$N$ 
we start with a 
partition of the vertices of $K_N$,
calling the parts `nations'.
We then partition each nation further into `counties'.
In \cite{MR1X} the counties are ordered  in each nation
and partitioned into two subsets. The set of such structures is the set $\SSS_N$
(whose elements are visualised in \cite{MR1X} as collections of two-coloured composition tableaux). 

In our loop-braid 
case we 
find that, in order for a braid representation to extend to loop-braid, we must 
restrict to at most two counties per nation. The counties in each nation are ordered, so this amounts to the choice of a (first county) subset for each nation. (We then colour the counties from two colours, but in our case the colours are forced.)

Next comes the new ingredient that is not merely a restriction, 
which is that a subset of the nations is chosen (those that will be associated to +, i.e. +1 eigenvalue of $S$ for individuals in the first county). 

The above features characterise the loop-braid index set  $\SSSS_N$. 
It is useful also to give  
a more formal construction for $\SSSS_N$ as follows.

We continue to use notation from \cite{MR1X}. 
And we add a few more devices. 
In particular for $S$ a set, $\Pas(S)$ denotes the set of partitions of $S$ into an ordered pair of parts
(thus $\Pas(S)$ is in bijection with the power set
$\Power(S)$).
Example:
\[
\Pas(\{x,y\}) = \{ 
  (\{x,y\},\emptyset),\; (\{x\},\{y\}), \;
  (\{y\},\{x\}),\; (\emptyset,\{x,y\}) \}
\]
And $\Pas'(S)$ the subset of these in which the first part is not empty. 
Further, given an indexed set of sets $\{ S_i \}_{i \in I}$ that are disjoint we write 
$\prod_{i \in I} S_i$ for the set whose elements are 
sets made by selecting one element from each $S_i$.

\newcommand{\tabb}{\ytableausetup{smalltableaux}}
\newcommand{\bu}[1]{\begin{ytableau} #1 \end{ytableau}}

Given a partition $p$ of $\ul{N} {=\{1,\ldots,N\}}$ let us write 
$$ 
Q_p =  
\prod_{q\in p} \Pas'(q)
$$  
(thus the set given by the choices of a non-empty subset $p_{i1}$ of 
each part $p_i$, whose elements are the collections
of pairs $(p_{i1},p_{i2})$ where 
$p_{i2} = p_i \setminus p_{i1}$).
Examples: 
$$
Q_{\{ \{ 1,2 \}\}} 
= \;  \; \bigcup_{r\in\Pas'(\{1,2\})} \; \{\{ r\} \}  
\; =\;  \{ \; \{  (\{1,2\},\emptyset) \},\; 
\{ (\{1\},\{2\}) \}, \;
\{  (\{2\},\{1\}) \} \; \}
\; = \;  
\left\{ \bu{1&2} , \; \bu{1\\2} , \; \bu{2\\1} 
\right\} 
$$
(hopefully a multi tableau visualisation is easier on the eye - here specifically the multi-tableaux consist only of single tableau);
and for a case requiring a proper multi-tableau:
$$
Q_{\{ \{ 1\},\{2 \}\}} 
  = \Pas'(\{1\}) \prod \Pas'(\{2\}) 
  = \{ \{ (\{1\},\emptyset), (\{2\},\emptyset) \} \} 
  =\left\{ \bu{1}\;\;\bu{2} \right\}   .
$$
Altogether then we have a set 
\[
\sS'_N = \bigcup_{p \in \PP(N)} Q_p
\]
{where $\PP(N)$ is the set of partitions of $\underline{N}$.}
It is convenient to draw nations and their counties - ordered two-part set partitions - as composition tableaux
(cf. e.g. \cite{Shubhankar18}). 
And collections of nations as collections of tableaux.
So for example
\[
\tabb 
\sS'_2 = Q_{\{ \{ 1,2 \}\}}  \bigcup Q_{\{ \{ 1\},\{2 \}\}}
= \left\{ \bu{1&2} , \; \bu{1\\2} , \; \bu{2\\1} 
\right\} 
\cup 
\left\{ \bu{1}\;\;\bu{2} \right\} 
= \left\{ \bu{1&2} , \; \bu{1\\2} , \; \bu{2\\1} 
, \;\; \bu{1}\;\;\bu{2} \right\} 
\]
(note the last set contains only a single composite element, drawn as a sequence of two nations,
although this drawn order has no significance).
For an example of an element in larger rank we have
{{ 
\beq \label{eq:9X3xx}
\ytableausetup{smalltableaux} 
\bu{6&8\\1}\;\;
\begin{ytableau} 9&X \\ 3  \end{ytableau} \;\;
\begin{ytableau} 2&4  \end{ytableau} \;\;
\begin{ytableau} 5 \\ 7  \end{ytableau} 
\;\;\; \in \sS'_{10} 
\eq
}}
This raises the question of how to order the nations in
drawing such a picture. This is unimportant here but useful later. Later we order with 
smaller 
nations first; and then larger first counties first.

\mdef \label{pa:SSSSN}
Finally 
\[
\SSSS_N = \bigcup_{z \in \sS'_N} \Pas ( z )
\]
Example:
\[
\SSSS_2 = \;\left\{\; \left( \,\bu{1&2},\right), \; \left(,\bu{1&2}\right), \;
\left(\bu{1\\2},\right),\; 
\left(,\bu{1\\2}\right),\; 
\left(\bu{2\\1},\right), \; 
\left(,\bu{2\\1}\right), \;
(\bu{1}\;\bu{2},), \;
(\bu{1},\bu{2}), \; (\bu{2},\bu{1}), \;
(,\bu{1}\;\bu{2}) \right\} 
\]

\mdef  \label{pa:shapes}
The symmetric group $\Sym_N$ acts on 
$\lambda \in \SSSS_N$ by
application to the individuals in the counties.

The {\it shape} of an element $\lambda \in \SSSS_N$ is the diagram obtained by   ignoring the entries in the boxes.
Note however that these combinatorial objects are multisets rather than sets. 
(As we will see, they are both beautiful and useful. 
We study them in  \S\ref{ss:J}.)

It will be evident that 
a transversal of the orbits of the $\Sym_N$ action is 
described by the set of shapes.
(In \S\ref{ss:recipe} we give a way to convert the set of 
shapes into a `standard' transversal in $\SSSS_N$.)

\mdef 
We fix the ground field $\C$. 
Associated to each $\lambda \in \SSSS_N$ there is a `type-space', call it $\C^\lambda$. 
There is a non-zero parameter for each county
- we require $\alpha_s + \beta_s \neq 0$ for the two
parameters in the same nation; 
and a pair of non-zero parameters for each pair of nations.  

Let $\SSSSC_N$ denote the set whose elements are pairs
$(\lambda,\ul{x})$ where $\lambda\in\SSSS_N$ and 
$\ul{x}$ is a point in $\C^\lambda$.

\subsection{Recipe constructing (all) representations}  \label{ss:recipe1}

\newcommand{\Cc}{C}  

Here we give a construction for varieties of cc
loop braid representations
({Theorem~\ref{th:main}} will show this is gives them all). 
The varieties are indexed by $\SSSS_N$, 
and $\SSSSC_N$
describes the parameter space of each variety.

\mdef 
It will be helpful to give names for the individual parameters in $\C^\lambda$. 
To this end we can order the nations (i.e. composition tableaux) in 
 $\lambda\in \SSSS_N$, 
$n_1, n_2, ..., n_m$, as follows ---
order first into the ordered pair; 
then within each component using the natural order on nation sizes;
and then at fixed size using the 
natural order on second-county sizes.
Finally the repeats of a given shape are ordered child-first
{(i.e. nation with lowest numbered resident first)}.

Now for each nation $n_s$ fix 
the two non-zero parameter names $\alpha_s,\beta_s$,
associated to the first (upper) and second county respectively.
(Recall we require in addition that $\alpha_s+\beta_s\neq 0$.)
For each pair of nations $n_s,n_t$ with $s<t$ fix 
the two non-zero parameter
names $\mu_{s,t}$ and $\Cc_{s,t}$.

\mdef  \label{de:RR}
Next we give a construction for each $N \in \N$
of a function 
\begin{align} 
\RR :  \SSSSC_N \;\;\;& \rightarrow \Match^N(2,2) \times \Match^N(2,2)
\\
 (\lambda,\xx) \;\; &\mapsto \hspace{.6541in} (S,R)
\end{align}
We will show that this gives all 
charge conserving loop braid representations.

\mdef  \label{lamex1}
For $(\lambda,\xx) \in \SSSSC_N$
we encode $R$ in $\RR(\lambda,\xx) $ by 
$\aalph(R) = (a_1,\ldots,a_N,A(1,2),\ldots, A(N-1,N))$ 
as in (\ref{eq:alphaform});
and $S$ by 
$\aalph(S) = (b_1,\ldots,b_N,B(1,2),\ldots, B(N-1,N))$.
To give a solution we give the scalars 
$a_1,\ldots,a_N$, $b_1,\ldots, b_N$ and matrices $A(i,j)$, $B(i,j)$ {for all $i< j$}.
 These depend on the relationship between the counties/nations that individuals $i<j$ reside in. 
Specifically:

Consider each individual $i$.
Let $n_s$ be the nation that $i$ resides in.
If $i$ is in 
the top county in  $n_s$  then 
$a_i=\alpha_s$. 
If $i$ is in 
the other county of $n_s$ then $a_i=\beta_s$.  
If $n_s$ comes from the first (resp. second) part 
of the pair $\lambda$ then the `sign' of $n_s$ 
is $\sgn(n_s) = +$  (resp. $-$).
If the sign is $+$
and $i$ is in the top county
then $b_i=1$; 
while if the sign of $n_s$ is $+$ and 
$i$ is in the second county then $b_i=-1$. 
If the sign of $n_s$ is $-$ then the cases 
are reversed.

Consider each pair of individuals $i<j$.
\\
(We proceed
as in Lem.\ref{lem:anof} but choosing the 
gauge/X-symmetry 
parameters $c_{s,t}=1$, adopting the principle that off-diagonal elements of $S$ are gauged to 1.)
\begin{enumerate}
    \item 
If $i\in n_s$ and $j \in n_t$ with $s\neq t$ then 
$A(i,j)=\begin{pmatrix} 0 & \mu_{s,t}/\Cc_{s,t}\\ \mu_{s,t}\Cc_{s,t} &0\end{pmatrix}$, 
and $B(i,j)=\begin{pmatrix} 0 & 1\\
1& 0
\end{pmatrix}$
--- here if $\Cc_{st}$ arises with $s>t$ 
{it is to be understood as $1/\Cc_{ts}$}.
\item 
If $i$ and $j$ are in the same nation $n_s$ but different counties 
with $i$ in the top county
then 
\footnote{Alternative version:
$A(i,j)=\begin{pmatrix} \alpha_s+\beta_s & \alpha_s\\ -\beta_s & 0\end{pmatrix}$ 
and $B(i,j)=\sgn{(n_s)}\begin{pmatrix} 0 & 1\\
1& 0 \end{pmatrix}$. 
{--- note this works but diverges from our gauge choice.}
}
$A(i,j)=\begin{pmatrix} \alpha_s+\beta_s & \sgn{(n_s)}\alpha_s\\ -\sgn{(n_s)}\beta_s & 0\end{pmatrix}$ 
and $B(i,j)=\begin{pmatrix} 0 & 1\\
1& 0
\end{pmatrix}$.
If $j$ is in the top county then 
$A(i,j)=
\begin{pmatrix} 0 & -\sgn{(n_s)}\beta_s
\\ \sgn{(n_s)}\alpha_s &   \alpha_s+\beta_s \end{pmatrix}$
and 
$B(i,j)=\begin{pmatrix} 0 & 1\\1& 0\end{pmatrix}$.
\item 
If $i$ and $j$ are both in the top, 
respectively bottom, county 
 in $n_s$  
then 
$A(i,j)=\begin{pmatrix} \alpha_s &0\\0&\alpha_s\end{pmatrix} $,
respectively
$A(i,j)=\begin{pmatrix} \beta_s &0\\0&\beta_s\end{pmatrix} $;
and $B(i,j)=\sgn{(n_s)}\begin{pmatrix} 1 & 0\\
0& 1
\end{pmatrix}$.
\end{enumerate}

\newcommand{\typo}[1]{\smat \alpha_{#1} &0\\0&\alpha_{#1}\stam}
\newcommand{\typa}[1]{\smat \alpha_{#1}+\beta_{#1} & \alpha_{#1}\\ -\beta_{#1} & 0 \stam}
\newcommand{\typax}[1]{\smat 0 & -\beta_{#1}\\ \alpha_{#1} & \alpha_{#1}+\beta_{#1} \stam}
\newcommand{\typay}[1]{\smat \alpha_{#1}+\beta_{#1} & -\alpha_{#1}\\ \beta_{#1} & 0 \stam}
\newcommand{\typs}[1]{\smat  0 & \mu_{#1}/C_{#1} \\ \mu_{#1} C_{#1} & 0 \stam}
\newcommand{\Styps}{
\smat 0 & 1\\ 1 & 0 \stam}

\mdef  \label{pa:123}    Example.
For $
\left( \bu{1&2\\3} ,\emptyset \right)$
we  have
\[
\aalph(R) =  \left( \alpha_1, \alpha_1, \beta_1, 
\typo{1} , 
\typa{1} , 
\typa{1} 
\right)
\]
\[
\aalph(S) =  \left( 1, 1, -1, 
\begin{pmatrix} 1 &0\\0&1\end{pmatrix},
\begin{pmatrix} 0 & 1\\ 1 & 0 \end{pmatrix},
\begin{pmatrix} 0 & 1\\ 1 &   0 \end{pmatrix}
\right)
\]
(RRR, SSS, RSS, RRS all hold and, as desired, 
SRR does not - these are somewhat large but routine calculations; we omit the details, but see \cite{N3aa0checknew.mw}).

\mdef  \label{exa:f23}
Example.
For $ 
\left( \bu{1&3\\2} ,\emptyset \right)$
we  have 
\[
\aalph(R) = \left( \alpha_1, \beta_1, \alpha_1, 
\typa{1},
\typo{1},
\typax{1} 
\right)
\]
\[
\aalph(S) =  \left( 1, -1, 1, 
\begin{pmatrix} 0 & 1\\ 1 & 0 \end{pmatrix},
\begin{pmatrix} 1 &0\\0&1\end{pmatrix},
\begin{pmatrix} 0 & 1\\ 1 &   0 \end{pmatrix}
\right)
\]
Observe that this is also 
(\ref{pa:123}) with $f^{(23)}$ with $(23)\in\Sym_3$ applied, cf. (\ref{eq:fsig}).
So again RRR, SSS, RSS, RRS all hold and, as desired, 
SRR does not, without further checking.

\mdef Non-Example.
For $ 
\left( \bu{1&3\\2} ,\emptyset \right)$
we {\em do not} have a solution taking 
\[
\aalph(R) =  (\alpha_1, \beta_1, \alpha_1, 
\typa{1},
\typo{1},
\typa{1}
)
\]
\[
\aalph(S) =  \left( 1, -1, 1, 
\begin{pmatrix} 0 & 1\\ 1 & 0 \end{pmatrix},
\begin{pmatrix} 1 &0\\0&1\end{pmatrix},
\begin{pmatrix} 0 & \pm1\\ \pm1 &   0 \end{pmatrix}
\right)
\]
(which are riffs on a first draft recipe, but both sign versions 
are 
checked as RRR:True; SSS:True; RSS:False!
in 
\cite{N3aa0checknew.mw}).

\mdef Example.
For $ \left(\emptyset, \bu{1&2\\3}  \right)$ we have:
\[
\aalph(R) =  \left( \alpha_1, \alpha_1, \beta_1, 
\typo{1},
\typay{1},
\typay{1} 
\right)
\]
\[
\aalph(S) =  \left( -1, -1, 1, 
\begin{pmatrix} -1 &0\\0&-1\end{pmatrix},
\begin{pmatrix} 0 & 1\\ 1 & 0 \end{pmatrix},
\begin{pmatrix} 0 & 1\\ 1 &   0 \end{pmatrix}
\right)
\]

\mdef Cautionary Example.
For $ \left(\emptyset, \bu{1&2\\3}  \right)$
we {\em do not} have
\[
\aalph(R) =  (\alpha_1, \alpha_1, \beta_1, 
\typo{1},
\typa{1},
\typa{1} 
)
\]
\[
\aalph(S) =  \left( -1, -1, 1, 
\begin{pmatrix} -1 &0\\0&-1\end{pmatrix},
\begin{pmatrix} 0 & 1\\ 1 & 0 \end{pmatrix},
\begin{pmatrix} 0 & 1\\ 1 &   0 \end{pmatrix}
\right)
\]
(RRR, SSS, RSS, all True; but RRS 
false with these signs - see \cite{N3aa0checknew.mw}. 
Of course it will work with 
all off-diagonals $-$ in $S$, since this is just $-S$ from 
(\ref{pa:123}), but this does not adhere to our gauge
choice).

\mdef  \label{exa:f13}
Example. For  
$\lambda= \left( \bu{1 \\ 2} , \;\bu{3} \right)$ we have 
\[
\aalph(R) =  (\alpha_1, \beta_1, \alpha_2, 
\typa{1},
\typs{12},
\typs{12} 
)
\]
\[
\aalph(S) =  \left( 1, -1, -1, 
\begin{pmatrix} 0 & 1\\ 1 & 0 \end{pmatrix},
\begin{pmatrix} 0 & 1\\ 1 & 0 \end{pmatrix},
\begin{pmatrix} 0 & 1\\ 1 & 0 \end{pmatrix}
\right)
\]
(RRR, SSS, RSS, RRS all True). 

\mdef For 
$ \left( \bu{3 \\ 2} , \;\bu{1} \right)$ we have: 
\[
\aalph(R) =  (\alpha_2, \beta_1, \alpha_1, 
\typs{21},
\typs{21},
\typax{1} 
)
\]
\[
\aalph(S) =  
\left( -1, -1, 1,  \Styps, \Styps, \Styps \right)
\]
so $\RR((13).\lambda)=f^{(13)}\RR(\lambda)$,
so this works.
$\;$ 
Meanwhile for 
$ \left( \bu{1 \\ 3} , \;\bu{2} \right)$ we have: 
\[
\aalph(R) =  (\alpha_2, \beta_1, \alpha_1, 
\typs{12},
\typa{1},
\typs{21} 
)
\]
\[
\aalph(S) =  
\left( 1, -1, -1,  \Styps, \Styps, \Styps \right)
\]
so $\RR((23).\lambda)=f^{(23)}\RR(\lambda)$.

\mdef\label{par:permcounties}
{ Note from the construction of $\SSSS_N$ that counties are unordered sets, thus permuting the indices within counties does not change the element of $\SSSS_N$. It is straightforward to see from the recipe that the map $\RR$ is well defined with respect to such a permutation.}

\subsection{Index sets for \texorpdfstring{$\Sym_N$}{SymN}-orbits of solution varieties} \label{ss:J}

{Here we give the construction of the set $J^{\pm}_N$ (the set of shapes from (\ref{pa:shapes})). There is an injection of $J^{\pm}_N$ into $\SSSS_N$, thus there are again  varieties of pairs $(S,R)\in \Match^N(2,2) \times \Match^N(2,2)$ associated to each member of $J^{\pm}_N$. We will show in {Theorem~\ref{th:main}} that, up to $\Sym_N$ symmetry, this subset is sufficient to give all charge conserving loop braid representations.}

We start by codifying the set of shapes. 

\mdef 
Recall that a multiset on a set $S$ is  a function 
$f:S \rightarrow \N_0$ (assigning a multiplicity to each element);
and that if $d:S\rightarrow \N$ is a degree function on
$S$, then $J_N(S)$ is the set of multisets $f$ that have
total degree $\sum_{s\in S} f(s) d(s) = N$.

A {\em signed} multiset of degree $N$ is equivalent to an
ordered pair
of multisets of total degree $N$
(the first multiset gives the multiplicities of objects signed +; and the second signed -).
Let us write 
$
J_{P,M}(S) = J_P(S) \times J_M(S)
$ 
for the set of pairs of multisets as indicated. 
So the set of pairs whose total degree is $N$ is 
$J_N^{\pm}(S) = \sqcup_P J_{P,N-P}(S)$.

\mdef   \label{pa:516}
As we will see, our 
indexing combinatorial objects 
at rank $N$ are 
signed multisets of compositions into at-most two parts
of total degree $N$.

Write $\Lambda^2$ for the set of  
at-most two-part compositions
--- equivalently the set 
\beq \label{eq:Lam2} 
\N \times \N_0 = 
\{ (1,0), \;
(2,0),(1,1), \;
(3,0),(2,1),(1,2),  \;
(4,0),(3,1),(2,2),(1,3), \; 
...  \}
\eq 
\[
= \{
\square ,\; \square\!\square ,  
\oneone,\; 
\square\!\square\!\square ,  \twoone , \onetwo , \;
\square\!\square\!\square\!\square , ...
\}
\]

Note that the number of elements of 
$\Lambda^2 \cong \N\times\N_0$ of degree $N$ is $N$;
and note the total order indicated by (\ref{eq:Lam2}). 
Here write just $J_N$ for $J_N(\N\times\N_0)$.
The ordinary multisets are 
$J_1 = \{ \one^1 \} $
;  $\;$ 
$$
J_2 = 
\{ \one^2 , \two^1 , \oneone^1 \} 
; \; 
\hspace{1in} 
J_3 
 =\; \{\;  \one^3 ,\;\; 
 \one^1 \two^1 , \; \one^1 \oneone^1 ,  \;\;
 \three^1 ,\; \twoone^1 , \onetwo^1 \; \}
$$ 
\[
 J_4 = \{ \one^4, \one^2\two^1, \one^2 \oneone^1,
 \one^1 \three^1, \one^1 \twoone^1, \one^1 \onetwo^1,
 \two^2, \two^1\;\oneone^{1}, \oneone^2 , 
\four^1 , \threeone^1, 
(2,2)^1, 
(1,3)^1   \} 
\]
and so on,
writing $\lambda^t$ to indicate for a given function that 
$f(\lambda )=t$, with other multiplicities 0.
 
\medskip

\mdef  \label{exa:J}
 For the signed versions, $J^\pm_1 = J_{1,0}\cup J_{0,1}$ where:
 $J_{1,0} = \{  
 (\one^1,) \}$ and 
 $J_{0,1}  = \{ 
 (,\one^1 ) \}$.
 Then $J_2^{\pm} =J_{2,0} \cup J_{1,1} \cup J_{0,2}$
 where 
 $$
 J_{2,0} = \{ 
 (\one^2 ,), (\two^1 , ) , (\oneone^1 ,) \},
 \hspace{.21in}
 J_{1,1} = \{ 
 (\one^1 , \one^1 ) \} , 
\hspace{.21in}
 J_{0,2} = \{ 
 (,\one^2) , (,\two^1), (,\oneone^1 ) \}
 $$
 
 Next 
 $
 J_3^{\pm} = J_{3,0} \cup J_{2,1} \cup J_{1,2} \cup J_{0,3}$ with 
 \[
 J_{3,0} = \{ 
 (\one^3,), (\one^1 \two^1,), (\one^1 \oneone^1,),
 (\three^1,), (\twoone^1,), (\onetwo^1,) \} , 
\]
\[
 J_{2,1} = \{ 
 (\one^2,\one^1), (\two^1,\one^1), (\oneone^1,\one^1) \}
 ,
 \hspace{1cm}
 J_{1,2} = \{ 
(\one^1,\one^2), (\one^1,\two^1), (\one^1,\oneone^1)
\} ,
\]
$$
J_{0,3} = \{ (,\one^3), (,\one^1 \two^1), (,\one^1 \oneone^1),
 (,\three^1), (,\twoone^1), (,\onetwo^1) \} .
 $$

Next $J_4^{\pm} = J_{4,0} \cup J_{3,1} \cup J_{2,2} \cup J_{1,3} \cup J_{0,4}$ with 
\[
J_{4,0} = 
\{ (\one^4,), (\one^2 \two^1,), (\one^2 \oneone^1,), 
(\one^1 \three^1,), (\one^1 \twoone^1,), 
(\one^1 \onetwo^1 ,) , (\two^2,), 
(\two^1 \oneone^1,),... \} 
\]
and so on.   

\subsection{Recipe constructing all representations
(up to symmetry)}  \label{ss:recipe}

{Here we give an injection of $J^{\pm}_N$ into $\SSSS_N$.}
 We thus give, via $\RR$, a construction for varieties of cc
loop braid representations up to $\Sym_N$ symmetry
({Theorem~\ref{th:main}} will show this is gives them all).

\mdef \label{pa:order}
Recall that a multiset 
$f:S \rightarrow \N_0$
can be considered as a set 
but where entries can be repeated, or indeed missing
(in this perception the repeats must somehow be given
individuality from each other,
for example if $f(s)=n$ then we have $n$ copies of $s$
individuated by $(s,1)$, $(s,2),\ldots,(s,n)$ perhaps).
Given an order on
the underlying set $S$ we can write   
the set form of $f$ 
as a sequence.
Individual terms that are repeated elements 
then receive `individuality' from their identical siblings 
by their order in the sequence.

\mdef Following on from (\ref{pa:order}), 
an ordered pair of multisets can then be seen  as 
an ordered pair of sequences 
(or, say, as a sequence where each 
same-type run is partitioned into two, 
but we will adopt the former organisation).

In this 
way $\lambda\in J^{\pm}_N$ gives an  
ordered set 
of nations (i.e. compositions) $n_1, \ldots, n_m$ ---
ordering first into the ordered pair $(+,-)$; 
then 
using 
the natural {ascending} order on nation sizes; 
and then at fixed size using 
the total order on $\N\times\N_0$ from 
(\ref{eq:Lam2}) {i.e., in ascending order by the size of the second county)}.
Finally the repeats of a given nation are of course
nominally indistinguishable, so simply ordered as written. 

\mdef Example: 
Suppose $f \in J_6$ is given by $f(\two)=2$ and $f(\oneone)=1$ and
others zero; and $g\in J_5$ by $g(\one)=3$ and 
$g(\two)=1$ and others zero. Then 
$\lambda=(f,g)\in J^{\pm}_{11}$ can be represented as 
\beq \label{eq:exa11}
\lambda = \; 
(\two \; \two \; \oneone, \; \one\;\one\;\one\;\two )
\eq 
Thus here, in this pair-of-sequences form, 
$n_1 = \two$, $n_2 = \two$, $n_3 = \oneone$,
$n_4 = \one$, $\ldots$, $n_7=\two$. 

\mdef \label{de:recJ} 
Given $\lambda\in J_N^{\pm}$ ordered as above we obtain a composition tableaux in $\SSSS_N$ by filling in the numbers in order, with the first $|n_1|$ numbers going into $n_1$ from left to right and then from top to bottom.  From the example above we see that (\ref{eq:exa11}) then becomes: 

\beq \label{eq:lamex}
\lambda \leadsto (\twoN{1}{2} \; \twoN{3}{4} \;
\oneoneN{5}{6} , 
\;\oneN{7} \; \oneN{8} \; \oneN{9} \; \twoNN{10}{11})
\eq 

This gives an injection of $J_N^{\pm}$ into $\SSSS_N$, so we may identify $\lambda\in J_N^{\pm}$ with the unique order preserving composition tableaux obtained as above, and abuse notation. Let $\JJJ_N$ denote the subset of $\SSSSC_N$ coming from the image of $J_N^{\pm}$, i.e., pairs $(\lambda,\xx)$
 with $\lambda\in J_N^\pm$.

 Now we may apply the recipe $\RR$ (\ref{de:RR}) to ${\JJJ_N}$ to obtain a solution.   We will see that, up to $X$-symmetry and $\Sigma_N$ symmetry, \emph{every} rank $N$ solution $(R,S)$ is obtained from an element of $\JJJ_N$.

\mdef See \ref{ss:AppC} for examples. 

\section{Main Theorem} \label{ss:MainTheorem}

\subsection{Prelude to the main Theorem: a key Lemma} \label{ss:N3s}

\mdef \mulem \label{lem:res2B}
Let $F: \{ \sigma, s \} \rightarrow \Match^N(2,2)$
be any pair as in (\ref{eq:Fss}). 
Then $F$ induces a functor 
$\F:\LL \rightarrow \Match^N$ 
if and only if 
every restriction to rank 2 is a functor
(i.e. belongs to the list in \ref{lem:anof});
and the restriction to $\Bcat$ is a functor
(i.e. every restriction to rank-3 belongs to the list in \ref{pr:9rule} up to symmetry - where we here use
Lemma~\ref{lem:restricto} to pass to rank-3).

\proof{
By Lemma~\ref{lem:restricto} 
(and Remark~\ref{rem:width 3}) 
it is enough to 
verify in case $N=3$. 
Thus it is enough to consider all extensions of the
set of functors $\F:\Bcat\rightarrow \Match^3$.
These are reproduced  in (\ref{pr:9rule}). 
A nominal superset of these is given by extending in all
nominally possible ways, 
according to \ref{lem:anof}, 
at each edge.
The complete enumeration of possibilities proceeds as 
in (\ref{pa:5rule}).
This finite set of groupings of cases is then verified by 
routine if lengthy
direct calculation (or see e.g. (\ref{pr:alternate})). 
\qed}

\newcommand{\proofff}{{\it Proof}. }

\mdef   \label{pr:alternate}
\proofff{(Alternate) 
We may verify explicitly by  
checking {that} all 
cubics in rank $3$ obtained from \eqref{eq:generalized braid}
with $z=Z$, and with $\zeta=Z$ (i.e. for \eqref{eq:sss}(II) and (III)) are satisfied for each matrix satisfying the assumptions of the Lemma.
It is assumed that each $\F$ is a representation of $\LL$ for every restriction to rank $2$, thus 
cubics containing 
a repeated index, e.g. $\vert iij\rangle$, are assumed to be satisfied, and we only need to check
equations in the permutation orbit of 
\eqref{eq:ijkbraidrel1}-\eqref{eq:ijkbraidrel5}.

\medskip

Recall that, as explained in \ref{par:KGraph}, rank $3$ solutions are represented by the complete graph $K_3$ with rank $2$ solutions attached to each edge.
We organise our sequence of checks by the number of restrictions of $S=F(s)$ to rank $2$ solutions that are $/$ edges.
The corresponding edges in $R=F(\sigma)$ are then of type $/$ or $\aaa$, with various conditions as in (\ref{pr:9rule}). We start with the case that no edges are $/$, then the case all edges are $/$. We then divide the remaining cases into the relations \eqref{eq:sss}(III) and \eqref{eq:sss}(III), which are each subdivided by the number of $/$ edges.
In each case we show that the permutation orbits of each of \eqref{eq:ijkbraidrel1}-\eqref{eq:ijkbraidrel5} are always satisfied.

Suppose first that all restrictions of $S$ and $R$ to rank $2$ solutions are in the zero varieties.
Then all off diagonal
entries are zero, it is thus immediate that all terms in the orbits of \eqref{eq:ijkbraidrel2}-\eqref{eq:ijkbraidrel5} go to zero.
For \eqref{eq:ijkbraidrel1}, notice that $\delta_{ij}D_{jk}d_{ij}=d_{jk}D_{ij}\delta_{jk}$, since, by considering the Lemma~\ref{lem:anof}, each matrix is a multiple of the identity matrix.

Now suppose all rank $2$ restrictions are of $/$ type, then all diagonal entries are $0$ and all equations trivialise.

\medskip

We now consider the relation \eqref{eq:sss}(III), so let $Z=\zeta$ in
\eqref{eq:ijkbraidrel1}-\eqref{eq:ijkbraidrel5}, in particular we use capital letters to refer to elements of $S$ and lower case to elements of $R$.
It is immediate that \eqref{eq:ijkbraidrel4} is trivial after this substitution. 
We then further replace lower case elements with greek letters to avoid confusion between this $c$ and $c$ denoting the gauge parameter in Lemma~\ref{lem:anof}.

Now suppose that $S$ restricts to two solutions in the $/$ variety, and one in the zero variety. 
By $\Sigma$ symmetry, we may assume that the solution in the zero variety lies on the $12$ edge.
First notice that $D_{ij}D_{jk}=0$ for all distinct $i,j,k$, since the $ij$ or the $jk$ edge is in the $/$ variety.
Thus, looking at \eqref{eq:ijkbraidrel1} we must have that, for all distinct $i,j,k,$ 
\[B_{ij}D_{ik}c_{ij}=b_{jk}D_{ik}C_{jk}.\]
If $ik$ is a $/$ edge, $D_{ik}=0$. Otherwise $ik$ is the $0$ edge, i.e. $ik\in \{1,2\}$, and in both cases the condition becomes $B_{13}\gamma_{13}=\gamma_{23}B_{23}$.
Let us now explain our notational convention: we use the parameter labels from Lemma~\ref{lem:anof} for each rank $2$ matrix, and, where necessary, add a subscript to indicate the edge.
Looking at (\ref{pr:9rule}), we have two cases.
The first is that the $13$ and $23$ edges in $R$ are $\aaa$-type, in which case all $A_i$ parameters from Lemma~\ref{lem:anof} are locked equal across the edges by (\ref{pr:9rule}), 
and the condition becomes $1/c_{13} (-c_{13}A_1) = (-c_{23}A_1)1/c_{23}$.
The second case is that the $13$ and $23$ edges in $R$ are $/$-type, 
in which case the $\mu$ and $C$ parameters from Lemma~\ref{lem:anof} are locked equal across the edges by (\ref{pr:9rule}), and the condition becomes $1/c_{13} (\mu C c_{13}) = (\mu Cc_{23})1/c_{23}$.

Using again that $D_{ij}D_{jk}=0$, \eqref{eq:ijkbraidrel2} becomes the condition $B_{ij}D_{ik}a_{ij}=d_{jk}B_{ij}D_{ik}$. If $ij$ is zero edge we are done, so suppose not.
This gives $D_{ik}a_{ij}=d_{jk}D_{ik}$. Both sides go to zero if $ik$ is a $/$, so suppose not. This means $i,k\in\{1,2\}$, and in both cases we have $a_{13}=a_{23}$, which is again true by observing that $/$ solutions are locked equal.

The arguments for \eqref{eq:ijkbraidrel3} and \eqref{eq:ijkbraidrel5} are similar.

\medskip

We now consider the relation \eqref{eq:sss}(II). We show that all cubics \eqref{eq:ijkbraidrel2}-\eqref{eq:ijkbraidrel5} are satisfied for the case $z=Z$, i.e. capital letters now refer to elements of $R$, and greek to elements of $S$.
Note we have already done the case that all restrictions of $S$ to rank $2$ solutions are in the zero variety.
Also \eqref{eq:ijkbraidrel5} becomes trivial.

Suppose $S$ restricts to two solutions in the $/$ variety, and one in the zero variety. As above, we may assume that this zero solution is on the $12$ edge.
We first prove that \eqref{eq:ijkbraidrel1} is always satisfied.

Suppose first that $ik$ is the zero edge, thus $\delta_{ij}=\delta_{jk}=0$.
Then \eqref{eq:ijkbraidrel1} becomes
$\beta_{ij}D_{ik}C_{ij}=B_{jk}D_{ik}\gamma_{jk}$.
 Since $ik$ is the zero edge, $D_{ik}\neq 0$, thus it is sufficient to observe that $\beta_{ij}C_{ij}=B_{jk}\gamma_{jk}$ becomes $\beta_{ij}C_{ij}=C_{kj}\beta_{kj}$ as $j>i,k$ and this is satisfied by the locking together of pairs of $/$ and $\aaa$ type solutions.

Now suppose $jk$ is the zero edge,
thus $\gamma_{jk}=\delta_{ij}=0$, and \eqref{eq:ijkbraidrel1} becomes 
$\beta_{ij}D_{ik}C_{ij}=D_{jk}D_{ij}\delta_{jk}$.
If restrictions of $R$ to the $ij$ and $ik$ solutions are in the $/$ variety, both sides become $0$.
For the $\aaa$ restrictions there are two cases, the first is $\gamma_{13}A_{23}B_{13}=D_{12}A_{13}\delta_{12}$, for which we have either $A_{23}=A_{13}=0$, or 
$c_{13} (A_1+A_2) A_1/c_{13} = A_1 (A_1+A_2) 1$.
These correspond to the two cases for the fourth triangle in (\ref{pa:5rule}). The other case is similar.

Finally for the case $ij$ is the zero edge, then $C_{ij}=0$, and, since $k=3$, $D_{jk}=D_{ik}=0$ and all terms in \eqref{eq:ijkbraidrel1} go to zero.
The arguments for \eqref{eq:ijkbraidrel2},\eqref{eq:ijkbraidrel3} and \eqref{eq:ijkbraidrel4} are similar.

Finally we consider the case that all edges of $S$ are of $/$ type, and $R$ has two edges of $/$ type and one of $a$ type. We will assume the $a$ type is the edge ${12}$.
Notice first that \eqref{eq:ijkbraidrel5} becomes trivial.

First consider \eqref{eq:ijkbraidrel1}. This becomes $\beta_{ij}D_{ik}C_{ij}=B_{jk}D_{ik}\gamma_{jk}$ 
Now either $D_{ik}$ is zero, or
$ik$ is the $\aaa$ edge, and we have 
$1/{c_{ij}}(A_1+A_2) (-c_{ij} \mu C) =\mu C c_{jk} (A_1+A_2) 1/c_{jk}$.
Equations \eqref{eq:ijkbraidrel2} and \eqref{eq:ijkbraidrel2} all go to zero by noting all terms contain either of diagonal elements of $S$, or pairs of diagonal elements from distinct edges of $R$.
Equation \eqref{eq:ijkbraidrel4} becomes $\beta_{ij}B_{ik}D_{jk}=D_{jk}B_{ij}\beta_{ik}$,
and we have either $D_{jk}=0$, or $jk$ is the $12$ edge and the condition becomes $\gamma_{ji}C_{ki}= C_{ji}\gamma_{ki}$ which follows from the locking together of pairs of $/$ edges in (\ref{pa:5rule}).
\qed}

\subsection{The main Theorem: statement and proof}

\newcommand{\AB}{A} 
\newcommand{\BA}{B} 

\begin{theorem}\label{th:main}
(\AB)
\\
(I) The construction $\RR(\lambda,\xx)$ gives a 
charge conserving
monoidal functor
$\FF : \LL \rightarrow \Match^N$ 
for every $ (\lambda,\xx) \in \SSSSC_N $. 
\\ 
(II) Every such
functor $\FF$ is in the orbit of some $\RR(\lambda,\xx)$ 
under the  X-symmetry (of  \ref{lem:X}).
\\
(\BA)
\\
(I) The construction $\RR(\lambda,\xx)$ gives a 
monoidal functor
$\FF : \LL \rightarrow \Match^N$ 
for every $ (\lambda,\xx) \in \JJJ_N $. 
\\
(II) Every such
functor is in the orbit of some $\RR(\lambda,\xx)$ 
under the $\Sigma_N$ and X-symmetries
(of \ref{lem:37III} and \ref{lem:X} respectively).
\end{theorem}

\proof
(A,B)
(I) Observe that the construction $\RR$ yields 
{\em by}{ construction} a solution on every $ij$
subspace 
(compare with Lemma~\ref{lem:anof});
and a braid representation
(compare with the recipe $\Rec$ in \cite{MR1X}). 
Thus it yields a solution by Lem.\ref{lem:res2B}. 
\\
(A,B) (II) See \S\ref{ss:BII} and \S\ref{ss:AII}.

\subsection{(\AB II) Proof }  \label{ss:BII}

\mdef \label{de:funcat}
As in \cite{MR1X} we write $\functor(\LL,\Match^N)$ for
the category of (monoidal) functors.
Let us write simply $(\LL,\Match^N)$ for the object set.
Consider an arbitrary such functor $\FF$, 
and hence the pair 
$S=\FF(s)$, $R=\FF(\sigma)$. 
We will determine the 
$(\lambda,\xx) \in  
\SSSSC_N$
such that $\RR(\lambda,\xx) \equiv \FF$ up to gauge choice. 
We do this below by interrogating $R,S$ for a suitable $\lambda$; and then further for a suitable $\xx$.
(This parallels the $\Bcat$ case, where we interrogate $R$ for a suitable $\lambda\in\SSS_N$ and so on.

\mdef  \label{pa:65}
Recall from 
(\ref{de:SSSN})
(or \cite[(4.4)]{MR1X}) that $\SSS_N$ is the set of all 
two-coloured multi-tableau 
on $N$ boxes.
Recall that $\Rec : \SSS_N \rightarrow \Match^N(2,2)$ is
the recipe constructing varieties of braid representations from $\SSS_N$.

Let
$R=\FF(\sigma)$
be any braid representation.
Observe from \cite[(4.7) and \S6.3]{MR1X} that 
the element of $\SSS_N$ associated to $R$ is given by
$(\partp(R), \partq(R),\partr(R), \parts(R))$
(all component functions as defined in \cite{MR1X}). 
Let us now write $\partt(R) $ for this element.
That is, $\Rec(\partt(R))$ is the variety containing 
$R$ up to gauge
(in this way it was shown that $\Rec$ constructs all 
representations). 

We will 
modify the 
pseudo-inverse function $\partt$ for our case. 

\mdef  \label{pa:538}
In our case 
firstly note that since, by Lem.\ref{lem:anof}(II), 
there are no $\pfii$ edges 
(still in the sense of restricting to the braid solution part)
the `colour' 
must differ between {\em every} pair of counties in each nation
of $\partt(R)$.
There are two colours, so this   
forces that there are at most two counties in 
each nation, always of different colour. 
- So, note, we do not need to further record colours. 
This $\partt(R)$ 
can be realised as 
a set of composition tableaux
--- tableaux where order in a row does not matter
(so we can use natural 
ascending order for free),
and order between rows does matter. 
For examples:
\beq \label{eq:9X3}
\ytableausetup{smalltableaux} 
\begin{ytableau} 9&X \\ 3  \end{ytableau} \;\;\;
\begin{ytableau} 6&8 \\ 1&2&4  \end{ytableau} \;\;\;
\begin{ytableau} 5 \\ 7  \end{ytableau} 
\eq
\[ \ytableausetup{smalltableaux} 
\begin{ytableau} 9&X \\ 3  \end{ytableau} \;\;\;
\begin{ytableau} 6&8 \\ 1  \end{ytableau} \;\;\;
\begin{ytableau} 5 \\ 2&4&7  \end{ytableau} 
\]
In such  
representations as these 
the nations also appear {\em arranged} as a sequence rather
than just a set, but so far this is arbitrary,
cf. (\ref{pa:order}).

\mdef 
Recall from (\ref{pa:SSSSN}) that $\SSSS_N$ denotes 
the set of partitions of 
such sets 
into an ordered pair of subsets, restricting to
the case of compositions with at most two parts
(two parts of different colours, in the original $\SSS_N$ idiom).

\mdef 
We define $\partrr(R,S)$ as the partition of 
$\partt(R)$ into an ordered pair of subsets, 
with nation $s$ in the first part if 
$\bra{ii} S \ket{ii} =+1$ for $i$ the individual in
the top left box of nation $s$.

\mdef
For example perhaps:
\beq \label{eq:9X3''}
\ytableausetup{smalltableaux} 
\partrr(R,S) =
\left( \; 
\begin{ytableau} 5 \\ 7  \end{ytableau} \;\;\;
\begin{ytableau} 9&X \\ 3  \end{ytableau} \;\;\;
, \;\;
\begin{ytableau} 6&8 \\ 1&2&4  \end{ytableau} \;\;\;
\right) 
\eq

For an example of large enough rank to be generic 
(but still neat to typeset!) we
use $\{a,b,c,...\}$ instead of $\{1,2,3,...\}$:
\[
\ytableausetup{smalltableaux} 
\left( \; 
\begin{ytableau} f  \end{ytableau} \;\;\;
\begin{ytableau} x  \end{ytableau} \;\;\;
\begin{ytableau} z  \end{ytableau} \;\;\;
\begin{ytableau} e&y \\ b  \end{ytableau} \;\;\;
\begin{ytableau} r \\ u&v  \end{ytableau} \;\;\;
\begin{ytableau} k&p \\ d&l  \end{ytableau} \;\;\;
, \; 
\begin{ytableau} s&t&w  \end{ytableau} \;\;\;
\begin{ytableau} c&i \\ j  \end{ytableau} \;\;\;
\begin{ytableau} n&o \\ h  \end{ytableau} \;\;\;
\begin{ytableau} q \\ a&g&m  \end{ytableau} \;\;\;
\right)
\]

\newcommand{\chix}{\ul{\chi}}

\mdef 
Having interrogated $R,S$ for a $\lambda\in\SSSS_N$, we now interrogate for an $\xx\in\C^\lambda$. 

Associated to each nation $n_s$, say, of 
$\alphaa = \partt(R)$ there are 
$\alpha_s$ and $\beta_s$ 
values that can be read off from $R$
---  inspecting $\bra{ii} R \ket{ii}$ gives
either $\alpha_s$ or $\beta_s$ for $i$'s nation,
depending on {whether $i$ is in the} first or second row.
Observe that $i,i'$ in the same row have a type-0 edge between them here, so the procedure is well-defined without specifying $i$ further.

Associated to each  
pair {$n_s,n_t$}  
of nations, with $s<t$, there is a 
$\mu_{st}$ parameter value, given by square root of product of
off-diagonals {in $R$} between any $i$ in $s$ and $j$ in $t$.
Observe again well-definedness. 

Continuing with this $n_s,n_t$,
we can also read off the value of $C_{st} c_{ij}$.
This does depend on $i,j$, but 
the value of $c_{ij}$ is determined from the $ij$ 
off-diagonals in $S$, i.e. $\bra{ij} S \ket{ji}$, 
so we can  determine $\Cc_{st}$.
Note that we can 
apply X-symmetry to $\FF$ 
to make $c_{ij}=1$ in all cases.
{To be   precise, 
in case
$i<j$, 
\begin{align}\label{eq:cijCst}
c_{ij}=
\bra{ii} S \ket{ii}
\bra{ji} S \ket{ij}, \;\;
C_{st}=\frac{\bra{ii} S \ket{ii} \bra{ji} R \ket{ij}}{c_{ij}\mu_{st}}
\end{align}
(recall that $\bra{ji} S \ket{ij}$, for example, is the bottom left entry in the $2\times 2$ matrix corresponding to the $12$ edge).
If instead $j< i$ then 
\begin{align}\label{eq:cijCstflip}
    c_{ij}=\frac{1}{
\bra{ii} S \ket{ii}\bra{ji} S \ket{ij}}, \; \; C_{st}=\frac{c_{ij}\mu_{st}}{\bra{ii} S \ket{ii}\bra{ji} R \ket{ij}}.
\end{align} }
{Observe that the operator spectrum depends on $C$ but is invariant under $C \leftrightarrow 1/C$.}

Write $\chix(R,S)$ for this collection of parameters.
Note that there is potential ambiguity in the organisation of the collection, resolved here by the total order on nations, as in (\ref{de:SSSN}).

\mdef
We claim that 
$\RR((\partrr(R,S),\chix(R,S)))$
is gauge-equivalent to $\FF$.
(N.B. This implies (\AB II).)

To see this first note that the restriction to $R$ is
correct up to gauge by (\ref{pa:65}). 
\\
A couple of examples will suffice to check that no
new gauge issues arise.
\\
1. Suppose $\partrr(R,S) =\left(  \bu{1},\bu{2\\3} \right)$. 
Necessarily then $\chix(R,S)$ gives values for two nations, thus $\alpha_1, \alpha_2, \beta_2, \mu_{12}, C_{12}, c_{12}, c_{13}, c_{23}$.
{For all pairs $i,j$ we are in case \eqref{eq:cijCst}, thus we}
 know from this and Lem.\ref{lem:anof} that $S$ takes the form
\[
\aalph(S) = (1,-1,1,\slush{12},\slush{13},-\slush{23})
\]
for some $c_{12}, c_{13}, c_{23}$,
with $\bra{33} S \ket{33}=1$ forced. 
Note that Lem.\ref{lem:anof} puts $-c_{23}$ here, since $\bra{22} S \ket{22}=-1$.
And given this $S$ then
$R$ must  be 
\[
\aalph(R) = \left( \alpha_1, \alpha_2, \beta_2,
\slesh{12},\sleesh{12}{13}, 
\smat \alpha_2+\beta_2 &\alpha_2 /c_{23} \\ -c_{23}\beta_2 & 0 \stam 
\right)
\]
It follows from  
(\ref{de:RR}) that the $\RR$ image takes the form 
\[
\hspace{-.25in} 
\aalph(F(s)) = \left( 1,-1,1,\smat 0&1\\1&0\stam,\smat 0&1\\1&0\stam,\smat 0&1\\1&0\stam \right),
\hspace{.51cm}
\aalph(F(\sigma))= \left( \alpha_1, \alpha_2, \beta_2,
\slosh{12},\slosh{12}, 
\smat \alpha_2 +\beta_2 &-\alpha_2 \\ \beta_2 &0 \stam 
\right)
\]
Observe that this is indeed 
in the X-orbit of $(R,S)$, by putting $c_{12}=c_{13}=1$, $c_{23}=-1$.
\\ 

\noindent 
2. Suppose $\partrr(R,S) =\left(  \bu{1&3},\bu{2} \right)$. 
Necessarily then $\chix(R,S)$ gives values for two nations, thus $\alpha_1, \alpha_2, \mu_{12}, C_{12}, c_{12}, c_{23}$. 
{Here we are in case \eqref{eq:cijCstflip} when $i=3, j=2$, and case \eqref{eq:cijCst} otherwise,} using this together with
Lem.\ref{lem:anof} we know that $S$ takes the form
\[
\aalph(S) = (1,-1,1,\slush{12},
\begin{pmatrix}
    1 &\\
    & 1
\end{pmatrix}
,-\begin{pmatrix}
    0 & c_{23}\\
    \frac{1}{c_{23}}& 0
\end{pmatrix})
\]
for some $c_{12}, c_{23}$,
with $\bra{33} S \ket{33}=1$ forced. 
Note that Lem.\ref{lem:anof} puts $-c_{23}$ here, since $\bra{22} S \ket{22}=-1$.
And given this $S$ then
$R$ must  be 
\[
\aalph(R) = \left( \alpha_1, \alpha_2, \alpha_1,
\slesh{12},
\begin{pmatrix}
    \alpha_1 &\\
    &\alpha_1
\end{pmatrix},
-\begin{pmatrix}
    0 & \mu_{12}c_{23}C_{12}\\
    \frac{\mu{12}}{c_{23}C_{12}} & 0
\end{pmatrix}
\right)
\]
It follows from  
(\ref{de:RR}) that the $\RR$ image takes the form
\[
\hspace{-.25in} 
\aalph(F(s)) = \left( 1,-1,1,
\smat 0&1\\1&0\stam,
\smat 1&0\\0&1\stam,
\smat 0&1\\1&0\stam \right),\]
\hspace{.51cm}\[
\aalph(F(\sigma))= \left( \alpha_1, \alpha_2, \alpha_1,
\slosh{12},
\begin{pmatrix}
    \alpha_1 &\\
    &\alpha_1
\end{pmatrix},
\begin{pmatrix}
    0 & C_{12}\mu_{12}\\
    \frac{\mu_{12}}{C_{12}}& 0
\end{pmatrix}
\right)
\]
Observe that this is indeed 
in the X-orbit of $(R,S)$, by putting $c_{12}=1$, $c_{23}=-1$.
\medskip 
\\ 
4. Suppose 
$\partrr(R,S) = (\square^3,\emptyset)$.
We know from Lem.\ref{lem:anof} that $S$ takes the form 
\[
\aalph(S) = (1,1,1,\slush{12},\slush{13},\slush{23})
\]
for some $c_{12}, c_{13}, c_{23}$. And given this then
$R$ must in particular be 
\[
\aalph(R) = \left( A_1, A_2, A_3,
\slesh{12},\slesh{13},\slesh{23} \right)
\]
for some values of the parameters.
In this case it follows from (\ref{de:RR}) that 
the $\RR$ image takes the form
\[
\hspace{-.25in} 
\aalph(F(s)) = \left( 1,1,1,\smat 0&1\\1&0\stam,\smat 0&1\\1&0\stam,\smat 0&1\\1&0\stam \right),\]
\hspace{.51cm}\[
\aalph(F(\sigma))= \left( A_1, A_2, A_3,
\slosh{12},\slosh{13},\slosh{23}  \right)
\]
This is in the X-orbit taking $c_{12}=c_{13}=c_{23}=1$.
\\
3. Suppose $\partrr(R,S) = \left(\bu{1&2} , \bu{3}\right)$.
The main new observation here is that $\bra{22} S \ket{22} =1$ by the nature of 0-edges.
So again the recipe agrees up to gauge.
\qed

\medskip \vspace{.1in} 

\subsection{(\BA II) Proof}   \label{ss:AII}

Let us write the action of $\Sym_N$ on 
$\lambda\in\SSSS_N$  
simply permuting the entries 
as $w.\lambda$.
For (\BA II) it is enough to show that the $f^w$ (with notation as in \ref{lem:37III}) and naive actions agree on $\SSSS_N$,
since $\JJJ_N$ is clearly a transversal of $\SSSS_N$
under the dot action.

\mdef   \label{lem:613}
We claim that
$ w.\partrr(\FF) = \partrr ( f^w \FF) $.
\\
(Note that the identity is one of varieties, not of 
individual solutions. We can see from the construction
that the size of the variety is not affected by $w$.) 

\proof  
Let us work through some specific kinds of cases.
We will be done when we have verified the action of 
(generating) elementary transpositions across the 
various contexts that arise. Several cases have
already been treated in (\ref{exa:f23}) {\it et seq}, so 
we continue from there. 
Note that the corresponding result connecting 
$\SSS_N$ and $\TTT_N$ is established in \cite{MR1X}.
\\
We must consider $f^{(ij)} \FF$ under the following scenarios:
\\
1. $i,j$ in the same county;
\\
2. $i,j$ in different counties in the same nation;
\\
3. $i,j$ in different nations with the same sign;
\\
4. $i,j$ in different nations with different sign.
\\
For case-1 \soutx{it will be clear that} both actions are trivial \ft(see also \ref{par:permcounties}).
\\
For case-2,  
consider applying $(3,X) \in \Sym_N$ to our $\FF$ in 
(\ref{eq:9X3''}).
{This flips the 3X edge, taking it from type $\mai$ to type $\aaa$
(cf. Example~\ref{exa:f23}).
}
{It also swaps the 1X and 13 edges (a 0 edge and an 
$\pai$ edge), and leaves all other edges and $S$ unchanged.}
{Looking at the construction of $\partrr$ we have that 
$\partrr(f^{(3X)}\F)$ is}
\[
\left( \; 
\begin{ytableau} 5 \\ 7  \end{ytableau} \;\;\;
\begin{ytableau} 9&3 \\ X  \end{ytableau} \;\;\;
, \;\;
\begin{ytableau} 6&8 \\ 1&2&4  \end{ytableau} \;\;\;
\right) 
\]
{which is immediately seen to be $(3X).\partrr(\F)$.\\} 
\\
For case-4,
suppose we apply $(1,9) \in\Sym_N$ to 
our $\FF$ in (\ref{eq:9X3''}),
i.e. $\FF \mapsto f^{(1,9)}\FF$ as in (\ref{lem:37III}).
Then
{we claim}
{$\partrr(\FF) \leadsto \partrr ( f^{(1,9)} \FF) $ is given by}
\[ 
\ytableausetup{smalltableaux} 
\left( \; 
\begin{ytableau} 5 \\ 7  \end{ytableau} \;\;\;
\begin{ytableau} 9&X \\ 3  \end{ytableau} \;\;\;
, \;
\begin{ytableau} 6&8 \\ 1&2&4  \end{ytableau} \;\;\;
\right) 
\;\;\leadsto\;\; \;\; 
\left(  \; 
\begin{ytableau} 5 \\ 7  \end{ytableau} \;\;\;
\begin{ytableau} 1&X \\ 3  \end{ytableau} \;\;\;
, \;
\begin{ytableau} 6&8 \\ 9&2&4  \end{ytableau} \;\;\;
\right) 
\]
(and cf. (\ref{exa:f13})).
This can be seen by noting the following.
In this new equivalent solution 
$f^{(1,9)}\FF $
vertex 1 will have the 
the $\alpha$ parameter value from the old 9X3 nation, {and vertex $9$ the $\beta$ parameter from the $68124$ nation} 
(so it 
is apt 
to think of the parameter
staying with the nation under exchange, even though 
two nations could entirely swap populations 
--- we can call
this Mach's other principle \cite{Mach-HawkingEllis73}).
Vertex 1 will also have `exchanged' its $S$ eigenvalue 
with vertex 9 
(although this is +1 to +1 so no change, here).
The 0-edge at 9X is now a 0-edge at 1X.
The /-edge at 69 is now a \off-edge at 61 (or 16, {note that the off diagonal entries will also flip}), and so on.
Similarly,
the /-edge at 1X is now a /-edge at 9X.
{The $\mai$ edge $93$ is now an $\aaa$ edge $13$, and similarly with $61$ to $69$.}
\\
Case-3 is similar. 
\qed 

\noindent 
This concludes the proof of the Theorem.

\section{Discussion and future directions} \label{ss:discus}

The input $J_N^\pm$ to our recipe yields some interesting 
combinatorial questions.

\mdef 
Note from (\ref{pa:516}) that 
the integer sequence for $|J_N|$ is 
the Euler transform of 1,2,3,4,...,
which begins 1,3,6,13,...
This is also MacMahon's sequence for plane partitions
(since this is the same Euler transform,
as noted in \cite[OEIS A000219]{oeis}). 

\mdef Exercise. Determine a bijection! 

\mdef  \label{pa:combover}
Let us determine the size of $J_N$
by another method.  
First note that there are exactly $k$ distinct compositions of $k$ into at most 2 parts.  To build the elements of $J_N$ we form a multiset of such compositions so that the sum of their sizes is $N$.  The generating function for the number of ordinary partitions is
$\sum_{i=0}^\infty p(n)x^n=\prod_{k\geq 1}\frac{1}{1-x^k}$: indeed, partitions 
can be regarded as a multiset of natural numbers $k\geq 1$.  The modification for our set up is that we are selecting multisets of compositions into two parts, and there are $k$ distinct compositions of each size $k$. Thus each factor should be taken $k$ times, giving $$\sum_{n=0}^\infty |J_n|x^n=\prod_{k\geq 1}\frac{1}{(1-x^k)^k}=1+x+3x^2+6x^3+13x^4+\cdots
$$
{This indeed coincides with MacMahon's GF for the
sequence of sizes of sets of plane partitions.
It is intriguing that the proof here is obtained 
very neatly from the very classical 
(due to Euler)
ordinary-partitions  case. The proof for plane partitions appears far more involved --- see e.g. \cite{stanley1999enumerative}.}
 
\mdef 
The sequence $|J_N^{\pm}|$ for $N=0,1,\ldots$  is (1), 2, 7, 18, 47, 110, ....
This is also the count for the set of ordered pairs
of plane partitions of total degree $N$
(see e.g. \cite[A161870]{oeis}).
Of course if a sequence has generating function $A(z)$
(in our case MacMahon's function)
then the sequence for ordered pairs 
has generating function $A^2(z)$. 

\medskip

It is an important question as to the measure of the
set of charge-conserving representations in the set of 
all representations
up to a suitable notion of isomorphism
(i.e. how restrictive were the constraints we put on in order to get a complete solution?). 
One way to think about this is by analogy with the 
corresponding problem for Hecke representations. 
To this end it is instructive to 
compare with the machinery of classical Hecke representation theory
as in \cite{Martin92}.
There we see, by a method involving Bruhat orders, 
that cc representations are, in a suitable sense, eventually faithful
(note in particular \cite[\S2.3(7)]{Martin92}). 
This raises several further representation-theoretic questions, such as the following. 
\\
What is the smallest (monoidal) subcategory of $\Mat$ that contains all perm matrices?
What is the smallest that contains all perm matrices
as ``additive'' components (subblocks)?
Is there a fusion trick?
What about monomial matrices?

\medskip 

This work can be seen as part of a more general programme on Representation theory of natural categories.
Here one starts by observing the naturalness/singular-ness of natural categories in the wider realm of monoidal categories (cf., for example, \cite{JOYAL199155}), as way of unifying groups and algebras {\it a la} statistical  mechanics (cf. for example \cite{Martin08a}); and observes key rigidification aspects. 
Many questions arise here, for example around generalisations and around applications (in particular to topological quantum computation - cf. for example \cite{baez_schreiber,porter_turaev,martins_picken,wang2015non}).

\medskip

\renewcommand{\theHsection}{A\arabic{section}} 
\appendix

\section*{Appendices} 

\section{Background for the Proof of Lemma~\ref{lem:res2B}}

\newcommand{\signa}[1]{signature: \; #1}
\newcommand{\SymXZ}{\Z_2 \times \Sigma}
\renewcommand{\ooo}[1]{}  

The simplicially-directed $K_3$ graph has edge orientations $1\rightarrow 2$, 
$2\rightarrow 3$, $1\rightarrow 3$.
The following 
decorated simplicially-directed $K_3$
diagrams should be understood as 
in \cite{MR1X},
giving $\aalph(R) = (a_1,a_2,a_3,A(1,2),A(1,3), A(2,3))$
with 
vertex parameters
$a_1,a_2,a_3$ given by vertex labels, and for example, $A(1,2)$ given by the label on the edge from vertex $1$ to vertex $2$.

\mdef \label{pr:N3} 
{\bf Proposition}. \cite[Prop.5.1]{MR1X} 
\label{pr:9rule}
{\it For $N=3$ the following types of
configurations yield a charge-conserving 
functor  $\FF : \Bcat \rightarrow \Match^3$
(showing one {variety} per $\SymXZ_3$ orbit):}
\[
\ber{2}{3}{\;\;\;/_{\mu} }{\;\; /_{\nu} }{\;/_{\lambda} \;}
\hspace{.21in}  
\bermss{1}{2}{1}{\;/_{\mu}\;}{\;/_{\mu}\;}{\uppfii_\beta} \hspace{.1in}
\hspace{.1241in} 
\bermss{1}{2}{3}{\;/_{\mu}\;}{\;/_{\mu}\;}{\uppai}
\hspace{.21in}
\bermss{1}{2}{1}{\;/_{\mu}\;}{\;/_{\mu}\;}{0} 
\]  \vspace{.05in}
\[  \hspace{.014in}
\bermpp{1}{1}{1}{\uppfii_\beta \;}{\;\uppfii_\beta \;}{\;\uppfii_\beta \;}
\hspace{.24in}
\bermpp{1}{2}{1}{\uppai}{\uppai}{\;\uppfii_{a_2} \;}
\hspace{.24in}
\bermpp{1}{2}{2}{\uppai}{\uppfii_{a_1}}{\; \uppai\;}  
\]  \vspace{.041in}
\[
\bermxx{1}{2}{2}{\;\pai\;}{0}{\pai}{\abcac}
\hspace{.21in}
\bermxx{1}{1}{1}{\pfii_\beta \;}{0}{\pfii_\beta}{\abcac}
\hspace{.21in}
\ber{1}{1}{0}{0}{0}   
\]

\medskip 

\mdef For example, the $\pai \pai 0$ case shown is 
$\bu{1 \\ *(yellow) 2& *(yellow) 3}$  (second row coloured). In this orbit we have also 
$\bu{2 \\ *(yellow) 1& *(yellow) 3}$
and
$\bu{3 \\ *(yellow) 1& *(yellow) 2}$,
collectively labelled $\onetwox$.
On the other hand, in the $\twoonex$ orbit (not explicitly represented above) we have
$\bu{1&2 \\ *(yellow) 3}
\stackrel{f^{(23)}}{\rightarrow}
\bu{1&3 \\ *(yellow) 2}
\stackrel{f^{(12)}}{\rightarrow}
\bu{2&3 \\ *(yellow) 1}
\;$ 
which unpacks to 
$$
\bermxx{1}{1}{2}{0}{\;\pai\;}{\;\pai\;}{\abcbc}
\stackrel{f^{(23)}}{\longrightarrow}
\bermxx{1}{2}{1}{\;\pai\;}{\;\mai\;}{0}{\abcab}
\stackrel{f^{(12)}}{\longrightarrow}
\bermxx{1}{2}{2}{\;\mai\;}{0}{\;\mai\;}{\abcac}
$$
- note that after the first move the 23 edge decoration is reversed, but the 
parameters on this edge are now tied to those on edge-12.

\medskip 

\mdef  \label{pa:5rule}  
Eliminating from (\ref{pr:N3}) the cases with $\fff$, 
which by Lem.\ref{lem:anof} do not 
extend, we have:
\[ 
\ber{2}{3}{\;\;\;/_{\mu} }{\;\; /_{\nu} }{\;/_{\lambda} \;}
\hspace{.21in} 
\bermss{1}{2}{3}{\;/_{\mu}\;}{\;/_{\mu}\;}{\uppai}
\hspace{.21in}
\bermss{1}{2}{1}{\;/_{\mu}\;}{\;/_{\mu}\;}{0} 
\hspace{.2in} \]\[
\bermxx{1}{2}{2}{\;\pai\;}{0}{\pai}{\abcac}
\hspace{.21in}
\ber{1}{1}{0}{0}{0}   
\]
with corresponding indices:
\[
\one^3 \hspace{0.8972in} \one\;\oneonex \hspace{0.8972in} 
\one\;\two \hspace{0.8972in} \twoonex  \; or\;\onetwox
\hspace{0.8972in}
\three 
\]
Note that we do not know {\em ab initio} that any 
of these extend.
The formally possible extensions in each case 
(according to the application of 
Lem.\ref{lem:anof} to each $K_2$ subgraph) are
represented by giving the corresponding decorations for 
$F(s)$. We write $+$ and $-$ for the diagonal eigenvalues $+1, -1$. We write $/$ for the edge decoration 
$\smat 0&1\\1&0\stam$
(taking advantage of the gauge freedom to fix this), 
but $/_{\aaa}$ to indicate that
the vertex eigenvalues are forced opposite.
For the 0 edges the vertex eigenvalues are forced the same. 
Note in particular that the diagonal eigenvalues 
within a nation are all determined by $a_1$.
Indeed they would be over-determined if there are more
than two $\aaa$ edges --- more than two counties, 
but this is already impossible.
On the other hand the relative eigenvalues between 
nations are not formally determined. Thus the possibilities overall correspond to a choice of sign
for (the lead $a_i$ in) each nation.
We have: 
\[
\berr{+}{+}{\pm}{\;/_{} }{\; /_{} }{\;/_{} \;}
\hspace{.21in} 
\berr{+}{+}{-}{\;/_{}\;}{\;/_{}\;}{/_{\aaa}}
\hspace{.21in}
\berr{+}{+}{+}{\;/_{}\;}{\;/_{}\;}{0} 
\hspace{.2in} 
\berr{+}{-}{-}{\;/_{\aaa}\;}{0}{/_{\aaa}}
\]
\[
\berr{+}{+}{+}{0}{0}{0}   \hspace{.21in}
\berr{+}{-}{\pm}{\;/_{} }{\; /_{} }{\;/_{} \;}
\hspace{.21in} 
\berr{+}{-}{-}{\;/_{}\;}{\;/_{}\;}{/_{\aaa}}
\hspace{.21in}
\berr{+}{-}{+}{\;/_{}\;}{\;/_{}\;}{0} 
\hspace{.2in}  
\]
together with the overall $s \leadsto -s$ flips.
For completeness (since we have not introduced the flip move here) we should include here, say, 
\[
\berr{+}{-}{-}{\;/_{\mai}\;}{0}{/_{\mai}}
\]

We find by direct calculation that all satisfy the $N=3$
constraints. 

\mdef 
Overall from (\ref{pa:5rule}), we have the following signed-multiset indices up to flip and symmetry:
\[
(\one^3,) \hspace{0.8972in} (\one\;\oneonex ,) \hspace{0.8972in} 
(\one\;\two,) \hspace{0.8972in} (\twoonex,) \hspace{0.8972in}
(\three,)
\] \[
(\one^2, \one) \hspace{0.8972in} (\oneonex, \one) \hspace{0.8972in}
(\two,\one) \hspace{0.8972in} \hspace{2in} 
\]

\medskip

\section{Ket calculus: the Texas braid relation}
\label{ss:ket}

Here we give a proof of  Lemma \ref{lem:X} by direct calculation.
In fact, we prove a slightly more general statement:

\mdef \mulem \label{lem:zZz}
Suppose $ z,Z,\zeta \in \Match^3(2,2)$ satisfy 
\beq \label{eq:zZz}  
z_1Z_2\zeta_1=\zeta_2Z_1z_2
\eq  
in $\Match^3(3,3)$
(recall $z_1 = z \otimes \IImat$ and so on).
For invertible $X\in\Match^3(2,2) $ let $z^X = XzX^{-1}$ and so on.
If $X$ is diagonal then 
$$z^X_1 Z^X_2 \zeta^X_1 = \zeta^X_2 Z^X_1 z^X_2$$.

\proof{ 
Consider the equation \begin{equation}
    \label{eq:generalized braid}
z_1Z_2\zeta_1=\zeta_2Z_1z_2
\end{equation} where 
$ z,Z,\zeta \in \Match^3(2,2)$ 
with 
nonzero entries lower case Roman ($a_{ij},b_{ij},c_{ij},d_{ij}$), upper case Roman ($A_{ij},B_{ij},C_{ij},D_{ij}$) and lower case Greek ($\alpha_{ij},\beta_{ij},\gamma_{ij},\delta_{ij}$), respectively, where $i\leq j$.  We have (for $z$, with analogous conventions for $Z,\zeta$):

$$z\ket{ij}=\begin{cases} d_{ij}\ket{ij}+b_{ij}\ket{ji} & i<j\\
a_{ji}\ket{ij}+c_{ji}\ket{ji} & j<i\\
a_{ii}\ket{ii} & i=j
\end{cases}$$

We will show that if $z,Z$ and $\zeta$ satisfy equation (\ref{eq:generalized braid})
then so do $z^X,Z^X$ and $\zeta^X$.

First note that conjugating each of $z,Z$ and $\zeta$ by a diagonal matrix leaves the diagonal entries $(a_{ij},A_{ij},\alpha_{ij})$ and $(d_{ij},D_{ij},\delta_{ij})$  invariant, while the effect on the other entries is: 
$\beth_{ij}\mapsto m_{ij}\beth_{ij}$ for $\beth\in\{b,B,\beta\}$ and $\daleth_{ij}\mapsto \daleth_{ij}/m_{ij}$ for $\daleth\in\{c,C,\gamma\}$ for some $m_{ij}$.  To see that this leaves the polynomial equations unchanged we compute a few entries.  First it should be clear that the $|iii\rangle$ entries give trivial conditions.  
{
For the $|ijk\rangle$ entries with $i,j,k$ distinct it is enough to consider $|123\rangle$ by a local permutation argument: one simply applies the permutation to all indices and then defines, for $i<j$, $a_{ji}:=d_{ij}$, $d_{ji}:=a_{ij}$, $b_{ji}:=c_{ij}$ and $c_{ji}:=b_{ij}$ and similarly for the $A,B,C,D$ and $\alpha,\beta,\gamma,\delta$.}

We have:
\begin{eqnarray*}
&&z_1Z_2\zeta_1|123\rangle=
\\
&&z_1Z_2 (\beta_{12}|213\rangle + \delta_{12}|123\rangle)=\\
&& z_1(\beta_{12}(B_{13}|231\rangle +D_{13}|213\rangle) + \delta_{12}(B_{23}|132\rangle + D_{23}|123\rangle))=\\
&& \beta_{12}(B_{13}(b_{23}|321\rangle + d_{23}|231\rangle) + D_{13} (a_{12}|213\rangle +c_{12}|123\rangle))+ \\
&&
\delta_{12}(B_{23}(b_{13}|312\rangle +d_{13}|132\rangle) + D_{23}(b_{12}|213\rangle +d_{12}|123\rangle))\\
\end{eqnarray*} 
whereas:
\begin{eqnarray*}
  &&  \zeta_2Z_1z_2|123\rangle=\\
  && \zeta_2Z_1 (b_{23}|132\rangle +d_{23}|123\rangle) =\\
  && \zeta_2 (b_{23}(B_{13}|312\rangle + D_{13}|132\rangle) +d_{23}(B_{12}|213\rangle + D_{12} |123\rangle))=\\
  && b_{23}(B_{13}(\beta_{12}|321\rangle +\delta_{12}|312\rangle ) + D_{13}({\alpha_{23}|132\rangle + \gamma_{23}|123\rangle }))+\\
  &&d_{23}(B_{12}(\beta_{13}|231\rangle +\delta_{13} |213\rangle ) + D_{12}(\beta_{23}|132\rangle +\delta_{23}|123\rangle))
\end{eqnarray*}
This yields 6 equations the last of which is trivial (denoted by (*)):

\begin{eqnarray}
    &&\beta_{12}D_{13}c_{12} + \delta_{12}D_{23}d_{12} = b_{23}D_{13}\gamma_{23} +d_{23}D_{12}\delta_{23} \label{eq:ijkbraidrel1}\\
    && \beta_{12}D_{13}a_{12} + \delta_{12}D_{23}b_{12} = d_{23}B_{12}\delta_{13} \label{eq:ijkbraidrel2}\\&& \delta_{12}B_{23}d_{13} = b_{23}D_{13}\alpha_{23} +d_{23}D_{12}\beta_{23}
    \label{eq:ijkbraidrel3}
    \\&&
    \beta_{12}B_{13}d_{23} = d_{23}B_{12}\beta_{13} \label{eq:ijkbraidrel4}\\&&\delta_{12}B_{23}b_{13} = b_{23}B_{13}\delta_{12} \label{eq:ijkbraidrel5}\\&& \nonumber
    \beta_{12}B_{13}b_{23} = b_{23}B_{13}\beta_{12} \quad (*)
\end{eqnarray}

 Next we compute the polynomial consequences of the generic braid relation on $|112\rangle$.

 Here we have: 

 \begin{eqnarray*}
    && z_1Z_2\zeta_1\ket{112}=\\
     &&\alpha_{11}(D_{12}a_{11}\ket{112}+B_{12}(d_{12}\ket{121}+b_{12}\ket{211}))
 \end{eqnarray*}
whereas 
\begin{eqnarray*}
    &&\zeta_2Z_1z_2\ket{112}=\\
    && d_{12}A_{11}(\delta_{12}\ket{112}+\beta_{12}\ket{121})+\\
    && b_{12}(D_{12}(\alpha_{12}\ket{121}+\gamma_{12}\ket{112})+B_{12}\alpha_{11}\ket{211})
    \end{eqnarray*}

These yield $3$ equations, that last of which is trivial:

\begin{eqnarray}
    &&\alpha_{11}D_{12}a_{11}=d_{12}A_{11}\delta_{12}+b_{12}D_{12}\gamma_{12}\\
    &&\alpha_{11}B_{12}d_{12}=d_{12}A_{11}\beta_{12}+b_{12}D_{12}\alpha_{12}\\
    &&\delta_{11}B_{12}b_{12}=
    b_{12}B_{12}\delta_{11}\quad (*)\nonumber
\end{eqnarray} 

Similar computations equating $z_1Z_2\zeta_1$ and $\zeta_2Z_1z_2$ on $\ket{121}$ and $\ket{211}$ yield:

\begin{eqnarray}
    &&\delta_{12}A_{12}a_{12}
    +\beta_{12}D_{12}c_{12}=a_{12}D_{12}\alpha_{12}+c_{12}A_{11}\beta_{12}\\
    && \alpha_{12}C_{12}a_{11}=a_{12}D_{12}\gamma_{12}+c_{12}A_{11}\delta_{12}\\
    &&\delta_{12}A_{12}b_{12}+\beta_{12}D_{12}a_{12}=a_{12}B_{12}\alpha_{11}\\
    &&\alpha_{12}A_{11}a_{12}+\gamma_{12}A_{12}b_{12}=a_{11}\alpha_{11}A_{12}\\
    &&\alpha_{12}A_{11}c_{12}+\gamma_{12}A_{12}a_{12}=a_{11}C_{12}\delta_{12}\\
    &&\gamma_{12}C_{12}a_{11}=a_{11}C_{12}\gamma_{12} \quad (*)\nonumber
\end{eqnarray}

{Again, we omit the calculations for $\ket{122},\ket{212}$ and $\ket{221}$ as the corresponding polynomial equations can be obtained from the above by applying the transposition $(1 \; 2)$ and defining $b_{21}:=c_{12}$, $a_{21}=d_{12}$ etc. }

Now we observe that
in each equation  
either 
\begin{enumerate}
    \item Each term has the same number of $b_{ij},B_{ij}$ or $\beta_{ij}$ (resp. $c_{ij},C_{ij}$ or $\gamma_{ij}$) terms so that the equation obtained by simultaneous $X$-conjugation has a negligible overall non-zero factor, or
    \item Each term with a $\{b_{ij},B_{ij},\beta_{ij}\}$ term also has a corresponding $\{c_{ij},C_{ij},\gamma_{ij}\}$ term (not necessarily matching) so that the effect of $X$-conjugation is nugatory.
\end{enumerate}
This completes the proof. \qed
}

Lemma \ref{lem:X} now follows by substituting $R,S$ for $z,Z,\zeta$ as appropriate.

\mdef {\em Remark}. The superficially {\em braid}-like equation (\ref{eq:zZz}) seems too general to have such a geometric-topological meaning. Some of its specialisations can be endowed with such meaning, as we can see from our context. But not all. Nonetheless, in $\Match^3$ it is useful. 

\medskip

\section{The loop braid category \texorpdfstring{$\LL$}{L}: technical asides}  \label{ss:AppA}

\newcommand{\sig}{\varsigma}
\newcommand{\rhoo}{\varrho}

\newcommand{\intt}[1]{\mathrm{int}(#1)}
\newcommand{\Motion}{\mathrm{Mot}}
\newcommand{\Mot}[2]{\Motion_{\underline{#1}}}
\newcommand{\Motempty}[1]{\Motion_{#1}}
\newcommand{\LLex}{\LL^{\mathrm{ext}}}
\newcommand{\Aut}{\mathrm{Aut}}
\newcommand{\TOPO}{\mathrm{TOP}}
\newcommand{\III}{\mathbb{I}}
\newcommand{\id}{\mathrm{id}}

Here we aim to provide,
in support of  Sec.\ref{ss:lab},
more of a rough substitute for the braid laboratory of \cite{MR1X} 
-  by working with an appropriate motion groupoid; and hence to give a bit more of a {\it feel} for the relations in the presentation (\ref{de:L'}). 
And to prove \ref{lem:loopbraidiso}, on non-monoidal presentation.

\subsection{Topological Background}
In what follows we denote the unit interval $[0,1]\subset \R$ by $\III$. 
We denote the unit ball  
$D^3=\{x\in \R^3 \, \mid \, \lvert x\rvert \leq 1 \}$ and the unit circle
$S^1=\{x\in \R^2 \, \mid \, \lvert x \rvert =1 \}$.

\mdef \label{pa:mot}
Let $M$ be a manifold, 
with boundary $\partial M$, 
and $\underline{M}=(M,\partial M)$.
Let $\TOPO^h(M,M)$ denote the set of self-homeomorphisms of $M$, 
made a space with the compact-open topology.

A {\it flow} of $\underline{M}$ is a path $f\colon \III \to \TOPO^h(M,M)$ with $f_0=\id_M$ and all $f_t$ pointwise fixing $\partial M$; or equivalently an automorphism $f'\colon M\times \III\to M\times \III$
such that
for each $t\in \III$, $f'$ restricts to a homeomorphism $f'_t\colon M\times \{t\} \to M\times \{t\}$ fixing $\partial M\times \{t\}$, and such that $f'_0$ is the identity.

A {\it motion} of $\underline{M}$
taking
$N$ to $N'$ is a triple $(f,N,N')$ consisting of a flow $f$ of $\underline{M}$, a subset $N\subset M$, and the image  $N'=f_1(N)$.

\newcommand{\from}{taking} 

Two motions in $\underline{M}$, $(f,N,N'), (g,N,N')$ are
{\it motion
equivalent} if $f$ can be continuously deformed into $g$ through motions \from\ $N$ to $N'$. 

For each $\underline{M}=(M,\partial M)$, 
there is a groupoid with objects the power set of $M$ and morphisms from $N\subset M$ to $N'\subset M$, motions up to motion equivalence. 
Following \cite{TMM} (where all details can be found), 
we denote this groupoid $\Mot{M}{A}$. 

\mdef 
For some manifold $M$
and set $Q$ of subsets of $M$, we use $\Mot{M}{\partial M}|_{Q}$ 
to denote the full subgroupoid of $\Mot{M}{\partial M}$
with objects $Q$. 

For example, if $p_n \subset (0,1)^2$ is a collection of $n$ points, for each $n \in \N$;
and $p_*$ is the set of all the $p_n$s, then
$\Mot{\III^2}{}(p_n,p_n)$ is a realisation of a braid group, and  $\Mot{\III^2}{}|_{p_*}$ is a realisation of the braid category.
This realisation bridges between the purely presentational one useful in representation theory, and the braid laboratory. Of course the connection to the braid laboratory makes several distinct uses of the continuum hypothesis.
What we do next will make several more!

\medskip 

\newcommand{\rr}{\mathfrak{r}}

{Fix a small number $\rr >0$.}
For each $n\in \N$, we fix an embedding $e_{n}\colon (S^1 \sqcup \ldots \sqcup S^1)\to \intt{D^3}$ of $n$ copies of $S^1$
into the interior of $D^3$
such that $l_n=e_n(S^1 \sqcup \ldots \sqcup S^1)$ is a configuration of unlinked circles of 
radius 
$\rr/n$ 
in the $xy$-plane with the $i$-th loop centred at 
$(\frac{i}{n}-\frac{1}{2n},0,0)$.
(Up to isomorphism it will not matter precisely
which configuration we take.
{The present choice is different from the one in the body of the paper, simply because we prefer not to have to unpick here the technical details of the limit required to have both the correct motion group on the nose, and the natural monoidal structure.}) 

\newcommand{\lS}{l_*}

Let $\lS$ 
be the set of all such $l_n$.
Note that 
$\Mot{D^3}{\partial D^3}|_{\lS}
\cong 
\sqcup_{l_n} 
\Mot{D^3}{\partial D^3}(l_n,l_n)$, 
since morphisms in  
$\Mot{D^3}{\partial D^3}$ 
exist only between homeomorphic subsets of $D^3$, and, by construction, $\lS$ contains precisely one element in each isomorphism class.

{The reader will readily confirm that for each pair $n,m \in \N$ there is a function taking $D^3 \sqcup D^3 $ to $D^3$ that takes the corresponding $l_n \sqcup l_m$ to $l_{n+m}$ in the spirit of Fig.\ref{fig:monoidal} (up to some rescalings),
and hence in the spirit of the braid laboratory of \cite{MR1X}.
It follows from (\ref{pa:mot}) that this 
collection of functions 
induces a (natural) monoidal structure on  
$ \sqcup_{n\in\N} \Mot{D^3}{\partial D^3}(l_n,l_n) $, 
and hence on $ \Mot{D^3}{\partial D^3}|_{\lS}$ 
(rigorous details will appear in \cite{MMT}).
As a brief abuse of notation we will use the same symbols for the monoidal categories.}

Let  
\[
\LLex \; =\; \Mot{D^3}{}|_{\lS}  \;\cong \; 
\sqcup_{n \in \N}   
\Mot{D^3}{}(l_n,l_n).
\]

\newcommand{\ofrog}{\varsigma}
\newcommand{\oswap}{\vars}
\newcommand{\oflip}{\uptau}

\mdef \label{de:mots} 
Next we consider some morphisms in $\LLex$. 
We have in mind motions inducing the circle trajectories 
of type $\varsigma$ and $\vars$ 
as in \S\ref{ss:lab}.
{That such trajectories extend to motions 
as in (\ref{pa:mot})
could  be shown by 
giving an explicit construction,
or proving
a suitable isotopy extension theorem, but here we simply lean on the (correct) intuition that there is one.
(See \cite[\S5.1]{TMM} for brief technicalities.
Complete details will appear in \cite{MMT}.)}

Let $\ofrog_i$ denote 
a morphism represented by a motion which swaps the positions of the $i$th and $i+1$th circle,
such that the circles remain parallel to the $xy$-plane, and
such that the $i$th circle `passes through' the disk parallel to the $xy$-plane bounded by the $i+1$th circle. 
The circle trajectory for 
such a motion is represented in Fig.~\ref{fig:79xx}.
(The {\em support} of a motion is the set of points that are not static throughout.
A key property of motions such as $\ofrog_i$ is that they include representatives where the support is localised in the region of $i$ and $i+1$.
Thus the local-view Figure conveys the motion -
it is recovered by suitably appending and prepending static loops.) 

\medskip 

Let $\oswap_i$ denote 
a morphism represented by a motion which swaps the positions of the $i$th and $i+1$th circle, such that the circles remain parallel to the $xy$-plane, and neither circle ever intersects the disk parallel to the $xy$-plane bounded by the other. Such a motion is represented in Fig.~\ref{fig:79}.

\medskip 

Let $\oflip_i$ denote
a morphism represented by a motion which rotates the $i$th circle by $\pi$ about an axis passing through the circle twice and its centre (thus the restriction of the endpoint of the motion to the $i$th circle is an orientation reversing homeomorphism). 

\newcommand{\agflip}{\mathfrak{t}}
\newcommand{\agswap}{\mathfrak{s}}
\newcommand{\agbraid}{\mathfrak{r}}

\medskip

The proof of (\ref{para:AutLiso})-(\ref{par:LBiso}) below is essentially contained in \cite{Dahm}, and reproduced in \cite{Goldsmith}, although we note that, as pointed out by \cite{Wattenberg}, the proof requires the additional assumption that each motion, and each homotopy, has compact support, rather than just each self homeomorphism.
With this additional assumption it can then be shown that the motion group of loops in $\R^3$ constructed in \cite{Dahm, Goldsmith} coincides with $\Mot{D^3}{\partial D^3}|_{l_*}$.

\mdef \label{para:AutLiso}
Consider the bouquet of $n$ loops homotopic to the complement of $l_n$ in $D^3$ (see e.g. \cite{DMM} and references therein).
Each motion induces an automorphism on the fundamental group of this bouquet. 
Clearly this fundamental group is generated by the $n$ simple loop paths. 
In fact it is free on these generators (for example by a universal cover argument). 

Let $\Aut (\mathrm{F}(x_1,\ldots,x_n))$ be the group of automorphisms of the free group generated by $\{x_1,\ldots,x_n\}$.
The given construction yields a homomorphism
\[
{\mathfrak D} \colon \LLex(l_n,l_n) \rightarrow \Aut (\mathrm{F}(x_1,\ldots,x_n)).
\]

\mdef
Let us consider the following
automorphisms in $\Aut (\mathrm{F}(x_1,\ldots,x_n))$.
\begin{align*}
    \agflip_i & \colon x_i\mapsto x_i^{-1}, \;\; x_k\mapsto x_k, \;\; k\neq i \\
    \agswap_i &\colon x_i\mapsto x_{i+1}, \;\; x_{i+1}\mapsto x_{i}, \;\; x_k\mapsto x_k, \;\; k\neq i, i+1 \\
    \agbraid_i & \colon x_i\mapsto x_{i+1}x_i x_{i+1}^{-1}, \;\; x_{i+1}\mapsto x_i , \;\; x_k\mapsto x_k, \;\; k\neq i
\end{align*}

While it is not easy to see that the construction in \ref{para:AutLiso} yields a well-defined group homomorphism from 
$ \LLex(l_n,l_n) $ it is straightforward to see that the images of our motions in (\ref{de:mots}) above are 
these 
automorphisms, in particular {$\mathfrak{D}(\oflip_i)=\agflip_i$, $\mathfrak{D}(\oswap_i)=\agswap_i$ and $\mathfrak{D}(\ofrog_i)=\agbraid_i$.
Moreover $\mathfrak{D}$ is an isomorphism 
onto the subgroup generated by $\agflip_i,\agswap_i,\agbraid_i $.
We thus have the following. 

\mdef \label{par:LBiso} The elements $ \oflip_i, \oswap_i, \ofrog_i$ 
generate the group $ \LLex(l_n,l_n)$.

\mdef \label{para:relationsL}
Observe that $\LL(n,n)$ as we have introduced it in the paper obeys
$\LL(n,n) \subset \LLex(n,n)$. 
It is
the 
{subgroup}
generated by $\rhoo_i$ and $\sig_i$, at each $l_n$.

Intuitively, we are restricting to classes which contain a motion which keeps all loops parallel to the $xz$-plane throughout the motion.

It follows from
from (\ref{pa:mot}) and (\ref{par:LBiso})
that there is an isomorphism (see also \cite{Dahm}) from $\LL(l_n,l_n)$ onto the 
{subgroup}
of $\Aut(F(x_1,\ldots, x_n))$ generated by the $\agswap_i, \agbraid_i$. 

The following identities are readily verified in $\Aut(F(x_1,\ldots, x_n))$:
\begin{align}\label{eq:loopbraidrels}
    \begin{cases}
    \agbraid_i\agbraid_j=\agbraid_j\agbraid_i & |i-j|>1 \\
\agbraid_i\agbraid_{i+1}\agbraid_i=\agbraid_{i+1}\agbraid_i\agbraid_{i+1} & i=1,\ldots, n-2\\
    \agswap_i\agswap_j=\agswap_i\agswap_j & \lvert i-j\rvert >1 \\
\agswap_i\agswap_{i+1}\agswap_i=\agswap_{i+1}\agswap_{i}\agswap_{i+1} & i=1,\ldots n-2 \\
    \agswap_i^2=1 & i=1,\ldots,n-1 \\
    \agswap_i\agbraid_j=\agbraid_j\agswap_i & \lvert i-j \rvert >1 \\
    \agbraid_i\agbraid_{i+1}\agswap_i = \agswap_{i+1}\agbraid_{i}\agbraid_{i+1} & i=1,\ldots n-2 \\
\agswap_{i}\agswap_{i+1}\agbraid_{1}=\agbraid_{i+1}\agswap_{i}\agswap_{i+1} &  i=1,\ldots n-2. \\
    \end{cases}
\end{align}

\mdef 
As proved by Savushkina \cite{Savushkina:96}, the 
{subgroup}
of $\Aut(F(x_1,\ldots, x_n))$ generated by the $\agswap_i, \agbraid_i$
has {\em presentation} as in (\ref{eq:loopbraidrels}) above.

\subsection{Generators and relations}\label{ss:Appg+r}

{Here we give a little background on two further technical aspects (again not strictly needed for this paper, but useful contextually). One is the nature of functors from finitely presented natural categories; and the other is on monoidal versus non-monoidal presentation.}

\newcommand{\D}{{\mathrm{D}}}
\newcommand{\CC}{{\mathrm{C}}}
\mdef \label{def:moncat} Let $\CC,\D$ be monoidal categories, with monoidal identity objects denoted by $\emptyset_{\CC}$ and $\emptyset_{\D}$, and monoidal composition maps denoted $\otimes_{\CC},\otimes_{\D}$ respectively. 
\\
A {\em strict monoidal functor} $\mathrm{F}\colon \CC\to \D$ 
is a functor 
such that
$\mathrm{F}(\emptyset_{\CC})=\emptyset_{\D}$,
$F(X)\otimes_{\D}F(Y)=F(X\otimes_{\CC}Y )$ for all pairs $X,Y$ of objects in $\CC$, and
$F(f)\otimes_{\D}F(g)=F(f\otimes_{\CC}g )$ for all pairs $f,g$ of morphisms in $\CC$.

\mdef {\em Strong} and {\em lax} monoidal functors weaken the equalities in Def.~\ref{def:moncat} to isomorphisms and morphisms respectively. Specifically a strong (resp. lax) monoidal functor consists of a functor $F\colon \CC\to \D$, together with an isomorphism (resp. morphism)
$F_\emptyset\colon\mathrm{F}(\emptyset_{\CC})\to \emptyset_{\D}$ and a natural isomorphism (resp. transformation) $\{F_2\colon F(X)\otimes_{\D}F(Y)\to F(X\otimes_{\CC}Y )\}_{X,Y}$,
with various coherence conditions. 

In the diagonal categories considered here, there are no morphisms between distinct objects so the conditions of strong and lax monoidal functors imply the equalities on objects as in Def.~\ref{def:moncat}.
Although this a priori leaves open the possibility of non-trivial automorphisms $F_\emptyset$ and $F_2$.

\medskip

We can
codify the strict monoidal functors from $\CC$ to $\D$
as follows.

\begin{lemma}
Let $\CC,\D$ be natural monoidal categories such that $\CC$ is finitely presented with 
non-empty set of generating morphisms $G(\CC)$ and a corresponding set of relations $R(\CC)$. 
{Of course any functor $F$ restricts to a pair consisting of the value of $F_0(1)$; and a function
$F_1 : G(\CC) \rightarrow \D$. 
In particular if $F_0(1)=M$ then 
$F_1(f:n\to m) \in \D(Mn,Mm)$.}
Let $M\in \N$, and
let $\mathfrak{F}\colon G({\CC})\to \D$ be a map sending 
each 
$f\colon n\to m$ in $G(\CC)$ to a morphism in $\D(nM,mM)$.
Then $\mathfrak{F}\colon G({\CC})\to \D$ extends to a  functor $\mathsf{F}\colon \CC\to \D$, which is unique, precisely if the images under $\mathfrak{F}$ of the relations $R(\CC)$ 
    are
    satisfied in $\D$.
\end{lemma}

\newcommand{\otimess}{\boxtimes}

\begin{proof}
{Since $G(\CC)$ generates, every morphism can be built (albeit non-uniquely) as a category-and-monoidal product of elements and identity morphisms.}
There is an extension of $\mathfrak{F}$ 
to a formal map $\mathfrak{F}'\colon \CC\to \D$, which on objects is the map $F_0\colon \N\to \N$ with $F_0(1)=M$, and on morphisms is defined by the relations $\mathfrak{F}'(1_X)=1_{\mathfrak{F}'(X)}$, $\mathfrak{F}'(f\circ g)=\mathfrak{F}'(f)\circ F(g)$, and 
$\mathfrak{F}'(f\otimes g)= \mathfrak{F}'(f) \otimes \mathfrak{F}'(g)$.
    This $\mathfrak{F}'$ is a well defined map $\mathsf{F}\colon \CC\to \D$ precisely when the relations $R(C)$ are satisfied, in which case it is a strict monoidal functor. By construction it is unique.
\end{proof}

\mdef Recall that the object monoid of 
the monoidal groupoid
$\LL'$ is the natural numbers, and that
both
generating morphisms are of the form $f\colon n\to n$, 
where $n$=2,
and 
therefore
all $\LL'(n,n')=\emptyset$ unless $n=n'$. 

\mdef 
After fixing $n$,
we denote the element $\sigma \otimes 1 \otimes \ldots \otimes 1  \in \LL'(n,n)$ of length $n-1$ by $\sigma_1$. We similarly label $\sigma_2,\ldots, \sigma_{n-1}$, and $s_1,\ldots, s_{n-1}$. Notice, using that $\otimes$ is a functor and thus preserves composition, that these elements generate {the group} $\LL'(n,n)$.

\begin{lemma}\label{lem:loopbraidiso}
There is an 
isomorphism of categories
(not using the monoidal structure)
\begin{align*}
\Theta \colon \LL' &\to \LL 
\end{align*}
which, at each $n$, sends $\sigma_i\mapsto 
\sig_i $ and $
s_i \mapsto \rhoo_i$. 
\soutx{(where using \ref{para:AutLiso} we label generating elements in $\LL$ by their preimages in $\Aut(F(x_1,\ldots x_n))$)}
\end{lemma}
\begin{proof}
First observe that since the $\sigma_i$, and $s_i$ generate $\LL'$, there is a unique map $\Theta$ satisfying the conditions of the Lemma and preserving composition. 
For this $\Theta$ to a be a well defined functor we must check that the relations in \eqref{eq:loopbraidrels} are satisfied in $\LL'$. 
    Surjectivity is satisfied since the $\sig_i$ and $\rhoo_i$ generate $\LL$.

We can prove both well-definedness and injectivity by noting that each relation in \eqref{eq:loopbraidrels} corresponds directly to a relation in $\LL'$ and vice versa.
The relations $\sigma_i \sigma_j =\sigma_j\sigma_i$ and $\rho_i\rho_j=\rho_j\rho_i$ when $\vert i-j \vert > 1$ correspond to relations coming from the fact that $\otimes$ is a functor.
    All other relations in \eqref{eq:loopbraidrels} correspond directly to the relations \eqref{eq:s21}, \eqref{eq:sss} and \eqref{eq:sigmass}.
\end{proof}

\section{Applying the recipe, and variations, to \texorpdfstring{$J_3^{\pm}$}{J3}} \label{ss:AppC}

\mdef \label{pa:N3ex}
Consider the set of solutions for $N=3$, thus 
associated to $J_3^{\pm}$ as in (\ref{exa:J})
by the recipe as in (\ref{de:recJ}).
Here firstly we have four solution sets associated to 
$\square^3$ (or eight ignoring $\Sigma_3$ symmetry):
\[
\hspace{-.25in} 
\aalph(F(s)) = (1,1,1,\smat 0&1\\1&0\stam,\smat 0&1\\1&0\stam,\smat 0&1\\1&0\stam),
\hspace{.51cm}
\aalph(F(\sigma))= (A_1, A_2, A_3,
\slosh{12},
\slosh{13},
\slosh{23}
\]
(we use up the gauge symmetry on $s$ and absorb
the effect of the gauge choice in free variables $C_{ij}$);
\[ \hspace{-.25in} 
\aalph(F(s)) = (1,1,-1,\smat 0&1\\1&0\stam,\smat 0&1\\1&0\stam,\smat 0&1\\1&0\stam),
\hspace{.41cm}
\aalph(F(\sigma))= (A_1, A_2, A_3,
\slosh{12},\slosh{13},\slosh{23}
\]
(here the $-1$ can be in each position); 
and similarly to those above with $s \leadsto -s$.
{The other two solutions from $J_3^{\pm}$ are given by two and three   $-1$'s in $\aalph(F(s))$.}
Note that overall this is every possible way of 
extending $\one^3$ in each $N=2$ subspace,
given that these solutions interlock at the vertices.
\\
Next we have four solution sets associated to 
$\square^1\;\two^1$:
\[
\aalph(F(s)) = (\pm1,1,1,\smat 0&1\\1&0\stam,\smat 0&1\\1&0\stam,\smat 1&0\\0&1\stam),
\hspace{.41cm}
\aalph(F(\sigma))= (A_1, A_2, A_2,
\slosh{12},\slosh{13},
\smat A_2&0\\0&A_2\stam)
\]
and $s \leadsto -s$ (with corresponding re-gauging if required).
\\
Four sets associated to $\square^1\;\oneone^1$:
$\;$ for example $(\square^1\;\oneone^1 ,)$ gives
\[ \hspace{-.25in} 
\aalph(F(s)) = (1,1,-1,\smat 0&1\\1&0\stam,\smat 0&1\\1&0\stam,\smat 0&1\\1&0\stam),
\hspace{.41cm}
\aalph(F(\sigma))= (A_1, A_2, A_3,
\slosh{12},\slosh{13},
\smat A_2+A_3&A_2\\-A_3&0\stam)
\]
while  $(\square^1 , \;\oneone^1 )$   gives
\[
\hspace{-.25in} 
\aalph(F(s)) = (1,-1,1,\smat 0&1\\1&0\stam,\smat 0&1\\1&0\stam,\smat 0&1\\1&0\stam),
\hspace{.41cm}
\aalph(F(\sigma))= (A_1, A_2, A_3,
\slosh{12},\slosh{13},
\smat A_2+A_3& -A_2\\ A_3&0\stam)
\]
Two sets associated to $\three^1$:
\[
F(s) = \pm 1_9, \hspace{1cm} F(\sigma) = A_1 1_9
\]
Two sets associated to $\twoone^1$:
\[
\aalph(F(s)) = (1,1,-1,\smat 1&0\\0&1\stam,\smat 0&1\\1&0\stam,\smat 0&1\\1&0\stam),
\hspace{.41cm}
\aalph(F(\sigma))= (A_1, A_1, A_3,
\smat A_{1}&0\\0&A_1\stam,
\smat A_1+A_3&A_1\\-A_{3}&0\stam,
\smat A_1+A_3&A_1\\-A_3&0\stam)
\]
and $s \leadsto -s$.
\\
Two sets associated to $\onetwo^1$:
\[
\aalph(F(s)) = (1,-1,-1,
\smat 0&1\\1&0\stam,
\smat 0&1\\1&0\stam,
\smat 1&0\\0&1\stam
),
\hspace{.41cm}
\aalph(F(\sigma))= (A_1, A_2, A_2,
\smat A_1+A_2&A_1\\-A_{2}&0\stam,
\smat A_1+A_2&A_1\\-A_2&0\stam,
\smat A_{2}&0\\0&A_2\stam
)
\]
and $s \leadsto -s$.

Note from \cite[Prop.5.1]{MR1X} 
that these exhaust the possibilities for $F(\sigma)$ and visit every possibility for $F(s)$ consistent with $N=2$ and the interlocking conditions.

\bibliographystyle{abbrv}
\bibliography{local.bib,HGT2.bib}

\begin{thebibliography}{10}

\bibitem{AS}
N.~Andruskiewitsch and H.-J. Schneider.
\newblock Pointed {H}opf algebras.
\newblock In {\em New directions in {H}opf algebras}, volume~43 of {\em Math.
  Sci. Res. Inst. Publ.}, pages 1--68. Cambridge Univ. Press, Cambridge, 2002.

\bibitem{Artin}
E.~Artin.
\newblock The theory of braids.
\newblock {\em American Scientist}, 38(1):112--119, 1950.

\bibitem{baez_schreiber}
J.~C. {Baez} and U.~{Schreiber}.
\newblock {Higher gauge theory.}
\newblock In {\em {Categories in algebra, geometry and mathematical physics.
  Conference and workshop in honor of Ross Street's 60th birthday, Sydney and
  Canberra, Australia, July 11--16/July 18--21, 2005}}, pages 7--30.
  Providence, RI: American Mathematical Society (AMS), 2007.

\bibitem{BCW}
J.~C. Baez, D.~K. Wise, and A.~S. Crans.
\newblock Exotic statistics for strings in 4d {BF} theory.
\newblock {\em Advances in Theoretical and Mathematical Physics},
  11(5):707--749, 2007.

\bibitem{Bardakov}
V.~G. Bardakov.
\newblock Extending representations of braid groups to the automorphism groups
  of free groups.
\newblock {\em J. Knot Theory Ramifications}, 14(8):1087--1098, 2005.

\bibitem{Baxter82}
R.~J. Baxter.
\newblock {\em {Exactly solved models in statistical mechanics}}.
\newblock 1982.

\bibitem{BH}
T.~E. Brendle and A.~Hatcher.
\newblock Configuration spaces of rings and wickets.
\newblock {\em Commentarii Mathematici Helvetici}, 88(1):131--162, 2013.

\bibitem{BMM19}
A.~Bullivant, J.~Faria~Martins, and P.~Martin.
\newblock Representations of the loop braid group and {A}haronov-{B}ohm like
  effects in discrete {$(3+1)$}-dimensional higher gauge theory.
\newblock {\em Adv. Theor. Math. Phys.}, 23(7):1685--1769, 2019.

\bibitem{Dahm}
D.~M. Dahm.
\newblock {\em A generalization of braid theory.}
\newblock PhD thesis, Princeton University, Mathematics, 1962.

\bibitem{Damiani}
C.~Damiani.
\newblock A journey through loop braid groups.
\newblock {\em Expositiones Mathematicae}, 35(3):252--285, 2017.

\bibitem{DMM}
C.~Damiani, J.~{Faria Martins}, and P.~P. Martin.
\newblock On a canonical lift of {A}rtin's representation to loop braid groups.
\newblock {\em Journal of Pure and Applied Algebra}, 225(12):106760, 2021.

\bibitem{DMR}
C.~Damiani, P.~P. Martin, and E.~C. Rowell.
\newblock Generalisations of {H}ecke algebras from loop braid groups.
\newblock {\em Pacific J. Math.}, to appear.

\bibitem{Elgueta}
J.~Elgueta.
\newblock Representation theory of 2-groups on {K}apranov and {V}oevodsky's
  2-vector spaces.
\newblock {\em Advances in Mathematics}, 213(1):53 -- 92, 2007.

\bibitem{martins_picken}
J.~{Faria Martins} and R.~{Picken}.
\newblock {Surface holonomy for non-Abelian 2-bundles via double groupoids.}
\newblock {\em {Adv. Math.}}, 226(4):3309--3366, 2011.

\bibitem{FRR}
R.~Fenn, R.~Rimányi, and C.~Rourke.
\newblock The braid-permutation group.
\newblock {\em Topology}, 36(1):123--135, 1997.

\bibitem{Goldsmith}
D.~L. Goldsmith.
\newblock The theory of motion groups.
\newblock {\em The Michigan Mathematical Journal}, 28(1):3--17, 1981.

\bibitem{Mach-HawkingEllis73}
S.~Hawking and G.~Ellis.
\newblock {\em The large scale structure of space-time}.
\newblock 1973.

\bibitem{JOYAL199155}
A.~Joyal and R.~Street.
\newblock The geometry of tensor calculus, i.
\newblock {\em Advances in Mathematics}, 88(1):55--112, 1991.

\bibitem{KMRW}
Z.~K\'{a}d\'{a}r, P.~Martin, E.~Rowell, and Z.~Wang.
\newblock Local representations of the loop braid group.
\newblock {\em Glasg. Math. J.}, 59(2):359--378, 2017.

\bibitem{KapranovVoevodsky}
M.~M. Kapranov and V.~A. Voevodsky.
\newblock {$2$}-categories and {Z}amolodchikov tetrahedra equations.
\newblock In {\em Algebraic groups and their generalizations: quantum and
  infinite-dimensional methods ({U}niversity {P}ark, {PA}, 1991)}, volume~56 of
  {\em Proc. Sympos. Pure Math.}, pages 177--259. Amer. Math. Soc., Providence,
  RI, 1994.

\bibitem{Kassel}
C.~Kassel.
\newblock {\em Quantum Groups}.
\newblock Springer--Verlag, 1995.

\bibitem{LKW}
T.~Lan, L.~Kong, and X.-G. Wen.
\newblock Classification of $\mathbf{(}3+1\mathbf{)}\mathrm{D}$ bosonic
  topological orders: The case when pointlike excitations are all bosons.
\newblock {\em Phys. Rev. X}, 8:021074, Jun 2018.

\bibitem{Lin}
X.-S. Lin.
\newblock {\em Nankai Tracts in Mathematics: Volume 12}, chapter Xiao-Song
  Lin's Unpublished Papers, pages 411--417.
\newblock World Scientific, 2008.

\bibitem{MaclaneBook}
S.~MacLane.
\newblock {\em Categories for the Working Mathematician}.
\newblock Graduate Texts in Mathematics. Springer New York, 1998.

\bibitem{Martin92}
P.~Martin.
\newblock On {Schur-Weyl} duality, {$A_n$} {Hecke} algebras and quantum
  {$SL(N)$} on {$\C_N^n$}.
\newblock {\em Int J Mod Phys A7 Suppl.1B}, page 645, 1992.

\bibitem{MR1X}
P.~Martin and E.~C. Rowell.
\newblock Classification of spin-chain braid representations.
\newblock {\em arXiv:2112.04533}, 2022.

\bibitem{N3aa0checknew.mw}
P.~Martin and E.~C. Rowell.
\newblock Maple worksheet n3aa0checknew.mw available., 2022.

\bibitem{Martin08a}
P.~P. Martin.
\newblock On diagram categories, representation theory and statistical
  mechanics.
\newblock {\em AMS Contemp Math}, 456:99--136, 2008.

\bibitem{TMM}
J.~F. Martins, P.~P. Martin, and F.~Torzewska.
\newblock Motion groupoids and mapping class groupoids.
\newblock {\em arXiv:2103.10377}, 2021.
\newblock arXiv preprint.

\bibitem{MMT}
J.~F. Martins, P.~P. Martin, and F.~Torzewska.
\newblock Monoidal motion groupoids: Examples, points and loops.
\newblock {\em in preparation}, 2023.
\newblock preprint in preparation.

\bibitem{Nayaketal}
C.~Nayak, S.~H. Simon, A.~Stern, M.~Freedman, and S.~Das~Sarma.
\newblock Non-abelian anyons and topological quantum computation.
\newblock {\em Rev. Mod. Phys.}, 80:1083--1159, Sep 2008.

\bibitem{NLYphys}
S.-Q. Ning, Z.-X. Liu, and P.~Ye.
\newblock Fractionalizing global symmetry on looplike topological excitations.
\newblock {\em Phys. Rev. B}, 105:205137, May 2022.

\bibitem{Soulieetal}
M.~Palmer and A.~Souli\'{e}.
\newblock The {B}urau representations of loop braid groups.
\newblock {\em C. R. Math. Acad. Sci. Paris}, 360:781--797, 2022.

\bibitem{porter_turaev}
T.~{Porter} and V.~{Turaev}.
\newblock {Formal homotopy quantum field theories. I: Formal maps and crossed
  $\mathcal{C}$-algebras.}
\newblock {\em {J. Homotopy Relat. Struct.}}, 3(1):113--159, 2008.

\bibitem{Rowell_Wang}
E.~C. {Rowell} and Z.~{Wang}.
\newblock {Mathematics of Topological Quantum Computing}.
\newblock {\em ArXiv e-prints}, May 2017.

\bibitem{Shubhankar18}
S.~Sahai.
\newblock Composition tableaux basis for {Schur} functors and the {Plücker}
  algebra, 1811.10687, 2018.

\bibitem{Savushkina:96}
A.~G. Savushkina.
\newblock On a group of conjugating automorphisms of a free group.
\newblock {\em Mat. Zametki}, 60(1):92--108, 159, 1996.

\bibitem{oeis}
N.~J.~A. Sloane and {The OEIS Foundation Inc.}
\newblock The on-line encyclopedia of integer sequences, 2020.

\bibitem{stanley1999enumerative}
R.~P. Stanley.
\newblock {\em Enumerative Combinatorics}, volume~2.
\newblock Cambridge University Press, 1999.

\bibitem{Vershinin}
V.~V. Vershinin.
\newblock On homology of virtual braids and {B}urau representation.
\newblock volume~10, pages 795--812. 2001.
\newblock Knots in Hellas '98, Vol. 3 (Delphi).

\bibitem{wang2015non}
J.~C. Wang and X.-G. Wen.
\newblock Non-abelian string and particle braiding in topological order:
  Modular sl (3, z) representation and (3+ 1)-dimensional twisted gauge theory.
\newblock {\em Physical Review B}, 91(3):035134, 2015.

\bibitem{Wattenberg}
F.~Wattenberg.
\newblock Differentiable motions of unknotted, unlinked circles in 3-space.
\newblock {\em Mathematica Scandinavica}, 30(1):107--135, 1972.

\end{thebibliography}

\end{document}